\documentclass[aos]{imsart}
\usepackage{xr}
\RequirePackage{amsthm,amsmath,amsfonts,amssymb}
\RequirePackage[numbers]{natbib}
\usepackage{bm}
\usepackage{algorithm}
\usepackage{xcolor}
\usepackage{stmaryrd}

\usepackage[flushleft]{threeparttable}
\RequirePackage[colorlinks,citecolor=blue,urlcolor=blue]{hyperref}
\RequirePackage{graphicx}

\usepackage{epstopdf}
\usepackage{float}
\usepackage{array}
\usepackage{booktabs}
\usepackage{multirow}
\usepackage{subcaption}

\usepackage[font=small,skip=0pt]{caption}
\DeclareMathOperator*{\argmin}{argmin}
\newcommand{\norm}[1]{\left\lVert#1\right\rVert}

\startlocaldefs

\theoremstyle{plain}
\newtheorem{thm}{Theorem}[section]

\newtheorem{lem}[thm]{Lemma}
\newtheorem{assu}[thm]{Assumption}
\newtheorem{prop}[thm]{Proposition}
\newtheorem{defn}[thm]{Definition}
\newtheorem{exam}[thm]{Example}
\theoremstyle{remark}
\newtheorem{rem}[thm]{Remark}

\numberwithin{equation}{section}

\endlocaldefs

\begin{document}

\begin{frontmatter}
\title{Auto-Regressive Approximations to Non-stationary Time Series, with Inference and Applications}
\runtitle{AR approximations to non-stationary time series}

\begin{aug}
\author[A]{\fnms{Xiucai} \snm{Ding}\ead[label=e1]{xcading@ucdavis.edu}} \and
\author[B]{\fnms{Zhou} \snm{Zhou}\ead[label=e2,mark]{zhou@utstat.utoronto.ca}}

\address[A]{Department of Statistics, University of California, Davis
\printead{e1}}

\address[B]{Department of Statistical Sciences,
University of Toronto
\printead{e2}}
\end{aug}

\begin{abstract}
Understanding the time-varying structure of complex temporal systems is one of the main challenges of modern time series analysis. In this paper, we show that every uniformly-positive-definite-in-covariance and sufficiently short-range dependent non-stationary and nonlinear time series can be well approximated globally by a white-noise-driven auto-regressive (AR) process of slowly diverging order. To our best knowledge, it is the first time such a structural approximation result is established for general classes of non-stationary time series. A high dimensional $\mathcal{L}^2$ test and an associated multiplier bootstrap procedure are proposed for the inference of the AR approximation coefficients. In particular, an adaptive stability test is proposed to check whether the AR approximation coefficients are time-varying, a frequently-encountered question for practitioners and researchers of time series. As an application, globally optimal short-term forecasting theory and methodology for a wide class of locally stationary time series are established via the method of sieves.
\end{abstract}



\end{frontmatter}

\section{Introduction}\label{sec_intro}

The Wiener-Kolmogorov prediction theory \cite{kolmogorov1941a, kolmogorov1941b,wiener1949} is a fundamental result in time series analysis which, among other findings, guarantees that a weakly stationary time series can be represented as a white-noise-driven auto-regressive (AR) process of infinite order under some mild conditions. The latter structural representation result has had profound influence in the development of the classic linear time series theory. Later, \cite{Baxter_1962, baxter1963} studied the truncation error of AR prediction of stationary processes when finite many past values, instead of the infinite history, were used in the prediction. Nowadays, as increasingly longer time series are being collected in the modern information age, it has become more appropriate to model many of those series as non-stationary processes whose data generating mechanisms evolve over time. Consequently, there has been an increasing demand for a systematic structural representation theory for such  processes. Nevertheless, it has been a difficult and open problem to establish linear structural representations for general classes of non-stationary time series. The main difficulty lies in the fact that the profound spectral domain techniques which were essential in the investigation of the AR($\infty$) representation for stationary sequences are difficult to apply to non-stationary processes where the spectral density function is either difficult to define or only defined locally in time.

The first main purpose of the paper is to establish a unified AR approximation theory for a wide class of non-stationary time series. Specifically, we shall establish that every short memory and uniformly-positive-definite-in-covariance (UPDC) non-stationary time series $\{x_{i,n}\}_{i=1}^n$ can be well approximated globally by a non-stationary white-noise-driven AR process of slowly diverging order; see Theorem \ref{thm_arrepresent} for a more precise statement.  Similar to the spirit of the Wiener-Kolmogorov prediction theory, the latter structural approximation result connects a wide range of fundamental problems in non-stationary time series analysis {such as optimal forecasting, dependence quantification, efficient estimation and adaptive bootstrap inference to those of AR processes and ordinary least squares (OLS) regression with diverging number of dependent predictors.} In fact, the very reason for us  to consider the AR approximation instead of a moving average approximation or representation (c.f. Wold decomposition \cite{wold1938study}) to non-stationary time series is due to its close ties with the OLS regression and hence the ease of practical implementation. Our proof of the structural approximation result resorts to modern operator spectral theory and classical approximation theory \cite{DMS} which control the decay rates of inverse of banded matrices. Consequently the decay speed of the best linear projection coefficients of the time series can be controlled via the Yule-Walker equations; see Theorem \ref{lem_phibound} for more details. 

The last two decades have witnessed the rapid development of locally stationary time series analysis in statistics.
Locally stationary time series refers to the subclass of non-stationary time series whose data generating mechanisms evolve smoothly or slowly over time. See  \cite{dahlhaus2012locally} for a comprehensive review. For locally stationary processes, we will show that the UPDC condition is equivalent to the uniform time-frequency positiveness of the local spectral density of $\{x_{i,n}\}_{i=1}^n$ (c.f. Proposition \ref{prop_pdc}) and the approximating AR process has smoothly time-varying coefficients (c.f. Theorem \ref{thm_locallynonzero}).

In practice, one may be interested in testing various hypotheses on the AR approximation such as whether some approximation coefficients are zero or whether the approximation coefficients are invariant with respect to time. The second main purpose of the paper is to propose a high-dimensional ${\cal L}^2$ test and an associated multiplier bootstrap procedure for the inference of the AR approximation coefficients of locally stationary time series. For the sake of brevity we concentrate  on the test of stability of the approximation coefficients with respect to time for locally stationary time series (c.f. (\ref{eq_testnonzero1})). It is easy  to see that  similar methodologies can be developed for other problems of statistical inference  such  as tests for parametric assumptions  on the approximation coefficients. Our test is shown to be adaptive to the strength of the time series dependence as well as the smoothness of the underlying data generating mechanism; see Propositions \ref{prop_normal} and \ref{prop_power} and Algorithm \ref{alg:boostrapping} for more details. The theoretical investigation of the test critically depends on a result on Gaussian approximations to quadratic forms of high-dimensional locally stationary time series developed in the current paper (c.f. Theorem \ref{thm_gaussian}). In particular, uniform Gaussian approximations over high-dimensional convex sets \cite{CF, FX}  as well as $m$-dependent approximations to quadratic forms of non-stationary time series are important techniques used in the proofs.

Interestingly, the test of stability for the AR approximation coefficients is asymptotically equivalent to testing correlation stationarity in the case of locally stationary time series; see Theorem \ref{lem_reducedtest} for more details. Here correlation stationarity means that the correlation structure of the time series does not change over time. As a result, our stability test can also be viewed as an adaptive test for correlation stationarity. In the statistics literature, there is a recent surge of interest in testing {\it covariance} stationarity of a time series using techniques from the spectral domain. See, for instance,  \cite{DPV,DR, GN, EP2010}. But it seems that the tests for correlation stationarity have not been carefully discussed in the literature. {Observe that the time-varying marginal variance has to be estimated and removed from the time series in order to apply the aforementioned tests to checking correlation stationarity \cite{dette2019change, zhao2015inference}. However, it is unknown whether the errors introduced in such estimation would influence the finite sample and asymptotic behaviour of the tests. Furthermore, estimating the marginal variance usually involves the difficult choice of a smoothing parameter. One major advantage of our test when used as a test of correlation stationarity is that it is totally free from the marginal variance as the latter quantity is absorbed into the errors of the AR approximation and hence is independent of the AR approximation coefficients. } 

Historically, the Wiener-Kolmogorov prediction theory was motivated by the optimal forecasting problem of stationary processes. Analogously, the AR approximation theory established in this paper is directly applicable to the problem of  optimal short-term linear forecasting of non-stationary time series. For locally stationary time series, thanks to the AR approximation theory, the optimal short-term forecasting problem boils down to that of efficiently estimating the smoothly-varying AR approximation coefficient functions at the right boundary. We propose a nonparametric sieve regression method to estimate the latter coefficient functions and the associated MSE of forecast. Contrary to most non-stationary time series forecasting methods in the literature where only data near the end of the sequence are used to estimate the parameters of the forecast, the nonparametric sieve regression is global in the sense that it utilizes all available time series observations to determine the optimal forecast coefficients and hence is expected to be more efficient. Furthermore, by controlling the number of basis functions used in the regression, we demonstrate that the sieve method is adaptive in the sense that the estimation accuracy achieves global minimax rate for nonparametric function estimation in the sense of \cite{stone1982} under some mild conditions; see Theorem \ref{thm_finalresult} for more details.  In the statistics literature, there have been some scattered works discussing non-stationary time series prediction from some different angles. See for instance  \cite{DP,dette2020prediction,FBS,KPF, RSP}, among others.  With the aid of the AR approximation, we are able to establish a unified globally-optimal short-term forecasting theory for a wide class of locally stationary time series asymptotically.

The rest of the paper is organized as follows. In Section \ref{sec:preliminary}, we introduce the AR approximation results for both general non-stationary time series and locally stationary time series.  In Section \ref{sec:test}, we test the stability of the AR approximation using  $\mathcal{L}^2$ statistics of the estimated AR coefficient functions for locally stationary time series. A multiplier bootstrap procedure is proposed and theoretically verified for practical implementation. In Section \ref{sec_application}, we provide one important application of our AR approximation theory in optimal forecasting of locally stationary time series.
%
In Section \ref{sec:simu}, we use extensive Monte Carlo simulations to verify the accuracy and power of our proposed methodologies. 
In  Section \ref{sec:realdata}, we conduct analysis on a financial real data set using our proposed methods.   Technical proofs  are deferred to the supplementary material \cite{DZ2supp}. 
\vspace{4pt}

\noindent {\bf Convention}. Throughout the paper, we will consistently use the following notations.  For a matrix $Y$ or vector $\bm{y},$ we use $Y^*$ and $\bm{y}^*$ to stand for their transposes. {For a scalar or vector $\bm{z}=(z_1, \cdots, z_p)^* \in \mathbb{R}^p,$ we use $| \bm{z} |=\sqrt{\sum_{j=1}^p z_j^2}$ to denote its $\ell_2$ (Euclidean) norm. For a random variable or vector $x$ and some constant $q\ge 1,$ denote by $\|x \|_q=(\mathbb{E}|x|^q)^{1/q}$ its $L^q$ norm.} For two sequences of real numbers $\{a_n\}$ and $\{b_n\},$  $a_n=O(b_n)$ means that $|a_n| \leq C|b_n|$ for some finite constant $C>0,$ and $a_n=o(b_n)$ means that $|a_n| \leq c_n |b_n|$ for some positive sequence $c_n \downarrow 0$ as $n \rightarrow \infty.$ For a sequence of  random variables $\{x_n\}$ and positive real values $\{a_n\},$ we use the notation $x_{n}=O_{\mathbb{P}}(a_n)$ to state that $x_n/a_n$ is stochastically bounded. 
Similarly, we use the notation $x_n=o_{\mathbb{P}}(a_n)$ to say that $x_n/a_n$ converges to 0  in probability. Moreover, we use the notation $x_n=O_{\ell^q}(a_n)$ to state that $x_n/a_n$ is bounded in $L^q$ norm; that is, $\|x_n/a_n\|_q\le C$ for some finite constant $C$. Similarly, we can define $ x_n=o_{\ell^q}(a_n)$.
We will always use $C$ as a genetic positive and finite constant independent of $n$ whose value may change from line to line.

\section{Auto-Regressive Approximations to Non-stationary Time Series } \label{sec:preliminary}

In this section, we establish a general AR approximation theory for a non-stationary time series $\{x_{i,n}\}$ under mild assumptions related to its covariance structure. 
Specifically, in Section \ref{sec:nonzeromeandiscussion}, we study general non-stationary time series. In Section \ref{sec:arappoximate}, we investigate the special case of locally stationary time series where the covariance structure is assumed to be smoothly time-varying. Before proceeding to our main results, we pause to introduce two mild assumptions.  

First, in order to avoid erratic behavior of the AR approximation,
the smallest eigenvalue of the time series {covariance} matrix should be bounded away from zero. For stationary time series, this is equivalent to the uniform positiveness of the spectral density function which is widely used in the literature. Further note that the latter assumption is mild and frequently used in the statistics literature of covariance and precision matrix estimation; see, for instance, \cite{cai2016, chen2013, Yuan2010} and the references therein. In this paper we shall call this \emph{uniformly-positive-definite-in-covariance (UPDC)} condition and formally summarize it as follows. 
\begin{assu}[UPDC]\label{assu_pdc} For all {sufficiently large} $n \in \mathbb{N}, $ we assume that there exists a universal constant $\kappa>0$ such that {
\begin{equation}\label{eq_defnkappa}
\lambda_n(\operatorname{Cov}(x_{1,n}, \cdots, x_{n,n})) \geq \kappa,
\end{equation} }
where $\lambda_n(\cdot)$ is the smallest eigenvalue of the given matrix and $\operatorname{Cov}(\cdot)$ is the covariance matrix of the given vector. 
\end{assu}

As discussed earlier, the UPDC is a mild assumption and is widely used in the literature. Moreover, for locally stationary time series, we will provide a necessary and sufficient condition from spectral domain (c.f. Proposition \ref{prop_pdc}) for practical checking. Second, we impose the following assumption to control the covariance decay speed of {$\{x_{i,n}\}$.}

\begin{assu}\label{assum_shortrange} {For all $n \in \mathbb{N},$ $1 \leq k \leq n$ and $-k+1 \leq r \leq n-k,$} we assume that there exists some constant $\tau>1$ such that {
\begin{equation}\label{eq_polynomialdecay}
\max_{k, n}\left| \operatorname{Cov}(x_{k,n}, x_{k+r,n}) \right| \leq C |r|^{-\tau},
\end{equation} }
where $C>0$ is some universal constant independent of $n.$ {In addition, we assume that $\sup_{i,n} \mathbb{E} |x_{i,n}|<\infty.$}
\end{assu}

Assumption \ref{assum_shortrange} states that the covariance structure of {$\{x_{i,n}\}$} decays polynomially fast and it can be easily satisfied for many non-stationary time series; see Example \ref{exam_onenon} for a demonstration. Note that $\tau>1$ amounts to a short range dependent requirement for {$\{x_{i,n}\}$} in the sense that {$|\sum_{l=1}^n\operatorname{Cov}(x_{k,n}, x_{l,n})|$} is bounded above by a fixed finite constant for all $k$ and $n$ while the latter sum may diverge if $\tau\le 1$.

\begin{rem}
In (\ref{eq_polynomialdecay}), we assume a polynomial decay rate. We can easily obtain analogous results to those established in this paper when the covariance decays exponentially fast, {
\begin{equation}\label{eq_exponentialdecay}
\max_{k, n}\left| \operatorname{Cov}(x_{k,n}, x_{k+r,n}) \right| \leq C a^{|r|}, \ 0<a<1. 
\end{equation}
}
For the sake of brevity, we focus on reporting our main results under the polynomial decay Assumption \ref{assum_shortrange}. 
From time to time, we will briefly mention the results under the exponential decay assumption (\ref{eq_exponentialdecay}) without providing extra details. 
\end{rem}

\subsection{AR approximation for general non-stationary time series} \label{sec:nonzeromeandiscussion}
{In this subsection, we establish an AR approximation theory for general non-stationary time series {$\{x_{i,n}\}$} satisfying Assumptions  \ref{assu_pdc} and \ref{assum_shortrange}.
{Denote by $b \equiv b(n)$ a generic value which specifies the order of the AR approximation. } In what follows, we investigate the accuracy of an AR($b$) approximation to {$\{x_{i,n}\}$} and provide the error rates using such an approximation. Observe that for theoretical and practical purposes $b$ is typically required to be much smaller than $n$ in order to achieve a parsimonious approximating model. For $i>b,$ the best linear prediction (in terms of the mean squared prediction error) {$\widehat{x}_{i,n} $}  of {$x_{i,n}$} 
 is denoted as   {
\begin{equation}\label{eq_hatxidefn}
\widehat{x}_{i,n}=\phi_{i0,n}+\sum_{j=1}^{i-1} \phi_{ij,n} x_{i-j,n}, \ i=b+1, \cdots, n,
\end{equation}}
where {$\phi_{ij,n}, 0 \leq j \leq i-1,$} are the prediction coefficients. 
Denote {$\epsilon_{i,n}:=x_{i,n}-\widehat{x}_{i,n}.$} It is well-known that {$\{\epsilon_{i,n}\}_{i=1}^n$} is a time-varying white noise process, i.e., {$\mathbb{E}\epsilon_{i,n}=0 \ , \operatorname{Cov} (\epsilon_{i,n}, \epsilon_{j,n})= \mathbf{1}(i=j) \sigma_{i,n}^2.$ }

Armed with the above notation, we write {
\begin{equation}\label{eq_arapproxiamtionnoncenter}
x_{i,n}=\phi_{i0,n}+\sum_{j=1}^{i-1} \phi_{ij,n} x_{i-j,n}+\epsilon_{i,n}, \ i=b+1,\cdots, n.
\end{equation} }
{ We point out that the coefficients {\{$\phi_{ij,n}$\}} are closely related to the Cholesky decomposition of the covariance and precision matrices of {$\{x_{i,n}\}$} \cite{DZ1, kang2020variable, pourahmadi1999joint}. For more details, we refer the readers to Section \ref{suppl_subsec_remark1} of our supplement \cite{DZ2supp}.}
To provide an AR approximation of order $b,$ where $b$ may be much smaller than $n,$ we need to examine the theoretical properties of the coefficients {$\phi_{ij,n}$.} We summarize the results in Theorem \ref{lem_phibound}. 

}


\begin{thm}\label{lem_phibound} 
Suppose Assumptions \ref{assu_pdc} and \ref{assum_shortrange} hold for {$\{x_{i,n}\}$}. 
For  $\tau$ in (\ref{eq_polynomialdecay}),  there exists some constant $C>0,$ when $n$ is sufficiently large,
we have that  {
\begin{equation}\label{eq_phibound1}
\max_i | \phi_{ij,n}| \leq C \left( \frac{\log j+1}{j} \right)^{\tau-1}, \ \ \text{for all} \ j \geq 1. 
\end{equation} }
Moreover, analogously to (\ref{eq_arapproxiamtionnoncenter}), denote by { $\{ \phi_{ij,n}^{b}\}$} the best linear forecast coefficients of { $x_{i,n}$} based on { $x_{i-1,n}, \cdots, x_{i-b,n},$} i.e.,  {
\begin{equation}\label{eq_arbapproximation}
x_{i,n}=\phi_{i0,n}^b+\sum_{j=1}^b \phi_{ij,n}^b x_{i-j,n}+\epsilon_{i,n}^b, \ i>b. 
\end{equation} }
{ Denote 
\begin{equation}\label{eq_newdefnphi}
\bm{\phi}_{i,n}=(\phi_{i1,n}, \cdots, \phi_{i,i-1,n})^*, \ \bm{\phi}_{i,n}^b=(\phi_{i1,n}^b, \cdots, \phi_{ib,n}^b, \bm{0})^* \in \mathbb{R}^{i-1}.
\end{equation}
}
Then we have that { for $\tau>2,$}
\begin{equation}\label{eq_phibound2}
\begin{aligned}{
\max_{i>b} | \bm{\phi}_{i,n}-\bm{\phi}_{i,n}^b | \leq C (\log b)^{\tau-1} b^{-(\tau-1)}, }  \\   
{\max_{i>b} |\phi_{i0,n}-\phi_{i0,n}^b| \leq C (\log b)^{\tau} b^{-(\tau-2)}. }
\end{aligned}
\end{equation} 
\end{thm}

On the one hand, Theorem \ref{lem_phibound} is general and only needs mild assumptions on the covariance structure of { $\{x_{i,n}\}.$} On the other hand, all error bounds in Theorem \ref{lem_phibound} are adaptive to the decay rate of the temporal dependence and the order of the AR approximation. Particularly, by \eqref{eq_phibound1}, we only need $\tau>1$ to ensure a polynomial decay of the coefficients {$\phi_{ij,n}$} as a function of $j$. Meanwhile, \eqref{eq_phibound2} establishes that the best linear forecast coefficients of {$x_{i,n}$} based on {$x_{i-1,n}, \cdots, x_{1,n}$} and { $x_{i-1,n}, \cdots, x_{i-b,n}$ } are close provided that $\tau$ and $b$ are sufficiently large. {We point out that unlike the results in \cite{Baxter_1962}, our result (\ref{eq_phibound2}) are stated in $\ell_2$ norm.}  

Based on Theorem \ref{lem_phibound}, we establish an AR approximation theory for the time series { $\{x_{i,n}\}$} in Theorem \ref{thm_arrepresent}. 
Denote the process { $\{x_{i,n}^*\}$} by {
\begin{equation}\label{defn_xistart}
x_{i,n}^*=
\begin{cases}
x_{i,n},  & i \leq b; \\
\phi_{i0,n}+\sum_{j=1}^b \phi_{ij,n} x_{i-j,n}^*+\epsilon_{i,n}, & i>b.  
\end{cases}
\end{equation}}
Since { $\{\epsilon_{i,n}\}$} is a time-varying white noise process, by construction, we have that { $\{x_{i,n}^*\}_{i\ge 1}$} is { a time-varying} AR($b$) process.

{
\begin{thm} \label{thm_arrepresent} Suppose the assumptions of Theorem \ref{lem_phibound} hold. Then we have that  for all $1 \leq i \leq n$
\begin{equation}\label{eq_nonzeromean11}
x_{i,n}=\phi_{i0,n}+\sum_{j=1}^{\min\{b,i-1\}} \phi_{ij,n} x_{i-j,n}+\epsilon_{i,n}+O_{\ell^2}\Big((\log b)^{\tau-1}  b^{-(\tau-1.5)}  \Big). 
\end{equation}
Furthermore, we have
\begin{equation}\label{eq_induction}
x_{i,n}-x_{i,n}^*=O_{\ell^2}\Big((\log b)^{\tau-1}  b^{-(\tau-1.5)} \Big ). 
\end{equation} 
\end{thm}
}

Recall from convention in the end of Section \ref{sec_intro} that the $O_{\ell^2}$ notation means bounded in the $L^2$ norm. Note that the AR approximation error diminishes as $b\rightarrow\infty$ if $\tau>1.5.$ Theorem \ref{thm_arrepresent} demonstrates that every sufficiently short-range dependent and UPDC non-stationary time series can be efficiently approximated by an AR process of slowly diverging order (c.f. (\ref{eq_induction})).  Furthermore, the approximation error is adaptive to the decay rate of the time series covariance as well as the approximation order $b$. 
 

{
\begin{rem}\label{rem_exponentialdecay}
Our results can be easily extended to the case when the temporal dependence is of  exponential decay, i.e., (\ref{eq_exponentialdecay}) holds true. 
In this case and under the UPDC condition, 
(\ref{eq_phibound1}) can be updated to ${ |\phi_{ij,n}|} \leq C \max\{n^{-2}, a^{j/2}\}, \ j>1, \ C>0 \ \text{is some constant}, $ and the magnitude of the error bounds in equations (\ref{eq_phibound2}), (\ref{eq_nonzeromean11}) and (\ref{eq_induction}) can all be changed to $\max\{ b^{3/2}/n^2, n^{-1}, a^{b/2}\}$.
\end{rem}
}

Before concluding this subsection, we provide an example of a general class of non-stationary time series using their physical representations \cite{WW,WZ1} and illustrate how to check the short range dependence assumption \ref{assum_shortrange} for this class of non-stationary processes. 

\begin{exam}\label{exam_onenon}
For a non-stationary time series { $x_{i,n}, 1 \leq i \leq n,$}, we assume that it  has the following form { 
\begin{equation}\label{eq_xi}
x_{i,n}=G_{i,n}(\mathcal{F}_i), \ i=1,2,\cdots,n,
\end{equation}
}
where $\mathcal{F}_i:=(\cdots, \eta_{i-1}, \eta_i)$ and $\eta_i, i \in \mathbb{Z}$ are i.i.d. random variables and the sequence of functions { $ G_{i,n}: \{1,2,\cdots, n\} \times \mathbb{R}^{\infty} \rightarrow \mathbb{R} $} are measurable such that for all $ 1 \leq
 i_0 \leq n,$ $G_{i_0,n}(\mathcal{F}_i) $ is a properly defined random variable. The above representation is very general since { under some mild regularity conditions,} any non-stationary time series $\{x_{i,n}\}_{i=1}^n$ can be represented in the form of \eqref{eq_xi} via the Rosenblatt transform \citep{rosenblatt1952}; { see \cite[Section 4]{wu2010new} and Section \ref{sec_suppl_remarks} of our revised supplement \cite{DZ2supp} for  more detailed discussion. } 
 
Under the representation (\ref{eq_xi}), temporal dependence can be quantified using physical dependence measures  \cite{WW, ZZ1, WZ2}.  Let $\{\eta_i^{\prime}\}$ be an i.i.d. copy of $\{\eta_i\}.$ 
For $j \geq 0,$ we define the physical dependence measure of { $\{x_{i,n}\}$} by 
\begin{equation}\label{eq_phygeneral}
\delta(j,q):=\sup_n\sup_{i} || G_{i,n}(\mathcal{F}_0)-G_{i,n}(\mathcal{F}_{0,j}) ||_q,
\end{equation}
where $\mathcal{F}_{0,j}:=(\mathcal{F}_{-j-1}, \eta_{-j}^{\prime},\eta_{-j+1},\cdots, \eta_0).$

From Lemma \ref{lem_coll} of \cite{DZ2supp}, Assumption \ref{assum_shortrange} will be satisfied if 
\begin{equation}\label{eq_physcialchecking}
\delta(j,2) \leq C j^{-\tau}. 
\end{equation}  

Hence (\ref{assum_shortrange}) can be directly checked by the physical dependence measures of a non-stationary time series.  For a more specific example, consider the following non-stationary linear processes { $x_{i,n}=\sum_{j=0}^{\infty} a_{ij,n} z_{i-j},$} where $\{z_i\}$ are i.i.d. random variables with finite variance and { $ a_{ij,n}$} are some constants. In this case, it is easy to see that (\ref{eq_physcialchecking}) is satisfied if $\sup_n { \sup_i|a_{ij,n}| } \leq C j^{-\tau}.$ { For more examples in the form of (\ref{eq_xi}), we refer the readers to \cite{WW} and \cite[Section 2.1]{DZ1}. } 
\end{exam}

{}

\subsection{AR approximation for locally stationary time series}\label{sec:arappoximate}

From the discussion of Section \ref{sec:nonzeromeandiscussion}, we have seen that every UPDC and sufficiently short-range dependent non-stationary time series can be well approximated by an AR process with diverging order (c.f. (\ref{defn_xistart}) and Theorem \ref{thm_arrepresent}). However, from an estimation viewpoint, since we assume that only one realization of the time series is observed, the Yule-Walker equations by which the AR coefficients (c.f. (\ref{eq_arapproxiamtionnoncenter}) or (\ref{eq_arbapproximation})) are governed are clearly underdetermined linear systems (i.e. there are more unknown parameters than the number of equations). Therefore, additional constraints/assumptions on the non-stationary temporal dynamics have to be imposed in order to estimate the AR approximation coefficients consistently. In this section, we shall consider an important subclass of non-stationary time series, the locally stationary time series \cite{dahlhaus2012locally,DPV,dette2020prediction,MR3097614,WZ1}.  This class of non-stationary time series is characterized by assuming that the underlying data generating mechanism evolves {\it smoothly} over time. 


 In this subsection, we will establish an AR approximation theory for locally stationary time series under certain smoothness assumptions of their covariance structure.  
We start with the following definition. { Analogous definitions have been used in \cite{KPF,RSP}. }
\begin{defn}[Locally stationary time series]\label{defn_locallystationary} A non-stationary time series { $\{x_{i,n}\}$} is a locally stationary time series (in covariance) if there exists a function $\gamma(t, k): [0,1] \times \mathbb{N} \rightarrow \mathbb{R}$ such that {
\begin{equation}\label{eq_covdefn}
\operatorname{Cov}(x_{i,n},x_{j,n})=\gamma(t_i, |i-j|)+O \left( \frac{|i-j|+1}{n} \right), \ t_i=\frac{i}{n}.
\end{equation} }
Moreover, we assume that $\gamma$ is Lipschitz continuous in $t$ and for any fixed $t \in[0,1],$ $\gamma(t,\cdot)$ is the autocovariance function of a stationary process.  
\end{defn}

Our  Definition \ref{defn_locallystationary} only imposes  a smoothness assumption on the covariance structure of { $\{x_{i,n}\}$.} In particular, \eqref{eq_covdefn} means that the covariance structure of { $\{x_{i,n}\}$} in any small time segment can be well approximated by that of a stationary process. Definition \ref{defn_locallystationary} covers many locally stationary time series models used in the literature \cite{dahlhaus2012locally,DPV,dette2020prediction,MR3097614,WZ1}. For more discussions, we refer the readers to Example \ref{eexample_linear} below. { Finally, we point out that for locally stationary time series with short-range dependence 
satisfying Assumption \ref{assum_shortrange}, we can update (\ref{eq_covdefn}) to {
\begin{equation}\label{eq_definedmodified} 
\operatorname{Cov}(x_{i,n},x_{j,n})=\gamma(t_i, |i-j|)+O \left( \min \left( \frac{|i-j|+1}{n}, \frac{1}{|i-j|^{\tau}} \right) \right), \ t_i=\frac{i}{n}.
\end{equation}
}
 }

%



Equipped with Definition \ref{defn_locallystationary}, we first provide a necessary and sufficient condition for the UPDC assumption in the case of locally stationary time series. {
For stationary time series, Herglotz's theorem asserts that UPDC holds if the spectral density function is bounded from below by a constant; see \cite[Section 4.3]{BD} for more details. Our next proposition extends such results to locally stationary time series with short-range dependence.   
\begin{prop}\label{prop_pdc} If $\{x_{i,n}\}$ is  locally stationary time series satisfying   Assumption \ref{assum_shortrange} and Definition \ref{defn_locallystationary}, and there exists some constant $\kappa>0$ such that $f(t,\omega) \geq \kappa$ for all $t$ and $\omega,$ where 
\begin{equation}\label{eq_spetraldensity}
f(t, \omega)=\sum_{j=-\infty}^{\infty} \gamma(t,j) e^{-\mathrm{i} j \omega}, \ \mathrm{i}=\sqrt{-1},
\end{equation} 
then $\{x_{i,n}\}$ satisfies Assumption \ref{assu_pdc}. Conversely, if $\{x_{i,n}\}$ satisfies   Assumptions \ref{assu_pdc} and \ref{assum_shortrange} and Definition \ref{defn_locallystationary}, then there exists some constant $\kappa>0,$ such that $f(t,\omega) \geq \kappa$ for all $t$ and $\omega.$
\end{prop}

{Note that $f(t,w)$ is the local spectral density function.} Proposition \ref{prop_pdc} implies that the verification of UPDC reduces to showing that the local spectral density function is uniformly bounded from below by a constant, which can be easily checked for many non-stationary processes. We refer the readers to Example \ref{eexample_linear} below for a demonstration. }

{
Next, we establish an AR approximation theory for locally stationary time series. As mentioned earlier, we need some smoothness assumptions such that the AR approximation coefficients in (\ref{eq_arbapproximation}) can be {estimated} consistently. Till the end, unless otherwise specified, we shall use the following  Assumption  \ref{assu_smoothtrend}, which states that the mean and covariance functions of $x_i$ are $d$-times continuously differentiable, for some positive integer $d$.}
\begin{assu}\label{assu_smoothtrend} For some given integer $d>0$, we assume that there exists a smooth function $\mu(\cdot) \in C^d([0,1]),$ where $C^d([0,1])$ is the function space on $[0,1]$ of continuous functions that have continuous first $d$ derivatives, such that { $\mathbb{E} x_{i,n}=\mu(t_i), \ t_i=\frac{i}{n}. $}  Moreover, we assume that $\gamma(t,j) \in C^d([0,1])$  for any $j \geq 0$.
\end{assu}

We now proceed to state the AR approximation theory for locally stationary time series (c.f. Theorem \ref{thm_locallynonzero}). We first prepare some notation.  
{ Denote $\bm{\phi}(t) \equiv \bm{\phi}_n(t):=(\phi_{1,n}(t), \cdots, \phi_{b,n}(t))^* \in \mathbb{R}^b$ such that }
\begin{equation}\label{eq_phismoothdefinition}
\bm{\phi}(t)=\Gamma(t)^{-1} \bm{\gamma}(t),
\end{equation}
where { $\Gamma(t) \equiv \Gamma_n(t) \in \mathbb{R}^{b \times b}$ and $\bm{\gamma}(t) \equiv \bm{\gamma}_n(t) \in \mathbb{R}^b$ are defined as}
$\Gamma_{ij}(t)=\gamma(t,|i-j|), \ \bm{\gamma}_i(t)=\gamma(t,i), \ i,j=1,2,\cdots,b. $
Here for any matrix $A$, $A_{ij}$ denotes its entry at the $i$th row and $j$th column. For a vector $V$, $V_i$ denotes its $i$th entry. As we will see in the proof of Theorem \ref{thm_locallynonzero},
$\Gamma(t)$ is always invertible under the UPDC assumption.
With the above notation, we further define { $\phi_{0,n}(t)$} as {
\begin{equation}\label{eq_phi0tdefn}
\phi_{0,n}(t)=\mu(t)-\sum_{j=1}^b \phi_{j,n}(t) \mu(t).
\end{equation}}
Analogous to (\ref{defn_xistart}), denote {
\begin{equation}\label{eq_xiss}
x_{i,n}^{**}=
\begin{cases}
x_{i,n},  & i \leq b; \\
\phi_{0,n}(\frac{i}{n})+\sum_{j=1}^b \phi_{j,n}(\frac{i}{n}) x_{i-j,n}^{**}+\epsilon_i, & i>b.  
\end{cases}
\end{equation}
}
{


\begin{thm} \label{thm_locallynonzero} Consider the locally stationary time series from Definition \ref{defn_locallystationary}. Suppose Assumptions  \ref{assu_pdc}, \ref{assum_shortrange} and \ref{assu_smoothtrend} hold true.  Then we have that $\phi_{j,n}(t) \in C^d([0,1]), \ 0 \leq j \leq b.$ { Recall (\ref{eq_newdefnphi}) and denote $\bm{\phi}(i/n)=(\phi_{1,n}^b(i/n), \cdots, \phi_{b,n}(i/n), \bm{0}) \in \mathbb{R}^{i-1}.$} Then there exists some constant $C>0,$ such that {
\begin{equation}\label{eq_controlarcoeff1}
 \max_{i>b}\left | \bm{\phi}_{i}-\bm{\phi}(\frac{i}{n}) \right| \leq C \left((\log b)^{\tau-1}  b^{-(\tau-1)}+\frac{b^2}{n} \right).
\end{equation} }
Moreover, we have that {
\begin{equation}\label{eq_controlarcoeff2}
\max_{i>b}\left| \phi_{i0,n}-\phi_{0,n}(\frac{i}{n}) \right| \leq C \left((\log b)^{\tau}  b^{-(\tau-2)}+\frac{b^{2.5}}{n} \right).
\end{equation} }
Finally, we have for $b+1 \leq i \leq n$ {
\begin{equation}\label{eq_choleskylocal}
x_{i,n}-\Big(\phi_{0,n}(\frac{i}{n})+\sum_{j=1}^b \phi_{j,n}(\frac{i}{n}) x_{i-j,n}+\epsilon_i\Big)=O_{\ell^2}\left((\log b)^{\tau}  b^{-(\tau-2)}+\frac{b^{2.5}}{n}\right),
\end{equation} }
and for all $1 \leq i \leq n$ {
\begin{equation}\label{cor_localboundmse}
x_{i,n}-x_{i,n}^{**}=O_{\ell^2}\left((\log b)^{\tau}  b^{-(\tau-2)}+\frac{b^{2.5}}{n}\right). 
\end{equation} }

\end{thm}


Theorem \ref{thm_locallynonzero} establishes that a locally stationary time series can be well approximated by an AR process of smoothly time-varying coefficients and a slowly diverging order under mild conditions. In particular, the AR coefficient functions { $\phi_{j,n}(\cdot)$} has the same degree of smoothness as the time-varying covariance functions $\gamma(\cdot,k), k\in\mathbb{Z}$.    
Observe that the smooth functions $\phi_{j,n}(\cdot)$  can be well approximated by models of small number of parameters using, for example, the theory of basis function expansion or local Taylor expansion. Therefore Theorem \ref{thm_locallynonzero} implies that the approximating AR model { $x_{i,n}^{**}$} can be consistently estimated using various popular nonparametric methods such as the local polynomial regression or the method of sieves provided that the underlying data generating mechanism is sufficiently smooth and the temporal dependence is sufficiently weak. { We point out some special form of (\ref{eq_controlarcoeff1}) has been established in Lemma 2.8 of \cite{DZ1} assuming mean zero time series, a physical representation form (\ref{defn_model}) and a specific $b$ satisfying $b \asymp n^{1/\tau}.$ In this regard, our result is a generalization of the existing result. }
}
\begin{rem}\label{rem:app}
 As we can see from Theorem \ref{thm_locallynonzero}, the approximation error for the locally stationary AR approximation comprises of two parts. { The first part is the truncation error, i.e., using an AR($b$) approximation instead of an AR($n$) approximation to { $\{x_{i,n}\}$.} This part of the error is represented by the first term on the right hand side of equations \eqref{eq_controlarcoeff1} to \eqref{cor_localboundmse}.  The second part is the error caused by using the smooth AR coefficients $\phi_{j,n}(\cdot)$ to approximate $\phi_{ij,n}$. This part of the error is represented by the second term on the right hand side of equations \eqref{eq_controlarcoeff1} to \eqref{cor_localboundmse}. In order to balance the aforementioned two types of errors, an elementary calculation shows that the choice of $b$ should satisfy that {
\begin{equation}\label{eq_boptimal}
\frac{b}{\log b} \asymp n^{\frac{1}{\tau+1}}.
\end{equation} }
We want to point out that \cite{DZ1} uses another special choice of $b$ in the setting of precision matrix estimation. Finally, we point out that in practice, $b$ can be chosen using a data-driven cross-validation procedure. The arguments  and details can be found in Section \ref{sec:choiceparameter} of our supplement \cite{DZ2supp}. 
}
\end{rem}

Before concluding this subsection, we provide an example to illustrate two frequently-used models of locally stationary time series in the literature and how the assumptions in this subsection can be verified for those models.

\begin{exam}\label{eexample_linear} 
We shall first consider the locally stationary time series model in \cite{WZ1,WZ2} using a physical representation. Specifically, the authors define locally stationary time series { $\{x_{i,n}\}$} as follows
\begin{equation}\label{defn_model}
{ x_{i,n}=G_n(\frac{i}{n}, \mathcal{F}_i),\ i=1,2,\cdots, n, }
\end{equation}
where 
${ G_n}:[0,1] \times \mathbb{R}^{\infty} \rightarrow \mathbb{R}$ is a measurable function such that { $\xi_{i,n}(t):=G_n(t, \mathcal{F}_i)$} is a properly defined random variable for all $t \in [0,1].$  In (\ref{defn_model}), by allowing the data generating mechanism { $G_n$} depending on the time
index $t$ in such a way that { $G_n(t,\mathcal{F}_i)$} changes smoothly with respect to $t$, one
has local stationarity in the sense that the subsequence { $\{x_{i,n}
, . . . , x_{i+j-1,n} \}$} is
approximately stationary if its length $j$ is sufficiently small compared to $n$. Analogous to (\ref{eq_phygeneral}), one can define the physical dependence measure for (\ref{defn_model}) as follows {
\begin{equation}\label{eq_phygeneral1}
\delta(j,q):=\sup_{t \in [0,1]} || G_n(t,\mathcal{F}_0)-G_n(t, \mathcal{F}_{0,j}) ||_q.
\end{equation} }
Moreover, the following assumptions are needed to ensure local stationarity. 
\begin{assu}\label{assum_local}
{ $G_n(\cdot, \cdot)$} defined in (\ref{defn_model}) satisfies the property of stochastic Lipschitz continuity, i.e.,  for some $q>2$ and $C>0,$ {
\begin{equation}\label{assum_lip}
\left| \left| G_n(t_1, \mathcal{F}_{i})-G_n(t_2,\mathcal{F}_i) \right|\right|_q \leq C|t_1-t_2|, 
\end{equation} }
where $t_1, t_2 \in [0,1].$ Furthermore, {
\begin{equation}\label{assum_moment}
\sup_{t \in [0,1]} \max_{1 \leq i \leq n} ||G_n(t,\mathcal{F}_i) ||_q<\infty.
\end{equation} }
\end{assu}

It can be shown that time series { $\{x_{i,n}\}$} with physical representation (\ref{defn_model}) and Assumption \ref{assum_local} satisfies Definition \ref{defn_locallystationary}. In particular, for each fixed $t \in [0,1],$ $\gamma(t,j)$ in Definition \ref{defn_locallystationary} can be found easily using the following {
  \begin{equation}\label{eq_defncov}
\gamma(t,j)=\operatorname{Cov}(G_n(t, \mathcal{F}_0), G_n(t, \mathcal{F}_{j})).
\end{equation} }
Note that the assumptions
(\ref{assum_lip}) and (\ref{assum_moment}) ensure that   $\gamma(t,j)$ is Lipschiz continuous in $t$. Moreover, for each fixed $t,$ $\gamma(t,\cdot)$ is the autocovariance function of { $\{G_n(t,\cdot)\},$} which is a stationary process.

The physical representation form (\ref{defn_model}) includes many commonly used locally stationary time series models. For example,  let $\{z_i\}$ be {zero-mean i.i.d. random variables (or a white noise) with variance $\sigma^2$.} We also assume $a_{j,n}(\cdot), j=0,1,\cdots$ be $C^d([0,1])$ functions such that{
\begin{equation}\label{ex_linear}
G_n(t, \mathcal{F}_i)=\sum_{k=0}^{\infty} a_{k,n}(t) z_{i-k}. 
\end{equation} } 
(\ref{ex_linear}) is a locally stationary linear process. It is easy to see that Assumptions \ref{eq_polynomialdecay}, \ref{assu_smoothtrend} and \ref{assum_local}  will be satisfied if 
$ \sup_{t \in [0,1]} |a_{j,n}(t)| \leq C j^{-\tau}, \ j \geq 1; \ \sum_{j=0}^{\infty} \sup_{t \in [0,1]} |a_{j,n}'(t)|<\infty, $ and 
\begin{equation}\label{eq_coeffdecay}
 \sup_{t \in [0,1]} |a_{j,n}^{(d)}(t)|\leq  C j^{-\tau}, \ j \geq 1.
\end{equation} 

{Furthermore, we note that the local spectral density function of \eqref{ex_linear} can be written as $f(t,w)= \sigma^2|\psi(t, e^{-\mathrm{i} j \omega})|^2,$ where $\psi(\cdot,\cdot)$ is defined such that $G_n(t, \mathcal{F}_i)=\psi(t, B) z_i$ with $B$ being the  backshift operator. By Proposition \ref{prop_pdc}, the UPDC is satisfied if $|\psi(t, e^{-\mathrm{i} j \omega})|^2 \geq \kappa$ for all $t$ and $\omega,$ where $\kappa>0$ is some universal constant.} For more examples of locally stationary time series in the form of (\ref{defn_model}) { especially nonlinear time series}, we refer the readers to \cite{WW}, \cite[Section 2.1]{DZ1} {, \cite[Example 2.2 and Proposition 4.4]{dahlhaus2019towards}, \cite[Proposition E.6]{karmakar2022simultaneous} and \cite{ding2021simultaneous, karmakar2022simultaneous,mayer2020functional}. Especially, the time-varying AR and ARCH models can be written into (\ref{ex_linear}) asymptotically (see Section \ref{secsec_arch} of our supplement \cite{DZ2supp} for more detail), and  Assumptions \ref{eq_polynomialdecay}, \ref{assu_smoothtrend} and \ref{assum_local} can be easily satisfied under mild assumptions. We refer the readers to the aforementioned references for more details. }

For a second example,  note that in \cite{dahlhaus2006statistical,DPV,MR3097614}, the locally stationary time series is defined as follows (see Definition 2.1 of \cite{MR3097614}). { $\{x_{i,n}\}$} is locally stationary time series if for each scaled time point $u \in [0,1],$ there exists a strictly stationary process $\{h_{i,n}(u)\}$ such that { 
\begin{equation}\label{eq_definition11111}
|x_{i,n}-h_{i,n}(u)| \leq \left(|t_i-u|+\frac{1}{n} \right)U_{i,n}(u),  \ \text{a.s},
\end{equation} 
where $U_{i,n}(u) \in L^q([0,1])$ } for some $q>0.$ By similar arguments as those of model \eqref{defn_model},  Definition \ref{defn_locallystationary} as well as assumptions of this subsection can be verified for \eqref{eq_definition11111}, { especially (\ref{eq_definition11111}) implies (\ref{eq_covdefn}). We refer the readers to Section \ref{suppl_subsec_remark3} of our supplement \cite{DZ2supp}. }  
\end{exam}

\section{A Test of Stability for AR Approximations}\label{sec:test}
In this section, we study a class of statistical inference problems for the AR approximation of locally stationary time series using a high dimensional $\mathcal{L}^2$ test. We point out here that, thanks to the AR approximation (\ref{eq_choleskylocal}), a wide class of hypotheses on the structure of $\phi_{j,n}(\cdot)$ can be performed using the aforementioned testing procedure. { Moreover,} for the sake of brevity, in this paper we concentrate  on the test of stability of the AR approximating coefficients with respect to time for locally stationary time series (c.f. (\ref{eq_testnonzero1})). The latter is an important problem as in practice one is usually interested in checking whether the time series can be well approximated by an AR process with time-invariant coefficients. { For notational convenience, till the end of the paper, unless otherwise specified, we omit the subscript $n$ and simply write $x_i \equiv x_{i,n}, \phi_{ij} \equiv \phi_{ij,n}$ and $\phi_j(\cdot) \equiv \phi_{j,n}(\cdot)$. From line to line, we will emphasize this dependence if some confusions can be caused.}

In order to theoretically investigate the test, time series dependence measures should be defined and controlled for the locally stationary time series $\{x_i\}$.  Throughout this section, we assume that the locally stationary time series admits the representation as in (\ref{defn_model}) equipped with physical dependence measures \eqref{eq_phygeneral1}. {In addition}, note that in Section \ref{sec:arappoximate}, all our AR approximation results are established under smoothness and fast decay assumptions of the covariance structure of $\{x_i\}$ without the need of any specific time series dependence measures. Therefore, we believe that the theoretical results of this section can be easily established using other measures of time series dependence such as the strong mixing conditions. For the sake of brevity, we shall concentrate on establishing results using the physical dependence measures in this paper. { Finally, in the current paper we focus on locally stationary time series with smoothly time-varying dynamics. Our results can be generalized to piecewise locally stationary time series by allowing for some abrupt changes in the underlying data generating mechanism as introduced in \cite{dette2019change,wu2019multiscale,ZZ1}. For more discussions, we refer the readers to Section \ref{sec_suppl_generalization2} of our supplement \cite{DZ2supp}. }

\subsection{Problem setup and test statistics}
In this subsection, we formally state the testing problems and propose our statistics based on nonparametric sieve estimators of $\{\phi_j(
\cdot)\}$ . 

Since $\phi_0(\cdot)$ is related to the trend of the time series and in many real applications the trend is removed via differencing or subtraction, we focus our discussion on the test the stability of $\phi_j(\cdot), \ j \geq 1.$  For the test of stability including the trend, we refer the readers to Remark \ref{rem_trendtest}. Formally, the null hypothesis we would like to test is
\begin{equation*}
\widetilde{\mathbf{H}}_0: \ \phi_j(\cdot) \ \text{is a  constant function} \ \text{on} \ [0,1], \ \text{for all} \ j \geq 1. 
\end{equation*}
Let $b_*$ diverges to infinity at the rate such that {
\begin{equation}\label{eq_choiceofb}
\frac{b_*}{\log b_*}\asymp n^{\frac{1}{\tau+1}}.
\end{equation} }
By Remark \ref{rem:app}, the AR approximation for locally stationary time series at order $b_*$ achieves the smallest error. According to (\ref{eq_phibound1}) and (\ref{eq_controlarcoeff1}), when $\tau$ is sufficiently large, we have that $\sup_t |\phi_j(t)|=o(n^{-1/2})$ for $j>b_*$. Therefore, from an inferential viewpoint, $\phi_j(\cdot)$ for $j>b_*$ can be effectively treated as zero. Together with the approximation (\ref{eq_choleskylocal}), it suffices for us to test
\begin{equation}\label{eq_testnonzero1}
\mathbf{H}_{0}: \ \phi_j(\cdot) \ \text{is a  constant function} \ \text{on} \ [0,1], \ j=1,2,\cdots,b_*.   
\end{equation} 
%
%
%
Before providing the test statistic for $\mathbf{H}_0$, we shall first investigate the interesting insight that $\mathbf{H}_0$ is asymptotically equivalent to testing whether $\{x_i\}_{i=1}^n$ is correlation stationary, i.e., there exists some function $\varrho$ such that
\begin{equation}\label{defn_ho}
\mathbf{H}^{\prime }_{0}: \ \operatorname{Corr} (x_i, x_j)=\varrho(|i-j|),
\end{equation}
where $\text{Corr}(x_i, x_j)$ stands for the correlation between $x_i$ and $x_j.$
We formalize the above statements in Theorem \ref{lem_reducedtest} below. 

 
\begin{thm}\label{lem_reducedtest}
Suppose Assumptions  \ref{assu_pdc}, \ref{assum_shortrange}, \ref{assu_smoothtrend}  and \ref{assum_local} hold true. { Recall (\ref{eq_phismoothdefinition}). }
 For  $j \leq b_*$, when $\mathbf{H}_0^{\prime}$ holds true, there exists some constants $\phi_j , \ 1 \leq j \leq b_*$  that  for $\bm{\phi}^{b_*}=(\phi_1, \cdots, \phi_{b_*}),$ we have that {
\begin{equation}\label{eq_errortermone} 
\left | \bm{\phi}(i/n)-\bm{\phi}^{b_*} \right|=O\left( \frac{b_*^2}{n} \right).
\end{equation} }
Consequently, let $\phi_i^{b_*}=(\phi_{i1}, \phi_{i2},\cdots, \phi_{ib_*}),$ together with (\ref{eq_controlarcoeff1}), we have {
\begin{equation*} 
\left | \bm{\phi}_i^{b_*}-\bm{\phi}^{b_*} \right |=O\left( (\log b_*)^{\tau-1} b_*^{-\tau+1}+\frac{b_*^2}{n} \right).
\end{equation*}
 }
{Moreover}, when $\mathbf{H}_0$ holds true that $\phi_j(\frac{i}{n})=\phi_j, j=1,2,\cdots, b_*$, there exists some smooth function $\varrho,$ such that 
{
\begin{equation}\label{eq:zzz}
\operatorname{Corr} (x_i, x_{i+j})=\varrho(j)+O\left((\log b_*)^{\tau-1}  b_*^{-(\tau-2)}+\frac{b_*^{2.5}}{n}\right).
\end{equation}  
} 
\end{thm}

Note that the right hand sides of (\ref{eq_errortermone}) and \eqref{eq:zzz} are of the order $o(n^{-1/2})$ if $\tau$ is sufficiently large. Hence Theorem \ref{lem_reducedtest} establishes the asymptotic equivalence between $\mathbf{H}^{\prime }_{0}$ and $\mathbf{H}_0$ for short range dependent locally stationary time series. { We point out that when the variance of the time series $\{x_i\}$ is constant, the error term $b_*^2/n$ will disappear from the right-hand side of (\ref{eq_errortermone}). We mention that there exist important time series models which are non-stationary in covariance but stationary in correlation. For instance, the following model has been widely used in the literature \cite{dahlhaus2000likelihood,DP,robinson1997large,wu2007inference},
$x_i=\mu\left( \frac{i}{n}\right)+\sigma\left( \frac{i}{n} \right) y_i, $
where $\{y_i\}$ is a stationary time series. { In fact, if a time series is correlation stationary with smooth mean and marginal variance, $x_i$ can always be approximated by the above form. In this sense, testing (\ref{defn_ho}) is actually a test for $x_i=\mu\left( \frac{i}{n}\right)+\sigma\left( \frac{i}{n} \right) y_i $ against a more general non-stationary structure in terms of $y_i.$}

 Finally, we point out that even though the primary goal of our test is to investigate the stability of the AR coefficients, it provides an efficient and easier way to study correlation stationarity. 
For more details, we refer the readers to Section \ref{suppl_subsec_remark2} of our supplement \cite{DZ2supp}. }

{
For the rest of this subsection, we shall propose a test statistic for  $\mathbf{H}_0$ in (\ref{eq_testnonzero1}). We start with the estimation of the coefficient functions $\phi_j(\cdot), j=0,1,2,\cdots,b,$ for a generic order $b$.  Since $\phi_j(t) \in C^d([0,1]),$ it is natural for us to approximate it via a finite and diverging term basis expansion (method of sieves \cite{CXH}).  Specifically, by \cite[Section 2.3]{CXH}, we have that
\begin{equation}\label{eq_phiform}
\phi_j(\frac{i}{n})=\sum_{k=1}^c a_{jk} \alpha_k(\frac{i}{n})+O(c^{-d}), \ 0 \leq j \leq b, \ i>b,
\end{equation}
where $\{\alpha_k(t)\}$ are some pre-chosen basis functions on $[0,1]$ and $c$ is the number of basis functions. For the ease of discussion, throughout this section, we assume that $c$ is of the following form
\begin{equation}\label{eq_defnc}
c=O(n^{\mathfrak{a}}), \ 0<\mathfrak{a}<1.  
\end{equation} 
Moreover, for the reader's convenience, in Section \ref{sec_basisfunctions} of \cite{DZ2supp}, we collect the commonly used basis functions.

{ In view of (\ref{eq_phiform}),  we need to estimate the $a_{jk}$'s in order to get an estimation for $\phi_j(t)$ as in \cite{DZ1}.} For $i>b,$ by (\ref{eq_choleskylocal}), write {
\begin{equation}\label{eq_choeq}
x_i=\sum_{j=0}^{b} \sum_{k=1}^c a_{jk} z_{kj}+\epsilon_i+O_{\ell^2}\left((\log b)^{\tau}  b^{-(\tau-2)}+\frac{b^{2.5}}{n}+bc^{-d}\right), 
\end{equation} }
where $z_{kj} \equiv z_{kj}(i/n):=\alpha_k(i/n) x_{i-j}$ for $j \geq 1$ and $z_{k0}=\alpha_k(i/n).$ {By (\ref{eq_choeq}), we can estimate all the $a_{jk}'s$ using only one ordinary least squares (OLS) regression with a diverging number of predictors. In particular, we write all $a_{jk}, \ j=0,1,2 \cdots, b, \ k=1,2,\cdots, c$ as a vector $\bm{\beta} \in \mathbb{R}^{(b+1)c},$  then the OLS estimator for $\bm{\beta}$ can be written as $\widehat{\bm{\beta}}=(Y^* Y)^{-1}Y^* \bm{x}$, where $\bm{x}=(x_{b+1}, \cdots, x_n)^* \in \mathbb{R}^{n-b}$ and $Y$ is the design matrix.  After estimating $a_{jk}'s,$ $\phi_j(i/n)$ is estimated using (\ref{eq_phiform}) as 
\begin{equation}\label{eq_phijest}
\widehat{\phi}_j(\frac{i}{n})=\widehat{\bm{\beta}}^*\mathbb{B}_j(\frac{i}{n}),
\end{equation} 
where $\mathbb{B}_j(i/n):=\mathbb{B}_{j,b}(i/n) \in \mathbb{R}^{(b+1)c}$  has $(b+1)$ blocks and the $j$-th block is $\mathbf{B}(\frac{i}{n})=(\alpha_1(i/n),\cdots,\alpha_c(i/n))^* \in \mathbb{R}^c, \ j=0,1,2,\cdots,b,$ and zeros otherwise. 
 }
} 

With the estimation (\ref{eq_phijest}), we proceed to provide the $\mathcal{L}^2$ test statistic. To test $\mathbf{H}_0,$ we use the following statistic in terms of (\ref{eq_phijest})
\begin{equation} \label{eq_defnntori}
T=\sum_{j=1}^{b_*} \int_0^1(\widehat{\phi}_j(t)-\overline{\widehat{\phi}}_j)^2 dt,  \  \overline{\widehat{\phi}}_j=\int_0^1 \widehat{\phi}_j(t) dt.
\end{equation}  
The heuristic behind $T$ is that $\mathbf{H}_{0}$ is equivalent to $\phi_j(t)=\overline{\phi}_j$ for $j=1,2,\cdots, b_*$, where $\overline{\phi}_j=\int_0^1 \phi_j(t) dt$.  Hence the ${\cal L}^2$ test statistic $T$ should be small under the null.

\begin{rem}\label{rem_trendtest}
We remark that in some cases practitioners and researchers may be interested  in testing whether all optimal forecast coefficient functions including the trend { $\phi_0(\cdot)$} do not change over time. That is equivalent to testing whether both the trend and the correlation structure of the time series stay constant over time. 
In this case, one will test 
\begin{equation}\label{eq_nonzeromeangeneral}
\mathbf{H}_{0,g}: \ \phi_j(\cdot) \ \text{is a constant function}  \ \text{on} \ [0,1], \ j=0,1,\cdots, b_*. 
\end{equation}
Similar to (\ref{eq_defnntori}), for the test of $\mathbf{H}_{0,g}$, we shall use 
\begin{equation}\label{eq_defntg} 
T_g=\sum_{j=0}^{b_*} \int_0^1(\widehat{\phi}_j(t)-\overline{\widehat{\phi}}_j)^2 dt,  \  \overline{\widehat{\phi}}_j=\int_0^1 \widehat{\phi}_j(t) dt.
\end{equation} 
\end{rem}

{
\begin{rem}
In this remark, we discuss how the statistics $T$ in (\ref{eq_defnntori}) and $T_g$ in (\ref{eq_defntg}) can be simplified under some specific basis functions as considered in \cite{jin2015new}. We focus our discussion on $T.$ For some specific bases, for example, the Walsh transform \cite{jin2015new}, the Fourier basis, the Legendre polynomial and the Haar wavelet basis as summarized in Section \ref{sec_basisfunctions} of our supplement \cite{DZ2supp}, the first basis function is always one.  That is to say
\begin{equation}\label{eq_basispropertyhaha}
\alpha_1(t) \equiv 1, \ \text{for all} \ t \in [0,1]. 
\end{equation}   
Due to the orthonormality, this leads to $\int \alpha_k(t) \mathrm{d} t=0, \ k>1. $ Consequently, using the definition $\widehat{\phi}_j(\cdot)$ in (\ref{eq_phijest}), i.e., $\widehat{\phi}_j(t)=\sum_{k=1}^c \widehat{a}_{jk} \alpha_k(t),$ we have that $\overline{\widehat{\phi}}_j=\int_0^1 \widehat{\phi}_j(t) \mathrm{d} t=\widehat{a}_{j1}.  $ Consequently, we have 
\begin{equation*}
T=\sum_{j=1}^{b_*} \int_0^1 \left( \sum_{k=2}^c \widehat{a}_{jk} \alpha_k(t) \right)^2 \mathrm{d} t, 
\end{equation*}
which does not include the global average anymore. Moreover, using the orthonormality of the basis functions, $T$ can be further simplified as
\begin{equation}\label{eq_simplified}
T=\sum_{j=1}^{b_*} \sum_{k=2}^c \widehat{a}_{jk}^2,
\end{equation}
which is defined only in terms of the OLS estimators. We point out that using the simplified the version (\ref{eq_simplified}) has slightly better finite sample performance in terms of both accuracy and power, especially when the sample size $n$ is smaller. For more details on numerical performance, we refer the readers to Section \ref{sec_differentTcomparison} of our supplement \cite{DZ2supp}.    
\end{rem}
}

\subsection{High dimensional Gaussian approximation and asymptotic normality} In this subsection, we prove the asymptotic normality of $T.$  The key ingredient is to establish Gaussian approximation results for quadratic forms of high dimensional locally stationary time series.

We first show that the study of the statistic $T$ reduces to the investigation of a weighted quadratic form of high dimensional locally stationary time series. We prepare some notation. Denote 
$\bar{B}=\int_0^1 \mathbf{B}(t)dt$ and $W=I-\bar{B} \bar{B}^*,$ { where we recall that $\mathbf{B}(t)=(\alpha_1(t), \cdots, \alpha_c(t))^* \in \mathbb{R}^c.$} Let $\mathbf{W}$ be a $(b_*+1)c \times (b_*+1)c$ dimensional diagonal block matrix with diagonal block $W$ and $\mathbf{I}_{b_*c}$ be a $(b_*+1)c \times (b_*+1)c$ dimensional diagonal matrix whose non-zero entries are ones and  in the lower $b_*c \times b_*c$ major part. {Recall $\bm{x}_i=(1, x_{i-1}, \cdots, x_{i-b_*})^*$} and set 
\begin{equation}\label{eq_dimensionratio}
p=(b_*+1)c.
\end{equation}
Recall $\epsilon_i$ in (\ref{eq_arapproxiamtionnoncenter}). We denote the sequence of $p$-dimensional vectors $\bm{z}_i$ by
\begin{equation}\label{eq_zikronecker}
\bm{z}_i=\bm{h}_i \otimes \mathbf{B}(\frac{i}{n}) \in \mathbb{R}^p, \ \bm{h}_i \equiv \bm{h}_{i,n}:=\bm{x}_{i,n} \epsilon_{i,n},
\end{equation}
where $\otimes$ is the Kronecker product. We point out that when $i>b_*,$ { from Lemma \ref{lem_lemporductlocally} of our supplement \cite{DZ2supp} we} see that $\bm{h}_i$ is a locally stationary time series { and can be expressed using a physical representation}. For notational convenience, we write  
\begin{equation}\label{eq_defnh}
\bm{h}_i=\mathbf{U}(\frac{i}{n}, \mathcal{F}_i), \ i>b_*. 
\end{equation}

 Recall (\ref{eq_defncov}). We also denote the $b_* \times b_*$ matrix $\Sigma^{b_*}(t)=(\Sigma^{b_*}_{ij}(t))$ such that $\Sigma^{b_*}_{ij}(t)=\gamma(t,|i-j|). $

\begin{lem}\label{lem_reducedquadratic}
Denote
$\mathbf{X}=\frac{1}{\sqrt{n}} \sum_{i=b_*+1}^n \bm{z}_i,$ 
and  the $p \times p$ matrix $\Gamma$ by
$
\Gamma=\overline{\Sigma}^{-1} \mathbf{I}_{b_*c} \mathbf{W} \overline{\Sigma}^{-1}, 
$
where 
\begin{equation}\label{eq_overlinesigma}
\overline{\Sigma}=\begin{pmatrix}
\mathbf{I}_c & \bm{0} \\
\bm{0} & \Sigma
\end{pmatrix}, \ \Sigma=\int_0^1 \Sigma^{b_*}(t) \otimes (\mathbf{B}(t) \mathbf{B}^*(t))dt.
\end{equation}
Suppose Assumptions \ref{assu_pdc},  \ref{assu_smoothtrend}, \ref{assum_local} and \ref{assu_basis} of \cite{DZ2supp} hold true. Moreover, we assume that the physical dependence measure $\delta(j,q), q>2,$ in (\ref{eq_phygeneral1}) satisfies 
\begin{equation}\label{eq_physcialbounbounbound}
\delta(j,q) \leq Cj^{-\tau}, \ j \geq 1, 
\end{equation}
for some constant $C>0$ and  $\tau>4.5+\varpi$, where $\varpi>0$ is some fixed small constant. Then for $c$ in the form of (\ref{eq_defnc}) and $b_*$  satisfying (\ref{eq_choiceofb}),
when $n$ is sufficiently large, we have that 
\begin{equation}\label{eq_tfinalexpress}
nT=\mathbf{X}^* \Gamma \mathbf{X}+o_{{\mathbb P}}(1). 
\end{equation}
\end{lem}

Based on Lemma \ref{lem_reducedquadratic}, for the purpose of statistical inference, it suffices to establish the distribution of $\mathbf{X}^* \Gamma \mathbf{X}$, which is  a high dimensional quadratic form of $\{\bm{z}_i\}$ since $p$ is divergent as $n\rightarrow\infty$.  To this end, we shall establish a Gaussian approximation result for the latter quadratic form.  Specifically, choose a sequence of centered Gaussian random vectors $\{\bm{v}_i\}_{i=b_*+1}^n$ which preserves the covariance structure of $\{\bm{h}_i\}_{i=b_*+1}^n$ and define 
$\bm{g}_i=\bm{v}_i \otimes \mathbf{B}(\frac{i}{n}).$
Denote $\mathbf{Y}=\frac{1}{\sqrt{n}} \sum_{i=b_*+1}^n \bm{g}_i^*. $
We shall establish a Gaussian approximation result by controlling the Kolmogorov distance 
\begin{equation}\label{eq_kolomogorv}
\mathcal{K}(\mathbf{X}, \mathbf{Y})=\sup_{x \in \mathbb{R}} \left| \mathbb{P} \Big( \mathbf{X}^*\Gamma \mathbf{X} \leq x \Big)-\mathbb{P} \Big( \mathbf{Y}^*\Gamma \mathbf{Y} \leq x \Big) \right|.
\end{equation}

\begin{thm} \label{thm_gaussian} Under the assumptions of Lemma \ref{lem_reducedquadratic}, there exist some constant $C>0$ such that 
\begin{equation*}
\mathcal{K}(\mathbf{X}, \mathbf{Y}) \leq C \Theta.
\end{equation*}
Here $\Theta \equiv \Theta(M_z, M, p, \xi_c, \tau, q, \delta)$ is defined as 
\begin{align}\label{eq_defntheta}
\Theta(M_z, M, p, \xi_c, \tau, q, \delta):= \log n\frac{\xi_c}{M_z} &+ p^{\frac{7}{4}} n^{-1/2}M_z^3 M^2+M^{\frac{-q\tau+1}{2q+1}}\xi_c^{(q+1)/(2q+1)} p^{\frac{q+1}{2q+1}} n^{\frac{\delta q}{2q+1}} \nonumber \\
& +p^{1/2}\xi_c^{1/2}\left(  (p\xi_c M^{-\tau+1}+p\xi_c^2 n M_z^{-(q-2)}) \right)^{1/2} +n^{-\delta},
\end{align}
where $M_z, M \rightarrow \infty$ when $n \rightarrow \infty$, $p$ is defined in (\ref{eq_dimensionratio}), $q>2$ is from (\ref{assum_lip}), $\delta$ is any fixed positive constant, and 
\begin{equation}\label{eq_xicdefinition}
\xi_c:=\sup_{1\leq i \leq c} \sup_{t \in [0,1]} \Big |\alpha_i(t) \Big |.
\end{equation} 
\end{thm}
\begin{rem}
Several remarks are in order. { First, $p=(b_*+1)c$ is the number of predictors in our least squares regression. $p$ is a tuning parameter which needs to be selected by the user; see Section \ref{sec:choiceparameter} of \cite{DZ2supp}. Second, the parameters $M, M_z$ and $\delta$ are parameters needed in the M-dependence approximation and the Stein's method implementation of our theoretical investigations. Those parameters are only needed in the theoretical investigations and are not needed in practical implementation of our methodology. In particular, $M_z$ is the truncation level for the locally stationary time series $\{{\bf z}_i\}$ in (3.16), $M$ is the level of the M-dependence approximation, and $\delta$ is defined in such a way that $1-O(n^{-\delta})$ is the probability that the truncated time series approximates the original time series sufficiently well. } Third, we remark that $\xi_c$ in (\ref{eq_xicdefinition}) can be well controlled for many commonly used basis functions. For instance, $\xi_c=O(1)$ for the Fourier basis and the normalized orthogonal polynomials and $\xi_c=O(\sqrt{c})$ for orthogonal wavelet; see Section \ref{sec_basisfunctions} of \cite{DZ2supp} for more details.  Fourth, { it is easy to see that the approximation rate in Theorem \ref{thm_gaussian} converges to 0 under mild conditions. For example, when $\tau$ is sufficiently large, i.e., the temporal relation decays fast polynomially, $q$ is sufficiently large, i.e., the time series have sufficiently large fiinte moments and $\xi_c=O(1),$ we only need $M_z \gg \log n, M \gg n^{\epsilon}$ for some fixed sufficiently small constant $\epsilon>0,$ $p \ll n^{2/7}$ and $\delta$ can be any fixed constant. }
\end{rem}

By Theorem \ref{thm_gaussian}, the asymptotic normality of $nT$ can be readily obtained as in Proposition \ref{prop_normal} below. Denote the long-run covariance matrix for $\{\bm{h}_i\}$ at time $t$ as
\begin{equation}\label{eq_longrunh}
\Omega(t)=\sum_{j=-\infty}^{\infty} \text{Cov} \Big(\mathbf{U}(t, \mathcal{F}_0), \mathbf{U}(t, \mathcal{F}_j) \Big),
\end{equation}
and the aggregated covariance matrix as 
$\Omega=\int_0^1 \Omega(t) \otimes \Big( \mathbf{B}(t) \mathbf{B}^*(t) \Big)\mathrm{d}t.$ 
$\Omega$ can be regarded as the integrated long-run covariance matrix of $\{\bm{z}_i\}.$ For $k \in \mathbb{N}$ and $\Gamma$ in (\ref{eq_tfinalexpress}), we define
\begin{equation} \label{eq_f}
f_k=\Big( \text{Tr}[ \Omega^{1/2} \Gamma \Omega^{1/2} ]^k \Big)^{1/k},
\end{equation}
where $\text{Tr}(\cdot)$ is the trace of the given matrix. 
\begin{prop} \label{prop_normal} Under the assumptions of Lemma \ref{lem_reducedquadratic},  assuming that $\Theta$ in (\ref{eq_defntheta}) satisfying  $\Theta=o(1)$, when $\mathbf{H}_0$ in (\ref{eq_testnonzero1}) holds true, we have  
\begin{equation*}
\frac{nT-f_1}{f_2} \Rightarrow \mathcal{N}(0,2).
\end{equation*}
\end{prop}
{
Next, we discuss the power of the test under a class of local alternatives. For a given $\alpha,$ set
\begin{equation*}
\mathbf{H}_a:   \sum_{j=1}^{\infty} \int_0^1 \Big( \phi_j(t)-\bar{\phi}_j\Big)^2 dt>C_{\alpha} \frac{\sqrt{b_*c}}{n},
\end{equation*}
where $\bar{\phi}_j=\int_0^1 \phi_j(t)dt$ and 
$C_{\alpha} \equiv C_{\alpha}(n) \rightarrow \infty $  as $n \rightarrow \infty.$ 

\begin{prop} \label{prop_power} Under the assumptions of Lemma \ref{lem_reducedquadratic}, assuming that $\Theta$ in (\ref{eq_defntheta}) satisfying  $\Theta=o(1)$, when $\mathbf{H}_a$ holds true, we have 
\begin{equation*}
\frac{nT-f_1-n\sum_{j=1}^{\infty} \int_0^1 \Big( \phi_j(t)-\bar{\phi}_j \Big)^2 dt}{f_2} \Rightarrow \mathcal{N}(0,2). 
\end{equation*}
Consequently, under  $\mathbf{H}_a,$ the power of our test will asymptotically be 1, i.e., 
\begin{equation*}
\mathbb{P} \Big( \left| \frac{nT-f_1}{f_2} \right| \geq \sqrt{2} \mathcal{Z}_{1-\alpha} \Big) \rightarrow 1, \ n \rightarrow \infty,
\end{equation*}
where $\mathcal{Z}_{1-\alpha}$ is the $(1-\alpha)$th quantile of the standard Gaussian distribution.
\end{prop}
}
Proposition \ref{prop_power} implies that our test $T$ can detect local alternatives when the ${\cal L}^2$ distance between $(\phi_1(t),\phi_2(t),\cdots)$ and $(\bar\phi_1,\bar\phi_2,\cdots)$ dominates $(b_*c)^{1/4}/\sqrt{n}\asymp p^{1/4}/\sqrt{n}$. Observe that Proposition \ref{prop_normal} requires that $p\ll n^{2/7}$. Therefore $p^{1/4}/\sqrt{n}$ converges to 0 faster than $n^{-3/7}$. { We mention that in the literature, for example \cite{paparoditis2016local}, the authors studied the local power properties of frequency domain based covariance stationarity tests. For some discussions and connections with our current time domain test of correlation stationarity, we refer the readers to Section \ref{suppl_subsec_remark4} of our supplement \cite{DZ2supp}.   }

\begin{rem}\label{rem_testtrend}
In this remark, we discuss how to deal with $T_g$ in (\ref{eq_defntg}). By a discussion similar to Lemma \ref{lem_reducedquadratic}, $T_g$ can also be expressed as a quadratic form $nT_g=\mathbf{X}^* \Gamma_g \mathbf{X}+o_{{\mathbb P}}(1), \ \Gamma_g=\overline{\Sigma}^{-1}  \mathbf{W} \overline{\Sigma}^{-1}, $ where we recall (\ref{eq_overlinesigma}). Consequently, the only difference lies in the deterministic weight matrix of the quadratic form. By Theorem \ref{thm_gaussian}, we can prove similar results to $n T_g$ as in Propositions \ref{prop_normal} and \ref{prop_power}.  We omit further details. 

\end{rem}

\subsection{Multiplier bootstrap procedure}\label{sec_bootstrapping}
In this subsection, we propose a practical procedure to implement the stability test based on a multiplier bootstrap procedure. 

On the one hand, it is difficult to directly use Proposition \ref{prop_normal} to carry out the stability test since the quantities $f_1$ and $f_2$ rely on $\Omega$ which is hard to estimate in general. 
On the other hand, the high-dimensional Gaussian quadratic form  $\mathbf{Y}^*\Gamma \mathbf{Y}$ converges at a slow rate.  To overcome these difficulties, we extend the strategy of \cite{ZZ1} and use a high-dimensional mulitplier bootstrap statistic to mimic the distributions of $nT.$ Note that (\ref{eq_tfinalexpress}) can be explicitly written as 
\begin{equation} \label{eq_defnnt}
nT= \Big( \frac{1}{\sqrt{n}} \sum_{i=b_*+1}^n \bm{z}_i^* \Big) \Gamma \Big( \frac{1}{\sqrt{n}} \sum_{i=b_*+1}^n \bm{z}_i \Big)+o_{{\mathbb P}}(1).
\end{equation}
Recall (\ref{eq_zikronecker}). For some positive integer $m,$ denote
\begin{equation}\label{eq_defnphi}
\Phi=\frac{1}{\sqrt{n-m-b_*+1}\sqrt{m}} \sum_{i=b_*+1}^{n-m} \Big[ \Big(\sum_{j=i}^{i+m} \bm{h}_j \Big) \otimes \Big( \mathbf{B}(\frac{i}{n}) \Big) \Big]  R_i,
\end{equation}
where $R_i, i=b_*+1, \cdots, n-m,$ are i.i.d. standard Gaussian random variables. $\Phi$ is an important statistic since the covariance of $\Phi$ is close to $\Omega$ conditional on the data; see (\ref{eq_lamdaomega}) of \cite{DZ2supp} for a more precise statement. 

Since $\{\bm{h}_i\}$ is based on $\{\epsilon_i\}$ which cannot be observed, we shall instead use the residuals
\begin{equation}\label{eq_residualdefinition}
\widehat{\epsilon}^{b_*}_i:=x_i-\widehat{\phi}_0(\frac{i}{n})-\sum_{j=1}^{b_*} \widehat{\phi}_j(\frac{i}{n})x_{i-j}.
\end{equation}
Denote $\{\widehat{\bm{h}}_i\}$ similarly as in (\ref{eq_zikronecker}) by replacing $\{\epsilon_i\}$ with $\{\widehat{\epsilon}^{b_*}_i\}.$ Accordingly, we denote $\widehat{\Phi}$ as in (\ref{eq_defnphi}) using $\{\widehat{\bm{h}}_i\}.$ With the above notations, we denote the bootstrap quadratic form  as
\begin{equation}\label{eq_defnmathcalt}
\widehat{\mathcal{T}}:=\widehat{\Phi}^* \widehat{\Gamma} \widehat{\Phi},
\end{equation}
where $\widehat{\Gamma}:= \widehat{\Sigma}^{-1} \mathbf{I}_{b_* c} \mathbf{W} \widehat{\Sigma}^{-1}$ with $\widehat{\Sigma}=\frac{1}{n} Y^* Y.$ Note that  $\widehat{\Gamma}$ is a consistent estimator of $\Gamma.$

In Theorem \ref{thm_bootstrapping} below, we prove that the conditional distribution of $\widehat{\mathcal{T}}$ can mimic that of $nT$ asymptotically. Denote
\begin{equation}\label{eq_zetacdefn}
\zeta_c:=\sup_t | \mathbf{B}(t) |. 
\end{equation}
\begin{thm}\label{thm_bootstrapping}
Suppose the assumptions of Lemma \ref{lem_reducedquadratic} hold  and  
\begin{equation}\label{eq_assumimply}
\sqrt{b_*}\zeta_c^2 c^{-1/2} \Big( \sqrt{\frac{m}{n}}+\frac{1}{m} \Big)=o(1).
\end{equation}
Furthermore, we assume that Assumption \ref{assum_local} holds with $q>4.$ When $\mathbf{H}_0$ holds true,  there exists some set $\mathcal{A}_n$ such that $\mathbb{P}(\mathcal{A}_n)=1-o(1)$ and under the event $\mathcal{A}_n,$ we have that conditional on the data $\{x_i\}_{i=b_*+1}^n,$   assuming that $\Theta$ in (\ref{eq_defntheta}) satisfying  $\Theta=o(1)$,
\begin{equation*}
\sup_{x \in \mathbb{R}} \left| \mathbb{P} \left( \frac{\widehat{\mathcal{T}}-f_1}{\sqrt{2}f_2} \leq x \right)-\mathbb{P} \left( \Psi \leq x \right) \right|=o(1),
\end{equation*} 
where $\Psi \sim \mathcal{N}(0,1)$ is a standard normal random variable. 
\end{thm}
\begin{rem}\label{remark_optimalm}
First, $\zeta_c$ can be well controlled by the commonly used sieve basis functions. For example, we have $\zeta_c=O(\sqrt{c})$ for the Fourier basis and orthogonal wavelets, and $\zeta_c=O(c)$ for the Legendre polynomials; see Section \ref{sec_basisfunctions} of \cite{DZ2supp} for more details.  Second, in the scenario where $\zeta_c=O(\sqrt{c}),$ \eqref{eq_assumimply} is equivalent to 
$\sqrt{p}   \Big( \sqrt{\frac{m}{n}}+\frac{1}{m} \Big)=o(1).$
Hence, in the optimal case when $m=O(n^{1/3}),$ we are allowed to choose $p \ll n^{2/3}$ if $\zeta_c=O(\sqrt{c})$. In this regime, Theorem \ref{thm_gaussian} still holds true.
Third,  for the detailed construction of $\mathcal{A}_n,$ we refer the reader to (\ref{eq_constructionset}) of \cite{DZ2supp}. Finally, a theoretical discussion of  the accuracy of the bootstrap can be found in Section \ref{sec_appendix_one} of \cite{DZ2supp} and the choices of the hyperparameters $m,c,b_*,$ are discussed in Section \ref{sec:choiceparameter} of \cite{DZ2supp}. { We point out that the performance of our proposed statistic and the multiplier bootstrap procedure are robust against these hyperparameters, especially $b_*$ and $m$. For more detailed discussion on this aspect, we refer the readers to Section \ref{sec_suppl_robustparameter} of our supplement \cite{DZ2supp} for more details. } 
\end{rem}



%

{
Based on Theorem \ref{thm_bootstrapping}, we can use Algorithm \ref{alg:boostrapping} for practical implementation to calculate the $p$-value of the stability test. 

\begin{algorithm}[!ht]
\caption{\bf Multiplier  Bootstrap}
\label{alg:boostrapping}

\normalsize
\begin{flushleft}
\noindent{\bf Inputs:}   tuning parameters $b_*$, $c$ and $m$ chosen by the data-drive procedure demonstrated in Section \ref{sec:choiceparameter} of \cite{DZ2supp}, time series $\{x_i\},$ and sieve basis functions.

\noindent{\bf Step one:} Compute $\widehat{\Sigma}^{-1}$ using $n(Y^*Y)^{-1}$ and the residuals $\{\widehat{\epsilon}^{b_*}_i\}_{i=b_*+1}^n$ according to (\ref{eq_residualdefinition}).

\noindent{\bf Step two:}  Generate B (say 1000) i.i.d. copies of $\{\widehat{\Phi}^{(k)}\}_{k=1}^B.$ Compute $\widehat{\mathcal{T}}_k, k=1,2,\cdots, B,$  correspondingly as in (\ref{eq_defnmathcalt}). 

\noindent{\bf Step three:} Let $\widehat{\mathcal{T}}_{(1)} \leq \widehat{\mathcal{T}}_{(2)} \leq \cdots \leq \widehat{\mathcal{T}}_{(B)}$ be the order statistics of $\widehat{\mathcal{T}}_k, k=1,2,\cdots, B.$ Reject $\mathbf{H}_0$ at the level $\alpha$ if $nT>\widehat{\mathcal{T}}_{(\lfloor B(1-\alpha)\rfloor)},$ where $\lfloor x \rfloor$ denotes the largest integer smaller or equal to $x.$  Let $B^*=\max\{r: \widehat{\mathcal{T}}_{r} \leq nT\}.$

\noindent{\bf Output:} $p$-value of the test can be computed as $1-\frac{B^*}{B}.$
\end{flushleft}
\end{algorithm}



%



%


\section{Applications to globally optimal forecasting}\label{sec_application}
In this section, independent of Section \ref{sec:test}, we discuss an application of our AR approximation theory in optimal global forecasting for locally stationary time series. We first introduce the notion of \emph{asymptotically optimal predictor}.
{
\begin{defn}\label{defn_asymptotic}
A linear predictor $\widetilde{z}$ of a random variable $z$ based on $x_1,\cdots, x_n,$ is called asymptotically optimal if
\begin{equation}\label{eq_asymptotic_weak}
\mathbb{E}(z-\widetilde z)^2\le \sigma_n^2+o(1),
\end{equation} 
and the predictor is called strongly asymptotically optimal if
\begin{equation}\label{eq_asymptotic}
\mathbb{E}(z-\widetilde z)^2\le \sigma_n^2+o(1/n),
\end{equation} 
where $\sigma_n^2$ is the mean squared error (MSE) of the best linear predictor of $z$ based on $x_1,\cdots, x_n$. 
\end{defn}
}

The rationale for the definition of strong asymptotic optimality is that, in practice, the MSE of forecast can only be estimated with a smallest possible error of $O(1/n)$ when the time series length is $n$. Specifically, it is well-known that the parametric rate for estimating the coefficients of a time series model is $O(n^{-1/2})$. When one uses the estimated coefficients to forecast the future, the corresponding influence on the MSE of forecast is $O(1/n)$ (at best). Therefore, if a linear predictor achieves an MSE of forecast within $o(1/n)$ range of the optimal one, it is practically indistinguishable from the optimal predictor asymptotically.  




In what follows, we shall focus on the discussion of one-step ahead prediction. The general { $h$-step ahead prediction for $h \leq h_0,$ where $h_0 \in \mathbb{N}$ is some fixed constant,} can be handled similarly { with some necessary modification; we refer the readers to Section \ref{sec_suppl_generalization1} of our supplement \cite{DZ2supp} for more details}.  In order to make the forecasting feasible, we assume that the smooth data generating mechanism extends to time $n+1$. That is, we assume that the time series $\{x_1,\cdots, x_{n+1}\}$ satisfies the locally stationary assumptions imposed in the paper.
Naturally, we propose the following estimate for  $\hat x_{n+1}$, the best linear predictor of $x_{n+1}$ based on its predecessors $x_n, x_{n-1},\cdots, x_{1}$,
\begin{equation}\label{eq_forecast}
\widehat{x}_{n+1}^b=\phi_0(1)+\sum_{j=1}^b \phi_j(1)x_{n+1-j}, \ n>b.
\end{equation}
Observe that \eqref{eq_forecast} is a truncated linear predictor where $x_n,x_{n-1},\cdots, x_{n-b+1}$ (instead of $x_n,\cdots,x_1$) is used to forecast $x_{n+1}$.  Note that here $b$ is a generic order which may be different from the order $b_*$ used in the test of stability. The next theorem shows that $\widehat{x}^b_{n+1}$ is an asymptotic optimal predictor satisfying  (\ref{eq_asymptotic_weak}) or (\ref{eq_asymptotic}) in Definition \ref{defn_asymptotic} under mild conditions. 




\begin{thm}\label{thm_prediction} Suppose Assumptions \ref{assu_pdc}, \ref{assum_shortrange}, \ref{assu_smoothtrend} and \ref{assum_local}  hold true. Then for sufficiently large $n$, {
\begin{equation}\label{eq_msemm}
\mathbb{E}(x_{n+1}-\widehat{x}^b_{n+1})^2 \leq \mathbb{E}(x_{n+1}-\widehat{x}_{n+1})^2+ O\left( \left( (\log b)^{\tau}  b^{-(\tau-2)}+\frac{b^{2.5}}{n} \right)^2 \right).
\end{equation} }
\end{thm}

{It is easy to see that the order $b$ which minimizes the right hand side of \eqref{eq_msemm} is of the same order as that in (\ref{eq_boptimal}). When $n$ is sufficiently large,  the corresponding error on the right-hand side of (\ref{eq_msemm}) equals
$O(n^{-2+\frac{5}{\tau-1}}).$
Hence Theorem \ref{thm_prediction} states that the estimator (\ref{eq_forecast}) is an asymptotic optimal one-step ahead forecast if $\tau>3.5$ and it is asymptotically strongly optimal if $\tau>6.$ 

Theorem \ref{thm_prediction} verifies the asymptotic global optimality of truncated linear predictors for locally stationary time series under mild conditions. For general stationary processes, \cite{Baxter_1962} and \cite{baxter1963}, among others, established profound theory on the decay rate of the AR approximation coefficients as well as the magnitude of the truncation error. As we mentioned in the Introduction, those results are derived using sophisticated spectral domain techniques which are difficult to extend to non-stationary processes. In this section, using the AR approximation theory established in this paper, we are able to establish a global optimal forecasting theory for the truncated linear predictors for a general class of locally stationary processes. 


}

In practice, one needs to estimate the optimal forecast coefficients $\phi_j(1)$, $j=0,\cdots, b$ as well as the MSE of the forecast $\mathbb{E}(x_{n+1}-\widehat{x}^b_{n+1})^2$. To investigate the estimation accuracy of those parameters, we need to impose certain dependence measures for the series. Therefore, for the rest of this subsection, we shall focus on the physical representation as well as dependence measures as in (\ref{defn_model}) and \eqref{eq_phygeneral1}. {To obtain an estimation for the predictor, 
in light of (\ref{eq_phijest}), based on (\ref{eq_forecast}), we shall estimate $\widehat{x}_{n+1}^b,$ or equivalently, forecast $x_{n+1}$ using
\begin{equation}\label{eq_forcaestequation}
\widehat{\mathsf{x}}_{n+1}^b=\widehat{\phi}_0(1)+\sum_{j=1}^b \widehat{\mathsf{\phi}}_j(1) x_{n+1-j}.
\end{equation}
} Next, we discuss the estimation of the MSE of the forecast.  Denote the series of estimated forecast error $\{\widehat{\epsilon}^b_i\}$ by
$\widehat{\epsilon}_i^b:=x_i-\widehat{\phi}_0(i/n)-\sum_{j=1}^b \widehat{\mathsf{\phi}}_j(i/n) x_{i-j}$
and the variance of $\{\epsilon_i\}$ by $\{\sigma_i^2\}.$  Recall the definition of $\epsilon_i$ in \eqref{eq_choeq}. { According to \cite[Lemma 3.11]{DZ1}, we find that there exists a smooth function $\varphi(\cdot) \in C^d([0,1])$ such that for some constant $C>0,$
\begin{equation}\label{eq_vvvffff}
\sup_{i>b} |\sigma_i^2-\varphi(\frac{i}{n})| \leq C\left((\log b)^{\tau-1}  b^{-(\tau-1.5)} +\frac{b^{2}}{n}\right). 
\end{equation}}

Therefore the estimation of $\sigma^2_i$ reduces to the estimation of the smooth function $\varphi$ as the estimation error of \eqref{eq_vvvffff} is sufficiently small for appropriately chosen $b$.
Similar to the estimation of the smooth AR approximation coefficients,  one can again use the method of sieves to estimate the smooth function $\varphi(\cdot)$. Specifically, similar to (\ref{eq_phiform}), write $\varphi(\frac{i}{n})=\sum_{k=1}^c \mathfrak{b}_k \alpha_k(\frac{i}{n})+O(c^{-d}). $ Furthermore,  by equation (3.14) of \cite{DZ1}, we write 
\begin{equation*}
(\widehat{\epsilon}_i^b)^2=\sum_{k=1}^c \mathfrak{b}_k \alpha_k(\frac{i}{n})+\nu_i+O_{\mathbb{P}}\Big(b(\zeta_c \frac{\log n}{\sqrt{n}}+c^{-d})\Big), \ i>b,
\end{equation*} 
where $\{\nu_i\}$ is a centered sequence of locally stationary time series satisfying Assumptions \ref{assu_pdc}, \ref{assum_shortrange}, \ref{assu_smoothtrend} and \ref{assum_local}. Consequently, we can use an OLS with $(\widehat{\epsilon}_i^b)^2$ being the response and $\alpha_k(\frac{i}{n})$, $k=1,\cdots, c,$ being the explanatory variables  to estimate $\mathfrak{b}_k$'s, which are denoted as $\widehat{\mathfrak{b}}_k, k=1,2,\cdots,c.$  
Finally, we estimate
\begin{equation}\label{eq_varianceestsieve} 
\widehat{\varphi}(i/n)=\sum_{k=1}^c \widehat{\mathfrak{b}}_k \alpha_k(i/n),
\end{equation}
and use $\widehat{\varphi}(1)$ to estimate the MSE of the forecast. We  now state the asymptotic behaviour of the MSE of (\ref{eq_forcaestequation}) in Theorem \ref{thm_finalresult} below. Recall (\ref{eq_zetacdefn}).

\begin{thm}\label{thm_finalresult}
Suppose Assumptions \ref{assu_pdc}, \ref{assum_shortrange}, \ref{assu_smoothtrend}, \ref{assum_local} and (1) and (2) of Assumption \ref{assu_basis} of \cite{DZ2supp}   hold true. 
 We have 
{
\begin{equation*}
\left |{\sigma_{n+1}^2}-\widehat{\varphi}(1)\right |=O_{\mathbb{P}}\Big(b(\zeta_c \sqrt{\frac{\log n}{n}}+c^{-d})+(\log b)^{\tau}  b^{-(\tau-2)}+\frac{b^{2.5}}{n}\Big). 
\end{equation*}
}
\end{thm} 
We point out that the error term on the right-hand side of the above equation vanishes asymptotically under mild conditions. For instance, assuming that $d$ is sufficiently large, $b$ slowly diverges as $n \rightarrow \infty$ (for example, as in (\ref{eq_boptimal})) and the temporal dependence decays fast enough (i.e. $\tau$ is some large constant), the leading error term in Theorem \ref{thm_finalresult} is $b \zeta_c \sqrt{\log n/n}.$ 
Moreover, if we assume exponential decay of temporal dependence as in Remark \ref{rem_exponentialdecay} and $\phi(\cdot)$ is infinitely differentiable,  then the error almost achieves the parametric $n^{-1/2}$ rate except a factor of logarithm.


{

}

\section{Simulation studies}\label{sec:simu}
 In this section, we perform Monte Carlo simulations to study the finite-sample accuracy and power of the multiplier bootstrap Algorithm \ref{alg:boostrapping}  for the test of stability of AR approximation coefficients and compare it with some existing methods on testing covariance stationarity in the literature. { We point out that in Section \ref{sec_suppl_forecastingexamination} of our supplement \cite{DZ2supp}, we also conduct some simulations to examine the numerical performance of our proposed forecast (\ref{eq_forcaestequation}). Due to space constraint, the simulation setups and results will be mainly reported in Section \ref{simu_intro} of our supplement \cite{DZ2supp}. } 

\subsection{Accuracy and power of the stability test}\label{sec_poer}
In this subsection, we study the performance of the proposed test (\ref{eq_testnonzero1}).  The simulation setups can be found in Section \ref{simu_intro} of our supplement \cite{DZ2supp}.
First, we study the finite sample accuracy of our test under the null hypothesis that  
\begin{equation}\label{eq_testingcases}
a_1(\frac{i}{n}) = a_2(\frac{i}{n}) \equiv 0.4.
\end{equation}
Observe that the simulated time series are not covariance stationary as the marginal variances change smoothly over time. We choose the values of $b_*,c$ and $m$ according to the methods described in Section \ref{sec:choiceparameter} of \cite{DZ2supp}. The simulation results can be found in Section \ref{simu_intro} of our supplement \cite{DZ2supp}. It can be seen from Table \ref{table_typei} of \cite{DZ2supp} that our bootstrap testing procedure is reasonably accurate for all three types of sieve basis functions even for a smaller sample size $n=256.$  


Second, we study the power of the tests and report the results in Table \ref{table_power} of \cite{DZ2supp} when the underlying time series is not correlation stationary, i.e., the AR approximation coefficients are time-varying.  Specifically, we use
\begin{equation}\label{eq_testingcasesalternative}
a_1(\frac{i}{n}) \equiv 0.4, \ a_2(\frac{i}{n})=0.2+ \delta \sin(2 \pi \frac{i}{n}),  \ \delta>0 \ \text{is some constant}, 
\end{equation}
for the model setups in Section \ref{simu_intro} of \cite{DZ2supp}. It can be seen that the simulated powers are reasonably good even for smaller values of $\delta$ and  sample sized, and the results will be improved when $\delta$ and the sample size increase. Additionally, the power performances of the three types of sieve basis functions are similar in general.

\subsection{Comparison with tests for covariance stationarity}\label{sec_simucovsta}
 
In this subsection, we compare our method with some existing works on the tests of covariance stationarity: the {non-smoothed} $\mathcal{L}^2$ distance method in \cite{DPV}, { the smoothed $\mathcal{L}^2$ distance method in \cite{paparoditis2009testing}, the Kolmogorov-Smirnov(KS) type test in \cite{preuss2013test},}
the discrete Fourier transform method in \cite{DR} and the Haar wavelet periodogram method in \cite{GN}. The first three methods are easy to implement; for the fourth method, we use the codes from the author's website (see \url{https://www.stat.tamu.edu/~suhasini/test_papers/DFT_covariance_lagl.R}); and for the last method, we employ the R package
 \texttt{locits}, which is contributed by the author.  The detailed setups of those models can be found in Section  \ref{simu_intro}  of our supplement \cite{DZ2supp}. 


For all the simulations, we report the type I error rates under the nominal levels $0.05$ and $0.1$ for all the seven models in Table \ref{table_compare_typeone} of \cite{DZ2supp}, where for models 1-5 we use the setup (\ref{eq_testingcases}). Our simulation results are based on 1,000 repetitions, { where NS-$\mathcal{L}^2$  refers to the non-smoothed $\mathcal{L}^2$ distance method, S-$\mathcal{L}^2$ refers to the smoothed $\mathcal{L}^2$ method, KS refers to the Kolmogorov-Smirnov type method,} DFT 1-3 refer to the approaches using the imagery  part, real part, both imagery and real parts of the discrete Fourier transform method, respectively, HWT is the Haar wavelet periodogram method and MB is our multiplier bootstrap method Algorithm \ref{alg:boostrapping} using orthogonal wavelets constructed by (\ref{eq_meyerorthogonal}) of \cite{DZ2supp} with Daubechies-9 wavelet. 

Since HWT needs the length to be a power of two, we set the length of time series to be 256 and 512. For the parameters of the first three tests, we use $M=8, N=32$ for $n=256,$ and $M=8, N=64$ for $n=512.$ For the DFT, we choose the lag to be $0$ as suggested by the authors in \cite{DR}. Since the mean of model 5 is non-zero, we test its first order difference for the methods mentioned above. Moreover, we report the power of the above tests under certain alternatives in Table \ref{table_comprare_power} of \cite{DZ2supp} for models $6^\#-7^\#$ and models 1-5 under the setup (\ref{eq_testingcasesalternative}).

{

}

We first discuss the results for models 6-7 since they are not only correlation stationary but also covariance stationary. 
It can be seen from Table \ref{table_compare_typeone} of \cite{DZ2supp} that all the methods including our MB achieve a reasonable level of accuracy for the linear model 6. However, for the nonlinear model 7, we conclude from Table \ref{table_compare_typeone} of \cite{DZ2supp} that the non-smoothed $\mathcal{L}^2$ method tends to be over-conservative due to the fact that the latter test is designed only for linear models driven by independent errors. { Moreover, the performance of the smoothed $\mathcal{L}^2$ method is better.} Regarding the power in Table \ref{table_comprare_power} of \cite{DZ2supp}, we shall first discuss the results for models $6^\#$ and $7^\#$ where the errors of the models are  i.i.d. For model $6^\#$, when the sample size and $\delta$ are smaller ($n=256$, $\delta=0.2$ or $0.35$), our MB method is significantly more powerful than the other methods. { When $n=256$ and $\delta$ increases, the first three tests starts to become similarly powerful. But when $\delta$ is smaller (i.e., the alternative is weaker), we find that the smoothed $\mathcal{L}^2$ test outperforms the non-smoothed $\mathcal{L}^2$. This is consistent with the observations as in Sections 4.2 and 5 of \cite{paparoditis2016local} and can be understood using the results of Section 3.2 therein. }Further, when both the sample size and $\delta$ increase, the HWT method becomes similarly powerful. Similar discussion holds for model $7^\#.$ Therefore, we conclude that, when the variances of the AR approximation errors stay constant, other methods in the literature are accurate for the purpose of testing for correlation stationarity (which is equivalent to covariance stationarity in this case). Furthermore, in this case the MB method is more powerful when the sample size is moderate and/or the departure from covariance stationarity is small for the alternative models experimented in our simulations.

Next, we study models 1-5. None of these models is covariance stationary. For the type I error rates, we use the setting (\ref{eq_testingcases}) where  all the models are correlation stationary. For the power, we use the setup (\ref{eq_testingcasesalternative}). We find that DFT-3 is accurate for models 1-4 but with low power across all the models. { Moreover, both smoothed and non-smoothed $\mathcal{L}^2$ tests seem to have a high power for models 3-5.} But this is at the cost of blown-up type I error rates. { Similar conclusions can be made for the KS type test}. This inaccuracy in Type-I error increases when the sample size becomes larger. { In addition, for models 1-2, the first three methods seem to be accurate. The KS and smoothed $\mathcal{L}^2$ methods have reasonably higher power, especially the KS method has a high power even $\delta$ is relatively small. In this regard, it seems that the conclusions of \cite{paparoditis2016local} still hold true beyond the time-varying linear Gaussian process.} For the HWT method, even though its power becomes larger when the sample size and $\delta$ increase, it also loses its accuracy. Finally, for all the models 1-5, our MB method  obtains both reasonably high accuracy and power. In summary, most of the existing tests for covariance stationarity are not suitable for the purpose of testing for correlation stationarity. Of course, the latter is expected as those tests are designed for testing covariance stationarity which is surely a different problem from correlation stationarity or stability of AR approximation. From our  simulation studies, our multiplier bootstrap method Algorithm \ref{alg:boostrapping} performs well for the latter purpose.

\section{An empirical illustration}\label{sec:realdata}
In this section, we illustrate the usefulness of our results by analyzing a financial data set. {In Section \ref{sec_suppl_anotherrealdata} of our supplement \cite{DZ2supp}, we also apply our method to study a global temperature data set. }

We study the stock return data of the Nigerian Breweries (NB) Plc. This stock is traded in Nigerian Stock Exchange (NSE). Regarding on market
returns, the brewery industry in Nigerian has done pretty well in
outperforming Brazil, Russia, India, and China (BRIC)
and emerging markets by a wide margin over the past
ten years. Nigerian Breweries Plc is the  largest brewing company in Nigeria, which  mainly serves the Nigerian market and also exports to other parts of West Africa. 
The data can be found on the website of morningstar  
(see  \url{http://performance.morningstar.com/stock/performance-return.action?p=price_history_page&t=NIBR&region=nga&culture=en-US}). We are interested in understanding the volatility of the NB stock. We shall study the absolute value of the  daily log-return of the stock for the latter purpose.

{
We perform our analysis on the time period 2008-2014 (Figure \ref{0814_timeseries} of our supplement \cite{DZ2supp}). This time series  contains  the data of the 2008 global financial crisis and its post period. As said in the report from the Heritage Foundation \cite{Report11}, "the economy is experiencing the slowest recovery in 70 
years" and even till 2014, the economy does not fully recover.    
}


{Then we apply the methodologies described in Sections \ref{sec:test} for the absolute values of log-return time series. { We point out that we focus on the log-return in our study without imposing any specific model assumption. We refer the readers to \cite{karmakar2022simultaneous} for a nonparametric model-based approach to volatility inference where a locally stationary GARCH model is imposed.} It is clear that we need to fit a mean curve for this model. Then we test the stability of the AR approximation as described in Section \ref{sec:test} using Algorithm \ref{alg:boostrapping}. For the sieve basis functions, we use the orthogonal wavelets constructed by (\ref{eq_meyerorthogonal}) of \cite{DZ2supp} with Daubechies-9 wavelet. We choose the parameters $b_*,c$ and $m$ based on the discussion of Section \ref{sec:choiceparameter} of \cite{DZ2supp} which yields $b_*=7,$ $J_n=5$ (i.e., $c=32$) and $m=18$. We apply the bootstrap procedure described in Algorithm \ref{alg:boostrapping} and find that the $p$-value is $0.0825$. We hence conclude that the prediction is likely to be unstable during this time period. }

Next, we use the time series 2008-2014 as a training dataset to study the (rolling) forecast performance over the first month of 2015 using (\ref{eq_forcaestequation}). We employ the data-driven approach from Section \ref{sec:choiceparameter} of \cite{DZ2supp} to choose $b=5$ and $J_n=3.$
The averaged MSE is  $0.189$.  We point out that this leads to a $20.9 \%$ improvement compared to simply fitting a stationary ARMA model using all the data from 2008 to 2014 where the MSE is 0.239, and leads to a $24.7 \%$ improvement compared to the benchmark of simply using $\hat{x}_{n+1}=x_n$ where the MSE is 0.251. { In Table \ref{table:fin} of our supplement \cite{DZ2supp}, we also compare our proposed forecasting (\ref{eq_forcaestequation}) with other methods in the literature.}

Finally, we study the absolute value of the stock return from 2012 to 2014. We apply our bootstrap procedure Algorithm \ref{alg:boostrapping} to test correlation stationarity of the sub-series. We select $b_*=6,$ $J_n=4$ (i.e., $c=16$) and $m=12$ for this sub-series and find that the $p$-value is $0.599$. We hence conclude that the prediction is stable during this time period. Therefore, we fit a  stationary ARMA model to this sub-series and do the prediction. This yields an MSE of 0.192 which is close to 0.189, the MSE when we use the whole non-stationary time series and the methodology proposed in Section \ref{sec_application}. The result from this sub-series shows an interesting trade-off between forecasting using a shorter and correlation-stationary time series and a longer but non-stationary series { as described by the Rules of Thumb in \cite{KPF}. The forecast model of the shorter stationary period via segment can be estimated at a faster rate but at the expense of a smaller sample size.}  The opposite happens to the longer non-stationary period. Note that 2012-2014 is nearly half as long as 2008-2014 and hence the length of the shorter stationary period is substantial compared to that of the long period. In this case we see that the forecasting accuracy using the shorter period is comparable to that of the longer period. { In many applications where the data generating mechanism is constantly changing, the stable period is typically very short and in this case the methodology proposed in Section \ref{sec_application} is expected to give better forecasting results under the assumption that the time series is locally stationary.} Finally, we emphasize that the correlation stationarity test proposed in this paper is an important tool to determine a period of prediction stability.
}

\vspace{4pt}

{
\section*{Acknowledgments} The authors would like to thank the editor, associated editor and three  anonymous reviewers for their valuable and insightful comments which have improved the paper significantly.  }

\vspace*{4pt}

\center{\bf  \uppercase{Supplement to "Auto-Regressive Approximations to Non-stationary Time Series, with Inference and Applications}}

\appendix

\renewcommand{\thesection}{\Alph{section}}
\counterwithin{table}{section}
\counterwithin{figure}{section}

\section{Some conventions}
Throughout the supplement, we consistently use the conventions listed in the end of Section \ref{sec_intro} of the main manuscript. { Moreover, for notational convenience and simplicity, till the end of the supplement, unless otherwise specified, we omit the subscript $n$ and simply write $x_i \equiv x_{i,n}, \phi_{ij} \equiv \phi_{ij,n}$ and $\phi_j(\cdot) \equiv \phi_{j,n}(\cdot)$. From line to line, we will emphasize this dependence if some confusions can be caused.} We also recall that for any random variable $x \in \mathbb{R},$ we simply write $\|x\|_q$ to denote the $L^q$ norm of $x$. For any deterministic vector $\bm{y},$ we use $|\bm{y}|$ to denote its Euclidean norm. For two sequences of positive real values $\{a_n\}$ and $\{b_n\},$ we write $a_n \asymp b_n$ if $a_n=O(b_n)$ and $b_n=O(a_n).$


Moreover, we use $\| A\|$ to denote the operator norm if $A$ is a matrix. Consequently, if $\{A_n\}$ is a sequence of deterministic matrices and $\{a_n\}$ is a sequence of positive real values, the notation $\|A_n \|=O(a_n)$ means that there exists some constant $C>0$ so that $\|A_n \| \leq C a_n.$
Moreover, if $\{A_n\}$ is a sequence of random matrices, the notation $\|A_n \|=O_{\mathbb{P}}(a_n)$ means that the operator norm of $A_n$ is stochastically bounded by $a_n.$ { Finally, for two integers $l_1<l_2,$ we define $\llbracket l_1, l_2 \rrbracket:=[l_1, l_2] \cap \mathbb{Z}$.}


\section{Results on simulations and real data analysis}


\subsection{Simulation setup and results}\label{simu_intro}
In this subsection, we introduce our simulation setups and summarize the main simulation results. 

We first consider four different types of non-stationary time series models: two linear time series models, a two-regime model, a Markov switching model and a bilinear model.  
\begin{enumerate}
\item Linear AR model: Consider the following time-varying AR(2) model
\begin{equation*}
x_i=\sum_{j=1}^2 a_j(\frac{i}{n}) x_{i-j}+\epsilon_i, \ \epsilon_i=\Big(0.4+0.4\Big|\sin(2\pi\frac{i}{n})\Big|) \eta_i,  
\end{equation*} 
where $\eta_i, i=1,2,\cdots, n,$ are i.i.d. random variables whose distributions will be specified when we finish introducing the models. It is elementary to see that when $a_j(\frac{i}{n}), j=1,2,$ are constants, the prediction is stable. 
{
\item Linear MA model: Consider the following time-varying MA(2) model
\begin{equation*}
x_i=\sum_{j=1}^2 a_j(\frac{i}{n}) \epsilon_{i-j}+\epsilon_i, \  \epsilon_i=\Big(0.4+0.4\Big|\sin(2\pi\frac{i}{n})\Big|) \eta_i. 
\end{equation*}
}
\item Two-regime model: Consider the following  self-exciting threshold auto-regressive (SETAR) model \cite{FY,HT2011}
\begin{equation*}
x_i=
\begin{cases}
a_1(\frac{i}{n})x_{i-1}+\epsilon_i, \ x_{i-1} \geq 0, \\
a_2(\frac{i}{n}) x_{i-1}+\epsilon_i, \ x_{i-1}<0. 
\end{cases}
\epsilon_i=\Big(0.4+0.4\Big|\sin(2\pi\frac{i}{n})\Big|) \eta_i.
\end{equation*}
It is easy to check that the SETAR model is stable if  { $a_j(\frac{i}{n}), \  j=1,2,$} are constants and bounded by one.
\item Markov two-regime switching model:  Consider the following Markov switching AR(1) model 
\begin{equation*}
x_i=
\begin{cases}
a_1(\frac{i}{n})x_{i-1}+\epsilon_i, \ s_i=0, \\
a_2(\frac{i}{n}) x_{i-1}+\epsilon_i, \ s_i=1.
\end{cases}
\epsilon_i=\Big(0.4+0.4\Big|\sin(2\pi\frac{i}{n})\Big|) \eta_i,
\end{equation*}
where the unobserved state variable $s_i$ is a discrete Markov chain taking values $0$ and $1,$ with transition probabilities $p_{00}=\frac{2}{3}, \ p_{01}=\frac{1}{3}, \ p_{10}=p_{11}=\frac{1}{2}.$  It is easy to check that the above model is stable if the
functions $a_j(\cdot), j=1,2,$ are constants and bounded by one \cite{REQ}. In the simulations, the initial state is chosen to be 1. 
\item Simple bilinear model: Consider the first order bilinear model
\begin{equation*}
x_i=\Big( a_1(\frac{i}{n})\epsilon_{i-1}+a_2(\frac{i}{n}) \Big)x_{i-1}+\epsilon_i , \ \epsilon_i=\Big(0.4+0.4\Big|\sin(2\pi\frac{i}{n})\Big|) \eta_i.
\end{equation*}
It is known from \cite{FY} that when the functions $a_j(\cdot),j=1,2,$ are constants and bounded by one, $x_i$ has an ARMA representation and hence stable.  

\end{enumerate} 

In the simulations below, we record our results based on 1,000 repetitions and for Algorithm \ref{alg:boostrapping}, we choose $B=1,000.$ {For the choices of random variables $\eta_i, i=1,2,\cdots,$ we set $\eta_i$ to be student-$t$ distribution with degree of $5$, i.e., $t$(5), for models 1-2 and standard normal random variables for models 3-5.} For the purpose of comparison of accuracy, besides the above five models,  we also consider two strictly stationary time series models, model 6 for a stationary ARMA(1,1) and model 7 for a stationary SETAR. Furthermore, for the comparison of power, we  consider two non-stationary time series models whose errors have constant variances, denoted as models $6^\#$ and $7^\#.$
%
%
%
%
%
%
\begin{enumerate}
\item[6.] Linear time series: stationary ARMA(1,1) process. We consider the following  process
\begin{equation*}
x_i-0.5x_{i-1}=\epsilon_i+0.5 \epsilon_{i-1}, 
\end{equation*}
where $\epsilon_i$ are i.i.d. $\mathcal{N}(0,1)$ random variables.
\item[7.] Nonlinear time series: stationary SETAR. We consider the following model 
\begin{equation*}
x_i=
\begin{cases} 
0.4 x_{i-1}+\epsilon_i, & x_{i-1} \geq 0, \\
0.5 x_{i-1}+\epsilon_i, & x_{i-1}<0,
\end{cases}
\end{equation*}  
where $\epsilon_i$ are i.i.d. $\mathcal{N}(0,1)$ random variables. 
\end{enumerate}  

\begin{enumerate}
\item[$6^{\#}$.] Non-stationary linear time series. We consider the following process
\begin{equation*}
x_i=\delta \sin (4\pi\frac{i}{n}) x_{i-1}+\epsilon_i,
\end{equation*} 
where $\epsilon_i, i=1,2,\cdots,n,$ are i.i.d. standard normal random variables. 
\item[$7^{\#}$.] Piece-wise locally stationary linear time series.  We consider the following process
\begin{equation*}
x_i=
\begin{cases}
\delta \sin(4 \pi \frac{i}{n}) x_{i-1}+\epsilon_i,& 1 \leq i \leq 0.75n, \\
0.4 x_{i-1}+\epsilon_i, & 0.75n<i \leq n \ \text{and} \ x_{i-1} \geq 0, \\
0.3 x_{i-1}+\epsilon_i,  & 0.75n<i \leq n \ \text{and} \ x_{i-1} < 0, \\
\end{cases}
\end{equation*}
where $\epsilon_i, i=1,2,\cdots,n,$ are i.i.d. standard normal random variables. 
\end{enumerate}  

\begin{table}[ht]
\begin{center}
\setlength\arrayrulewidth{1pt}
\renewcommand{\arraystretch}{1.3}
{\fontsize{9}{9}\selectfont 
\begin{tabular}{|c|ccccc|ccccc|}
\hline
      & \multicolumn{5}{c|}{$\alpha=0.1$}                                                                                                                       & \multicolumn{5}{c|}{$\alpha=0.05$}                                                                                                                        \\ \hline
Basis/Model & \multicolumn{1}{c|}{1} & \multicolumn{1}{c|}{2} & \multicolumn{1}{c|}{3} & \multicolumn{1}{c|}{4} & \multicolumn{1}{c|}{5}  & \multicolumn{1}{c|}{1} & \multicolumn{1}{c|}{2} & \multicolumn{1}{c|}{3} & \multicolumn{1}{c|}{4} & \multicolumn{1}{c|}{5} \\ 
\hline
     & \multicolumn{10}{c|}{$n$=256}                                                                                                                                                                                                                                                                                          \\
   \hline
Fourier     &          0.132  & 0.11                          &                          0.12 &         0.13                  &                          0.11 &                    0.067   & 0.07    & 0.06   &                                     0.04 &       0.06                    \\
Legendre    &     0.091       & 0.136                          &                          0.13 &   0.12                        &           0.13                &              0.06    & 0.059         &  0.041  &                                     0.07 &                            0.07 \\
Daubechies-9    &  0.132 & 0.12   & 0.11                        &      0.133             &         0.132                 & 0.063                 &0.067 &  0.059  &          0.068                            &                 0.065   \\
\hline
      & \multicolumn{10}{c|}{$n$=512}                                                                                                                                                                                                                                                                                         \\
       \hline
Fourier     &            0.09      & 0.13                 &                          0.11 &      0.13                     &  0.127  &                         0.05 & 0.06 & 0.067   &    0.068                                  &              0.069                 \\
Legendre     &    0.09    & 0.094    &                  0.092                                &          0.12                 &                       0.118   & 0.04   & 0.058                       & 0.07    &        0.043                              &                            0.057 \\
Daubechies-9     &  0.091  & 0.11    &      0.098                   &                         0.11   & 0.118                          &          0.048               &  0.052 & 0.054 &                                 0.053  &                  0.054           \\
 \hline
\end{tabular}
}
\end{center}
\caption{Simulated type I errors using the setup (\ref{eq_testingcases}). The models are listed in Section \ref{simu_intro} and the basis functions can be found in Section \ref{sec_basisfunctions}. The results are reported based on 1,000 simulations. We can see that our multiplier bootstrap procedure is reasonably accurate for both $\alpha=0.1$ and $\alpha=0.05$. 
}
\label{table_typei}
\end{table}

\begin{table}[ht]
\begin{center}
\setlength\arrayrulewidth{1pt}
\renewcommand{\arraystretch}{1.3}
{\fontsize{9}{9}\selectfont 
\begin{tabular}{|c|ccccc|ccccc|}
\hline
      & \multicolumn{5}{c|}{$\delta=0.2/0.5$}                                                                                                                       & \multicolumn{5}{c|}{$\delta=0.35/0.7$}                                                                                                                        \\ \hline
Basis/Model & \multicolumn{1}{c|}{1} & \multicolumn{1}{c|}{2} &  \multicolumn{1}{c|}{3} & \multicolumn{1}{c|}{4} & \multicolumn{1}{c|}{5}  & \multicolumn{1}{c|}{1} & \multicolumn{1}{c|}{2} &  \multicolumn{1}{c|}{3} & \multicolumn{1}{c|}{4} & \multicolumn{1}{c|}{5} \\ 
\hline
      & \multicolumn{10}{c|}{$n$=256}                                                                                                                                                                                                                                                                                          \\
   \hline
Fourier     & 0.84     & 0.86                              &                          0.84 &                  0.837          &                      0.94 &                        0.97 &  0.97 & 0.96  & 0.99                                      &              0.98              \\
Legendre     &         0.8   & 0.806                          &                          0.81 &               0.84           & 0.83                          & 0.97    & 0.968                       &  0.95  &                                     0.97 &        0.91                     \\
Daubechies-9    &     0.81 & 0.81 & 0.86                        &                       0.81   & 0.81                          &           0.97       & 0.96 & 0.983  & 0.98                                    &        0.98           \\
\hline
      & \multicolumn{10}{c|}{$n$=512}                                                                                                                                                                                                                                                                                         \\
       \hline
Fourier     &  0.91         & 0.9                       &                           0.96&          0.9              &  0.93                      &           0.96               &  0.97  &  0.973 &    0.98                                &                  0.97             \\
Legendre    &     0.9         & 0.91                       &                          0.92&                        0.893   & 0.91                           &           0.94    & 0.95          &0.98  &        0.97                             &        0.96   \\
Daubechies-9     &             0.87  & 0.88                       &                      0.93     &  0.91                          &                       0.91    &  0.96                         & 0.99 & 0.97  &        0.97                              &  0.96                             \\
 \hline
\end{tabular}
}
\end{center}
\caption{Simulated power under the setup (\ref{eq_testingcasesalternative}) using nominal level $0.1.$ For models 1-2, we consider the cases $\delta=0.2 $ and $\delta=0.35$, whereas for models 3-5, we use $\delta=0.5$ and $\delta=0.7.$ The results are based on 1,000 simulations. 
}
\label{table_power}
\end{table}

\begin{table}[ht]
\begin{center}
\setlength\arrayrulewidth{1pt}
\renewcommand{\arraystretch}{1.3}
\newcolumntype{L}[1]{>{\raggedright\arraybackslash}p{#1}}
\hspace*{-3cm}
{\fontsize{8}{8}\selectfont 
\begin{tabular}{|p{0.9cm}|p{0.75cm}p{0.7cm}p{0.7cm}p{0.7cm}p{0.7cm}p{0.7cm}p{0.7cm}p{0.8cm}|p{0.75cm}p{0.7cm}p{0.7cm}p{0.7cm}p{0.7cm}p{0.7cm}p{0.7cm}p{0.8cm}|}
\hline
      & \multicolumn{8}{c|}{$\alpha=0.1$}                                                                                                                       & \multicolumn{8}{c|}{$\alpha=0.05$}                                                                                                                        \\ \hline
Model & \multicolumn{1}{l|}{NS-$\mathcal{L}^2$} & \multicolumn{1}{l|}{S-$\mathcal{L}^2$} & \multicolumn{1}{l|}{KS} & \multicolumn{1}{l|}{DFT1} & \multicolumn{1}{l|}{DFT2} & \multicolumn{1}{l|}{DFT3} & \multicolumn{1}{l|}{HWT} & MB & \multicolumn{1}{l|}{NS-$\mathcal{L}^2$} & \multicolumn{1}{l|}{S-$\mathcal{L}^2$} & \multicolumn{1}{l|}{KS} & \multicolumn{1}{l|}{DFT1} & \multicolumn{1}{l|}{DFT2} & \multicolumn{1}{l|}{DFT3} & \multicolumn{1}{l|}{HWT} & MB \\ \hline
      & \multicolumn{16}{c|}{$n$=256}                                                                                                                                                                                                                                                                                         \\ \hline
1     &          0.08  & 0.084 & 0.078                          &                          0.148 &         0.057                  &                        0.13   &                          0.18 & 0.132   &                                     0.024 & 0.032 & 0.029                          &                          0.067 &     0.017                      &         0.063                  &                          0.083& 0.063   \\
2     &      0.081 & 0.088 & 0.082                                                       &                          0.097 &    0.068                       &                          0.12 &       0.085                    &    0.12  &           0.038                          & 0.042 & 0.039                          &                          0.04 &     0.07                     &          0.057                &               0.028         &  0.067   \\
3     &     0.171                                &0.19 & 0.134 &                          0.183 &   0.04                        &                          0.137 &                          0.227&  0.11  &                                     0.087 & 0.087& 0.092&                         0.103 &    0.011                      &           0.033               &                          0.093 &   0.059 \\
4     &     0.2      & 0.24 & 0.19                           &                        0.163   &      0.05                    &                          0.12 &                   0.176       & 0.133   &                                     0.077 &  0.086& 0.082&                  0.087       &  0.013 &              0.034             &               0.113           & 0.068    \\
5     &   0.46      & 0.39 & 0.23                            &                          0.293 &         0.077                  &                          0.19 &                          0.153 & 0.132  &                                     0.29 & 0.19 & 0.15&                       0.21    &        0.03                   &             0.14              &       0.12 & 0.065   \\ 
6     &         0.11 & 0.103& 0.089                         &                          0.105 &           0.096                &                          0.09 &                 0.087          &  0.088  &   0.047 & 0.053 & 0.059                                   &       0.053                    &        0.053                 &             0.039              &               0.052          &  0.057   \\
7     &      0.051  & 0.088 & 0.13                             &                          0.097 &    0.08                       &                          0.092 &       0.085                    &    0.127  &           0.018   & 0.039 & 0.036                        &                          0.04 &     0.06                     &          0.047                &               0.038         &  0.061   \\
\hline
      & \multicolumn{16}{c|}{$n$=512}                                                                                                                                                                                                                                                                                         \\ \hline
1     &   0.087                                  & 0.089 & 0.083 &                          0.127 &       0.03                    &                          0.13 &     0.237                     &  0.091  &                                     0.023 & 0.035 & 0.03&     0.1                     &                          0.02 &                 0.043          &                         0.137 &  0.048  \\
2     &      0.051   & 0.087 & 0.079                          &                          0.096 &    0.085                       &                          0.093 &       0.075                    &    0.11  &           0.026    & 0.035 & 0.033                       &                          0.036 &     0.067                     &          0.044                &               0.033         &  0.052   \\
3     &            0.26 & 0.17& 0.22                          &                          0.16 &      0.04                     &                          0.117 &       0.243                  & 0.098   &                                    0.127  & 0.18 & 0.233&    0.1                      &                          0.007 &      0.037                     &                         0.14 & 0.054   \\
4     &  0.287  & 0.22 & 0.197                                  &                          0.167 &             0.027              &                          0.09 &         0.247                 & 0.11   &                                     0.177 &  0.16 & 0.11&      0.103                   &                          0.013 &           0.073                &                         0.163 & 0.053    \\
5     &     0.64   & 0.26 & 0.22                              &                          0.303 &        0.087                   &                          0.283 &          0.35                &   0.118 &                                     0.413 & 0.15& 0.17&     0.26                     &                          0.063 &      0.167                     &                         0.23 & 0.054   \\
6    &    0.11 & 0.1 & 0.104                                 &                          0.093 &        0.084                   &                          0.088 &      0.088                     &   0.092 &                                     0.035 &  0.043 & 0.048 &       0.046                 &   0.047      &                        0.048 &    0.053    &  0.048  \\
7     &      0.051 & 0.078 & 0.077                                & 0.087                           &          0.113             &      0.083                     &       0.093                    &  0.092  &                                     0.013 & 0.037 & 0.039& 0.037                         &                          0.047 &       0.043                    &                        0.04 & 0.051   \\
 \hline
\end{tabular}
}
\end{center}
\caption{Comparison of accuracy for models 1-7 using different methods.
}
\label{table_compare_typeone}
\end{table}

\begin{table}[ht]
\begin{center}
\setlength\arrayrulewidth{1pt}
\renewcommand{\arraystretch}{1.3}
\newcolumntype{L}[1]{>{\raggedright\arraybackslash}p{#1}}
\hspace*{-3cm}
{\fontsize{8}{8}\selectfont 
\begin{tabular}{|p{0.9cm}|p{0.75cm}p{0.7cm}p{0.7cm}p{0.7cm}p{0.7cm}p{0.7cm}p{0.7cm}p{0.8cm}|p{0.75cm}p{0.7cm}p{0.7cm}p{0.7cm}p{0.7cm}p{0.7cm}p{0.7cm}p{0.8cm}|}
\hline
      & \multicolumn{8}{c|}{$\delta=0.2/0.5$}                                                                                                                       & \multicolumn{8}{c|}{$\delta=0.35/0.7$}                                                                                                                        \\ \hline
Model & \multicolumn{1}{l|}{NS-$\mathcal{L}^2$} & \multicolumn{1}{l|}{S-$\mathcal{L}^2$} & \multicolumn{1}{l|}{KS} & \multicolumn{1}{l|}{DFT1} & \multicolumn{1}{l|}{DFT2} & \multicolumn{1}{l|}{DFT3} & \multicolumn{1}{l|}{HWT} & MB & \multicolumn{1}{l|}{NS-$\mathcal{L}^2$} & \multicolumn{1}{l|}{S-$\mathcal{L}^2$} & \multicolumn{1}{l|}{KS} & \multicolumn{1}{l|}{DFT1} & \multicolumn{1}{l|}{DFT2} & \multicolumn{1}{l|}{DFT3} & \multicolumn{1}{l|}{HWT} & MB \\ \hline
      & \multicolumn{16}{c|}{$n=256$}                                                                                                                                                                                                                                                                                         \\ \hline
1     &      0.263                                &   0.29 & 0.33&                              0.14 & 0.03                   &                          0.07 &                  0.3         & 0.81  &                             0.503       &  0.49 & 0.43&                         0.113 &               0.053            &       0.089                   &                          0.4 & 0.97   \\
2     &      0.183  & 0.28 & 0.3                              &                          0.497 &    0.08                       &                          0.092 &       0.585                    &    0.81  &           0.68 & 0.69 & 0.85                          &                          0.14 &     0.06                     &          0.047                &               0.38         &  0.96   \\
3     &         0.44 & 0.53 & 0.62                             &                          0.153 &           0.04                &                          0.16 &                   0.393       &  0.86  &                          0.7           &   0.84& 0.79&                       0.14 &               0.05            &        0.09                   &                         0.64 &    0.983 \\
4     &    0.603   & 0.598 & 0.71                               &                         0.16 &         0.04                  &                          0.203 &            0.44              &  0.81 &                             0.86 & 0.83 & 0.79       &                 0.2          &          0.07                 &           0.12               &       0.647                  &  0.98  \\
5     &      0.92   & 0.83 & 0.88                             &                          0.243 &             0.143              &                          0.24 &        0.57                 &    0.81 &                           0.997           &    0.9& 0.85&                      0.347 &            0.193               &       0.397                    &                         0.797 &  0.98   \\
$6^{\#}$     &      0.697                               & 0.88 & 0.9 &                          0.12 &             0.093             &                          0.11 &        0.327                  & 0.86   &                                    0.923 & 0.93& 0.98&  0.16                         &            0.15               &       0.15                    &                         0.563 & 0.94     \\
$7^{\#}$     &      0.463                               & 0.54 & 0.69 &                          0.137 &             0.107             &                          0.133 &        0.273                 &  0.85   &                           0.81          & 0.79 & 0.87&                          0.193 &            0.203             &       0.223                   &                         0.483 &  0.96   \\
 \hline
      & \multicolumn{16}{c|}{$n=512$}                                                                                                                                                                                                                                                                                         \\ \hline
1     &     0.477                                & 0.532& 0.48 &                          0.173 &        0.04                   &                          0.08 &            0.52              &  0.87  &                                     0.857 &    0.86 & 0.86&     0.137                  &                          0.03 &         0.1                  &                         0.75 &  0.96  \\
2     &      0.51   & 0.58 & 0.49                            &                          0.297 &    0.082                       &                          0.092 &       0.385                    &    0.88  &           0.918                          & 0.9 & 0.94 &                          0.24 &     0.06                     &          0.047                &               0.838         &  0.99   \\
3     &        0.657                             & 0.66 & 0.65&                          0.24 &          0.05                 &                          0.083 &                  0.61        & 0.93   &                                     0.96 & 0.98& 0.93&      0.17                    &                          0.24 &      0.113                     &                         0.95 &  0.97  \\
4     &   0.84 & 0.8 & 0.798                                  &                          0.23 &                    0.043       &                          0.143 &               0.773           &0.91    &                                     0.987 &  0.97& 0.976&     0.293                     &                          0.053 &             0.19              &                         0.97 &0.97    \\
5     &          0.963 & 0.93& 0.95                          &                          0.297 &  0.127                         &                          0.263 &  0.87                         &  0.91  & 0.983                                     & 0.97 &0.964 &          0.523                 &     0.24                      &       0.478 &    0.994 & 0.96  \\
$6^{\#}$     & 0.847                                & 0.89 & 0.94 &                          0.147 &             0.087          &                          0.103 &          0.67              &  0.88   &                                     0.95 &        0.96 & 0.98&        0.13            &                      0.09 &                         0.133 &                         0.963 &  0.95   \\
$7^{\#}$     &    0.69                              & 0.79 & 0.88 &                          0.14 &         0.13                  &                          0.217 &       0.383                &  0.91 &                                     0.953 & 0.99 & 0.978&            0.3             &                          0.313 &           0.383                &                         0.823 &  0.943    \\ 
 \hline
\end{tabular}
}
\end{center}
\caption{ Comparison of power at nominal level 0.1 using different methods.
}
\label{table_comprare_power}
\end{table}

{
\subsection{Performance of forecasting for locally stationary time series}\label{sec_suppl_forecastingexamination} In this subsection, we study the prediction accuracy of our proposed adaptive sieve forecast (\ref{eq_forcaestequation}) by comparing it with some state-of-the-art methods. Specifically, we compare with the Tapered Yule-Walker estimate (TTVAR) in \cite{RSP}, the non-decimated wavelet estimate (LSW) in \cite{FBS}, the model switching method (SNSTS) in \cite{KPF},  the best linear prediction using all the previous samples (SBLP) \footnote{The prediction is based on the stationary assumption and an ARMA model.}, the best linear prediction using $b$ recent samples (PBLP) and our adaptive sieve forecast (\ref{eq_forcaestequation}) (Sieve). We implement TTVAR with constant taper function $g \equiv 1$ and the bandwidth is selected according to \cite[Corollary 4.2]{RSP}. For the wavelet method, we use the matlab codes from the first author's website (see \url{http://stats.lse.ac.uk/fryzlewicz/flsw/flsw.html}) and for the model switching method, we use the R package \texttt{forecastSNSTS}. For our sieve method, we use the orthogonal wavelets (\ref{eq_meyerorthogonal}) with Daubechies-9 wavelet and the data-driven approach described in Section \ref{sec:choiceparameter} to choose $b$ and $c.$ 

In Table \ref{table_compare_prediction}, we record the mean square error over 1,000 simulations  for one-step ahead prediction of the models 1-5 in Section \ref{simu_intro} with the coefficients chosen according to (\ref{eq_testingcasesalternative}). Here we choose
$\delta=0.35$ for models 1-2 and $\delta=0.5$ for models 3-5. It can be seen that our proposed method outperforms the other methods in literature for five models in both sample sizes $n=256$ and $n=512$. The forecasting accuracy improvement is more significant for non-AR type models such as the MA and bilinear models.

\begin{table}[ht]
\begin{center}
\setlength\arrayrulewidth{1pt}
\renewcommand{\arraystretch}{1.5}
{\fontsize{10}{10}\selectfont 
\begin{tabular}{|c|ccccccc|lllllll|}
\hline
Model & \multicolumn{1}{c|}{TTVAR} & \multicolumn{1}{c|}{LSW} & \multicolumn{1}{c|}{SNSTS} & \multicolumn{1}{c|}{SBLP} & \multicolumn{1}{c|}{PBLP} & \multicolumn{1}{c|}{Sieve} & Improvement  \\ \hline
      & \multicolumn{7}{c|}{$n$=256}                                                                                                                                                                                                                                                                                         \\ \hline
1     &   0.24                                  &     0.21                     &     0.45                      &                     0.284      &                          0.24 & {\bf 0.189} & 10 $\%$  \\
2     &                 0.28                    &      0.27                   &       0.28                   &                  0.273         &                 0.283        &  {\bf 0.22} & 18.5 $\%$  \\ 
3     &     0.21                                 &                          0.185 & 0.198                         &                 0.241          &                          0.194 &  {\bf 0.178} & 3.8 $\%$  \\
4     &        0.207                              &                          0.195 & 0.2                         &                      0.247     &                   0.199       & {\bf 0.187} & 4.1 $\%$     \\
5     &                 0.22                    &      0.22                   &       0.24                   &                  0.246         &                 0.273        &  {\bf 0.176} & 20 $\%$  \\ 
\hline
      & \multicolumn{7}{c|}{$n$=512}                                                                                                                                                                                                                                                                                         \\ \hline
1     &                        0.21             &                          0.2 & 0.2                          &                   0.233      &                         0.209 &  {\bf 0.181} & 9.5 $\%$\\
2     &                 0.26                   &      0.26                   &       0.264                   &                  0.276         &                 0.283        &  {\bf 0.196} & 24.62 $\%$ \\ 
3     &    0.207                                &                          0.183 & 0.192                         &                0.213           &                      0.194   & {\bf 0.18} & 1.7 $\%$   \\
4     &      0.205                               &                          0.175 & 0.188                           &                   0.211        &         0.181              & {\bf 0.17 } & 2.86 $\%$   \\
5     &       0.23                               &                          0.21 & 0.24                          &       0.23                    &                        0.22 &  {\bf  0.183} & 12.86 $\%$    \\
 \hline
\end{tabular}
}
\end{center}
\caption{Comparison of prediction accuracy for models 1-5 using different methods.  We highlight the smallest mean square errors and record the percentage of improvement of our method compared with the next best method.  }
\label{table_compare_prediction}
\end{table}

Before concluding this subsection, we compare the (rolling) forecasting performance of the aforementioned different methods for the real application (i.e., the stock return data of the Nigerian Breweries Plc.) in Section \ref{sec:realdata}. Especially,  we use the time series 2008-2014 as the training dataset to study the prediction performance over the first month of 2015. We employ the data-driven approach from Section \ref{sec:choiceparameter} to choose the parameters and we obtain that $b=5$ and $J_n=3.$
The MSE is  $0.189$ for our proposed Sieve prediction (\ref{eq_forcaestequation}).  We compare this result with the other methods and record the results in Table \ref{table:fin}. We find that our prediction performs better than the other methods. Especially,  we get a $20.9 \%$ improvement compared to simply fitting a stationary model using all the time series from 2008 to 2014.    
\begin{table}[H]
\center{
\begin{threeparttable}
\begin{tabular}{ccccccc}
\hline
Method & {\bf Sieve} & TTVAR & LSW & SNSTS & SBLP & PBLP \\ \hline
MSE  &  {\bf 0.189}       &    0.198        &   0.198            &      0.202        &  0.239            & 0.249 \\ 
\hline
\end{tabular}
\end{threeparttable}}
\caption{Comparison of prediction accuracy for stock return data of Nigerian Breweries Plc. For SBLP,  we use all the time series from 2008 to 2014 to fit a stationary ARMA model. For PBLP, we used the most recent $b=5$ samples. }\label{table:fin}
\end{table}

Finally, we make a comment on our proposed method with the model choice methodology proposed by \cite{KPF}, where one of two competing approaches was chosen based on its empirical, finite-sample performance with respect to forecasting in terms of the empirical mean squared prediction error (MSPE). The two competing approaches are: a stationary AR($p$) prediction model and a time-varying AR($p$) model. For the stationary AR($p$) model, the paper estimated the coefficients using the standard Yule-Walker equation, and for the time-varying AR($p$) model, the paper used the stationary method on short overlapping segments of the time series as in \cite{dahlhaus1998optimal}. Our method utilizes all the data points and \cite{KPF} only uses the most recent data points via segmentation. In practice, the data generating mechanism can be complicated so that the stationary period can be very short. Therefore, the segmentation approach could be misleading. In fact, based on our simulations and data analysis, we see improvements in all the simulations and real data analysis. In this regard, for the estimation of the time-varying AR model of \cite{KPF}, we suggest using our proposed optimal sieve prediction instead of the segment based estimator.

\subsection{Discussion on the robustness of the choices of parameters}\label{sec_suppl_robustparameter} In this subsection, we use Monte-Carlo simulations to conduct sensitivity analysis to our proposed statistics and the multiplier bootstrap procedure. Especially, we will examine the robustness of the key hyperparameters $b_*, c$ and $m.$ In what follows, we focus on reporting the results of  Models 2 and 5 as in Section \ref{simu_intro}. In fact, we also conducted simulations for the AR type Models 1,3 and 4,  the results and discussions are similar. Due to space constraint, we will not report these results here.

First, we examine the sensitivity of the hyperparameters to the simulated type I error under the null hypothesis (\ref{eq_testingcases}). As concluded in Section \ref{sec_poer}, the performance of different basis functions are similar in general. Hence, for ease of discussion, we use the Fourier basis functions in the following simulations. Moreover, we use  the sample size $n=512$ and focus on the type I error $\alpha=0.1.$ For other simulation settings (e.g., $n=256$ or $\alpha=0.5$), the results are similar and we will not present the detail to avoid distraction.  

In order to have a through understanding of the impact of all three parameters, we discuss a wide range of all three parameters so that
\begin{equation*}
(b_*, c, m) \in \llbracket 1, 20 \rrbracket \times  \llbracket 1, 20 \rrbracket \times \llbracket 1, 20 \rrbracket.
\end{equation*}
Then for each chosen triplet $(b_*, c, m),$ we apply our proposed multiplier bootstrap procedure Algorithm \ref{alg:boostrapping} to obtain the simulated type I error $\widehat{\alpha}\equiv \widehat{\alpha}(b_*, c, m).$ For comparison, for each triplet, we consider the discrepancy between the true type error $\alpha=0.1$ and the simulated type I error $\widehat{\alpha}$
\begin{equation*}
\mathrm{d} \equiv \mathrm{d}(b_*,c, m):=|\widehat{\alpha}-0.1|.
\end{equation*}
In order to visualize how the values of $\mathrm{d}$ change with each of the parameters, we discuss them one by one by fixing the rest two of the parameters. The results are recorded in Figure \ref{fig_senm1} for Model 2 and in Figure \ref{fig_senm4} for Model 5. Based on the simulations, we see that the performance is overall robust against the choices of the parameters. Especially, we observe a U curve for $b_*$ and $m$, and the choices of these two parameters are quite flexible. For example, $b_* \in \llbracket 5,14 \rrbracket$ and $m \in \llbracket 4,16 \rrbracket$ can result in accurate testings. Moreover, we also observe that our results are more sensitive to the value of $c$. Under the null hypothesis (\ref{eq_testingcases}) (i.e., $c=1$), we should choose a smaller value of $c$ like $c \in \llbracket 1,3 \rrbracket.$ This range is much narrower than those of $b_*$ and $m.$ We emphasize that our proposed data-driven procedure in Section \ref{sec:choiceparameter} can successfully provide a triplet $(b_*,m, c)$ lying in the range of these hyperparameters which can result in accurate testings. For example, our data-driven procedure will select $(b_*, m, c)=(7,8,1)$ for Model 2 and $(b_*, m, c)=(6,6,1)$ for Model 5 which match our sensitivity analysis. 

\begin{figure}[!ht]
\hspace*{-2.0cm}
\begin{subfigure}{0.35\textwidth}
\includegraphics[width=5.1cm,height=4.8cm]{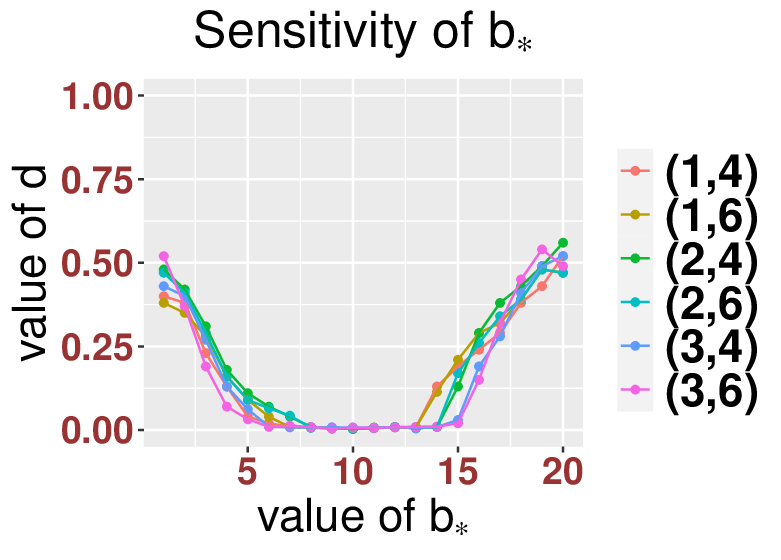}
\end{subfigure}
\begin{subfigure}{0.35\textwidth}
\includegraphics[width=5.1cm,height=4.8cm]{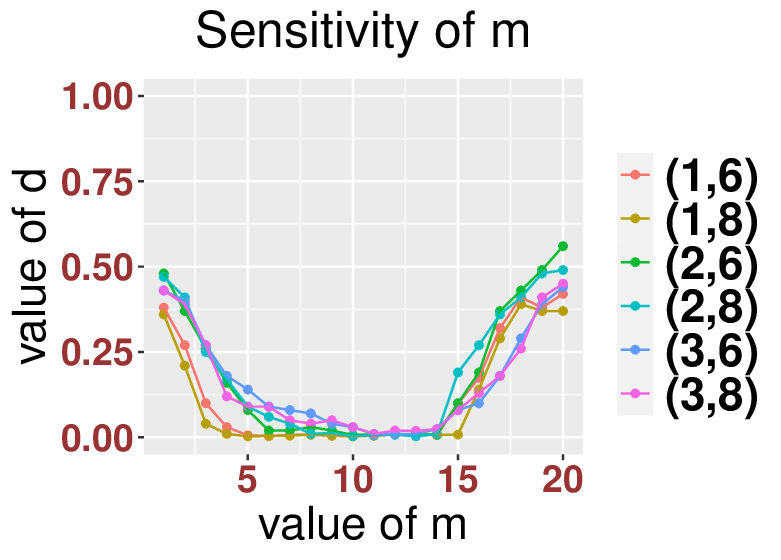}
\end{subfigure}
\begin{subfigure}{0.35\textwidth}
\includegraphics[width=5.1cm,height=4.8cm]{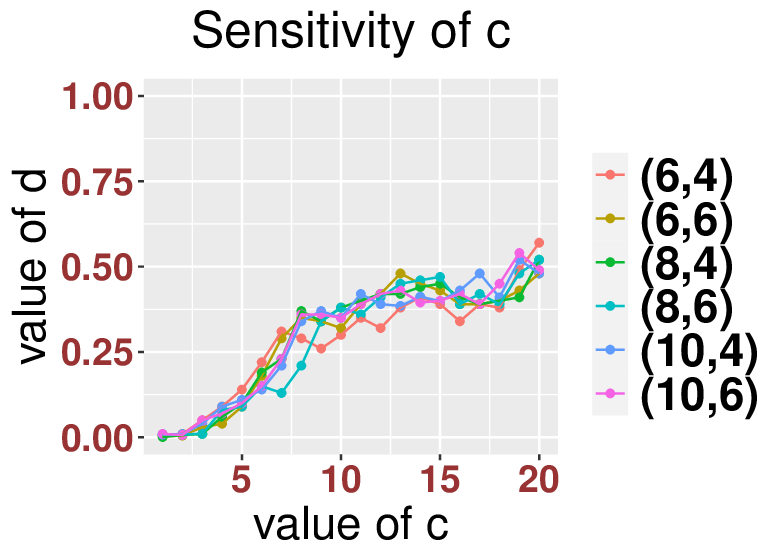}
\end{subfigure}
\caption{ {\footnotesize Type I error sensitivity analysis of hyperparameters for Model 2. We study each parameter by fixing the rest two of them. For example, the left panel studies the parameter $b_*$ by considering various fixed combinations of $c$ and $m.$ The legends there are defined as follows. $(1,4)$ means that $c=1$ and  $m=4.$ Similar definitions apply to the other two figures. The simulation results are based on 1,000 repetitions. }  }
\label{fig_senm1}
\end{figure}

\begin{figure}[!ht]
\hspace*{-2.0cm}
\begin{subfigure}{0.35\textwidth}
\includegraphics[width=5.1cm,height=4.8cm]{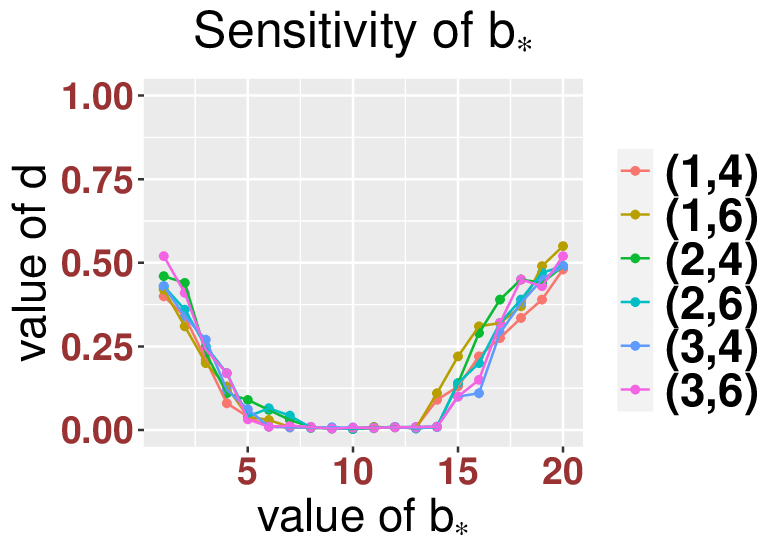}
\end{subfigure}
\begin{subfigure}{0.35\textwidth}
\includegraphics[width=5.1cm,height=4.8cm]{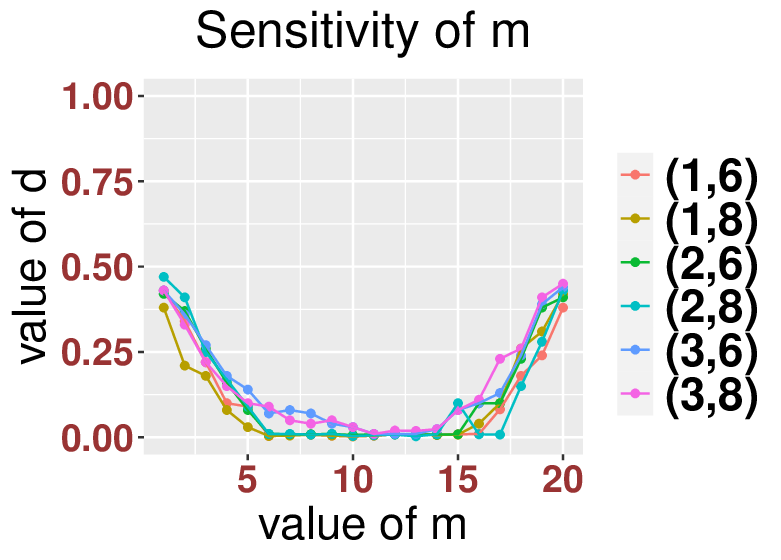}
\end{subfigure}
\begin{subfigure}{0.35\textwidth}
\includegraphics[width=5.8cm,height=5cm]{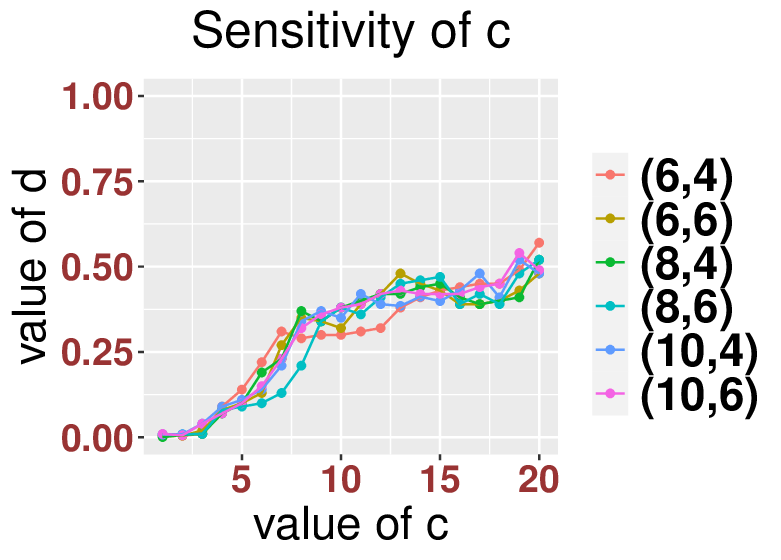}
\end{subfigure}
\caption{ {\footnotesize Type I error sensitivity analysis of hyperparameters for Model 5. See the caption of Figure \ref{fig_senm1}  for more detail of the legends.}  }
\label{fig_senm4}
\end{figure}

Second, we examine the sensitivity of the hyperparameters to the simulated power under the alternative hypothesis (\ref{eq_testingcasesalternative}) with $\delta=0.35$ for Model 2 and $\delta=0.7$ for Model 5. We mention that other values of $\delta$ have similar performance and we will not report such results. Analogous to the findings for type I error,  we see that the performance is overall robust against the choices of the parameters in terms of power. Especially, we observe an upside down U curve for $b_*$ and $m$, and the choices of these two parameters are quite flexible. For example, $b_* \in \llbracket 5,16 \rrbracket$ and $m \in \llbracket 4,16 \rrbracket$ can result in accurate testings. We mention that the U shape is not as clear as that in the type I error analysis. The main reason is because it is very likely that larger values of $b_*$ and $m$ will result in larger errors so that the testing statistic will be in favor of rejecting.  Moreover, we also observe that our results are more sensitive to the value of $c$. Under the alternative (\ref{eq_testingcasesalternative}) (i.e., $c=2$), we should choose $c$ that $c \geq 2.$ We also observe that no U curve exists for $c$. The main reason is that a larger value of $c$ will increase the estimation error significantly. Therefore, the value of the test statistics will be much larger so that it will reject the null hypothesis.  We emphasize that our proposed data-driven procedure in Section \ref{sec:choiceparameter} can successfully provide a triplet $(b_*,m, c)$ lying in the range of these hyperparameters which can result in accurate and powerful testings. For example, our data-driven procedure will select $(b_*, m, c)=(7,7,2)$ for Model 2 and $(b_*, m, c)=(6,7,3)$ for Model 5 which match our sensitivity analysis.

\begin{figure}[!ht]
\hspace*{-2.0cm}
\begin{subfigure}{0.35\textwidth}
\includegraphics[width=5.1cm,height=4.8cm]{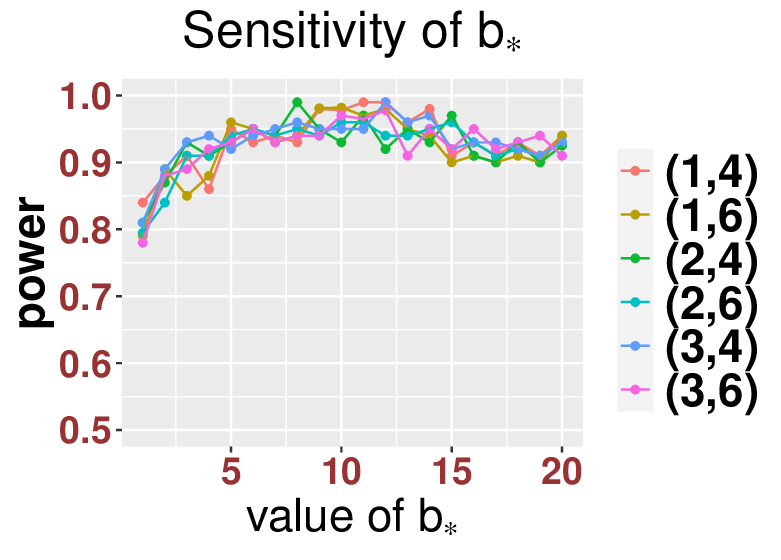}
\end{subfigure}
\begin{subfigure}{0.35\textwidth}
\includegraphics[width=5.1cm,height=4.8cm]{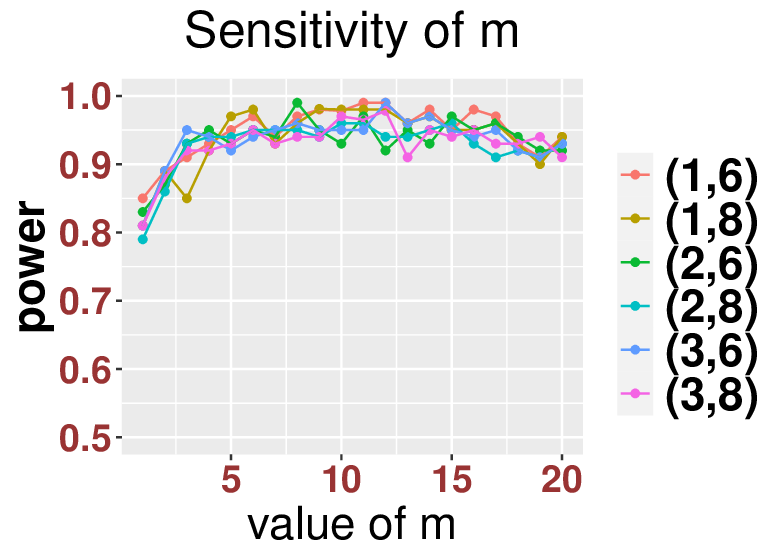}
\end{subfigure}
\begin{subfigure}{0.35\textwidth}
\includegraphics[width=5.1cm,height=4.8cm]{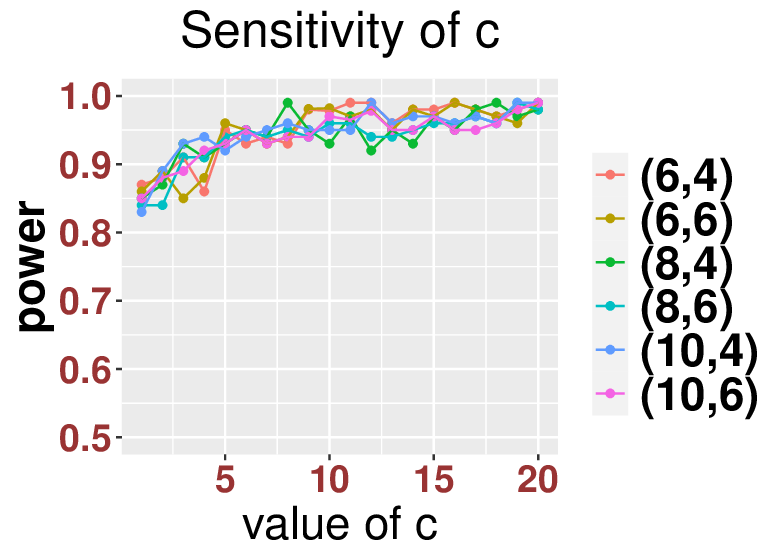}
\end{subfigure}
\caption{ {\footnotesize Power sensitivity analysis of hyperparameters for Model 2. See the caption of Figure \ref{fig_senm1}  for more detail of the legends. }}
\label{fig_senm1p}
\end{figure}

\begin{figure}[!ht]
\hspace*{-2.0cm}
\begin{subfigure}{0.3\textwidth}
\includegraphics[width=5.8cm,height=5cm]{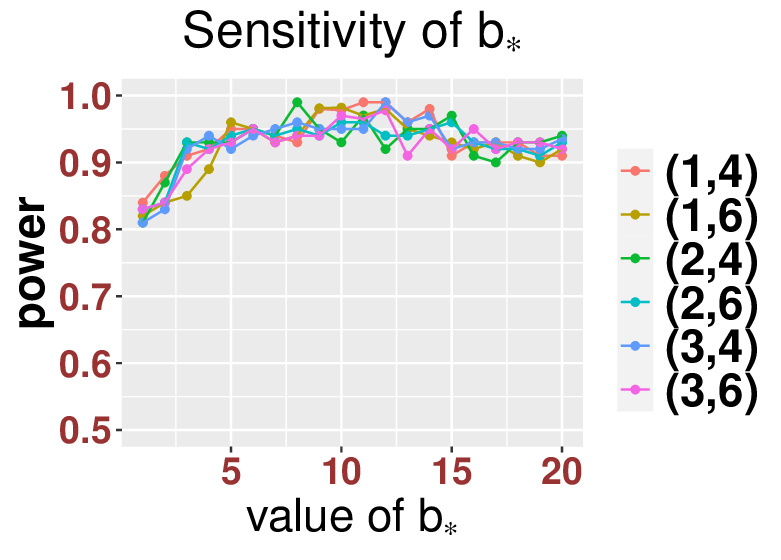}
\end{subfigure}
\begin{subfigure}{0.3\textwidth}
\includegraphics[width=5.8cm,height=5cm]{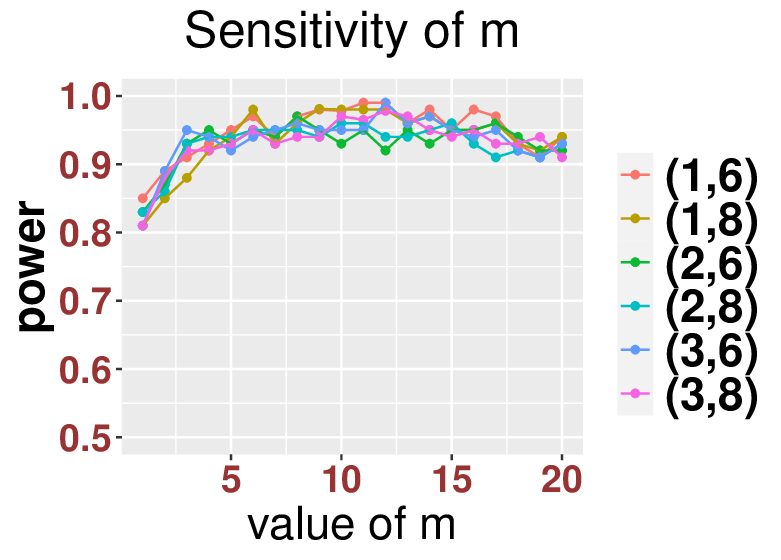}
\end{subfigure}
\begin{subfigure}{0.3\textwidth}
\includegraphics[width=5.8cm,height=5cm]{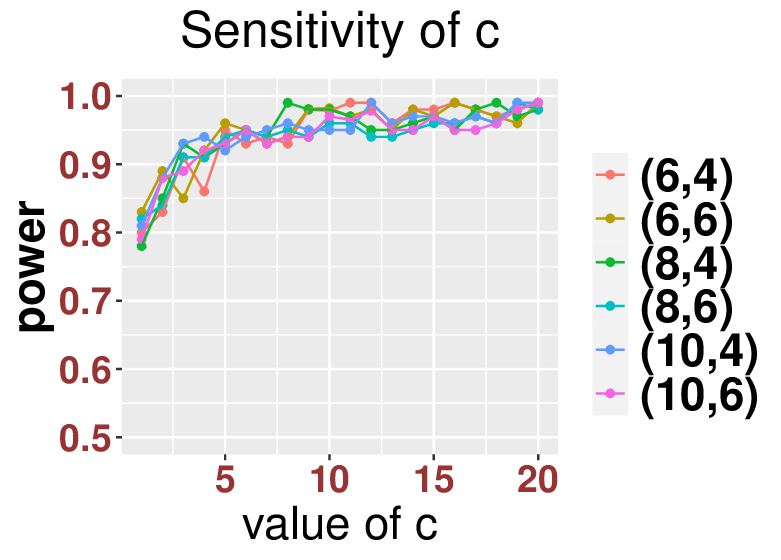}
\end{subfigure}
\caption{ {\footnotesize Power sensitivity analysis of hyperparameters for Model 5. See the caption of Figure \ref{fig_senm1}  for more detail of the legends. } }
\label{fig_senm4p}
\end{figure}


\subsection{Examination for roots of AR polynomials: additional simulation setting}\label{sec_sub_suppledalhaus} In this section, we consider some additional simulations using the locally stationary AR(2) considered by Dahlhaus in \cite{dahlhaus1997fitting}. Especially, Dahlhaus considered the model that
\begin{equation}\label{eq_ar2model}
x_i=a_1(i/n) \epsilon_{i-1}+a_2(i/n) \epsilon_{i-2}+\epsilon_i,
\end{equation} 
where $\epsilon_i, 1 \leq i \leq n,$ are  i.i.d. Gaussian random variables with variance $\sigma^2_i$ 
and 
\begin{equation}\label{eq_additionalexample}
a_1(t)=-1.8 \cos(1.5-\cos 4 \pi t), \ a_2(t) \equiv 0.81. 
\end{equation}
It was shown in Section 6 of \cite{dahlhaus1997fitting} that for the above AR(2) model, when $t$ is fixed, the roots of the AR(2) characteristics polynomials $z_{\pm} \equiv z_{\pm}(t)$ are complex number so that 
\begin{equation*}
z_{\pm}=\frac{10}{9} \exp(\pm \mathrm{i}(1.5-\cos 4 \pi t)).
\end{equation*}
Note that the AR polynomial of the above model has complex roots and these roots are rather close to the unit circle.

In Dahlhaus's original paper, he considered the  standard Gaussian random variable that $\sigma_i \equiv 1$ for all $1 \leq i \leq n.$ In our current paper, we consider the locally stationary white noise process as in Section \ref{simu_intro} so that 
\begin{equation}\label{eq_variancesettup}
\sigma_i=0.4+0.4 \left|\sin(2 \pi i/n) \right|.
\end{equation} 
In order to examine the accuracy and power of our proposed test, we consider the setting that (\ref{eq_additionalexample}) holds for some pre-given $t_0.$ That is to say, for the model (\ref{eq_ar2model}),  to study the simulated type I error rate, for some fixed $0 \leq t_0 \leq 1,$ we consider the null hypothesis that 
\begin{equation}\label{eq_h0additional}
\mathbf{H}_0: \ a_1(i/n) \equiv -1.8 \cos(1.5-\cos 4 \pi t_0), \ a_2(i/n) \equiv 0.81. 
\end{equation}
Moreover, to study the power, we consider the following alternative as in (\ref{eq_additionalexample}), i.e., 
\begin{equation}\label{eq_h0additiona11}
\mathbf{H}_a:  \ a_1(i/n) \equiv -1.8 \cos(1.5-\cos 4 \pi i/n), \ a_2(i/n) \equiv 0.81. 
\end{equation} 

First, we study the finite-sample accuracy of our test under the null hypothesis (\ref{eq_h0additional}) for various choices of $t_0=0, 0.6, 1$ with the time-varying Gaussian white noise with (\ref{eq_variancesettup}).  We also examine how the sample size $n$ can affect the accuracy. For simplicity, we focus on the type I error $\alpha=0.1$ and analyzing the discrepancy between the simulated type I error $\widehat{\alpha}$ and $0.1,$ denoted as 
\begin{equation*}
\mathrm{d} \equiv \mathrm{d}(n, t_0)=|\widehat{\alpha}-0.1|.  
\end{equation*}  
For convenience, in this subsection, we report the results for Fourier basis functions. For orthogonal polynomials and wavelet functions, the results are similar and deferred to Section \ref{sec_additionalofadditionalbasis}. In the left panel of Figure \ref{fig_additional1} below, we show how the value of $\mathrm{d}$ changes with the sample size $n$ under different values of $t_0$ for the null hypothesis in (\ref{eq_h0additional}). We find that even for this complicated case, our method still works well when the sample size is sufficiently large. The near unit root feature seems to have a stronger impact when the sample size is smaller. Similar observations have been made in other settings for forecasting under the near unit root case in \cite{KPF}. Second, we study the power in the right panel of Figure \ref{fig_additional1}. It can be concluded that once the sample size is larger,  our method will be able to obtain high power. 

\begin{figure}[!ht]
\hspace*{-2.0cm}
\begin{subfigure}{0.55\textwidth}
\includegraphics[width=7.8cm,height=4.8cm]{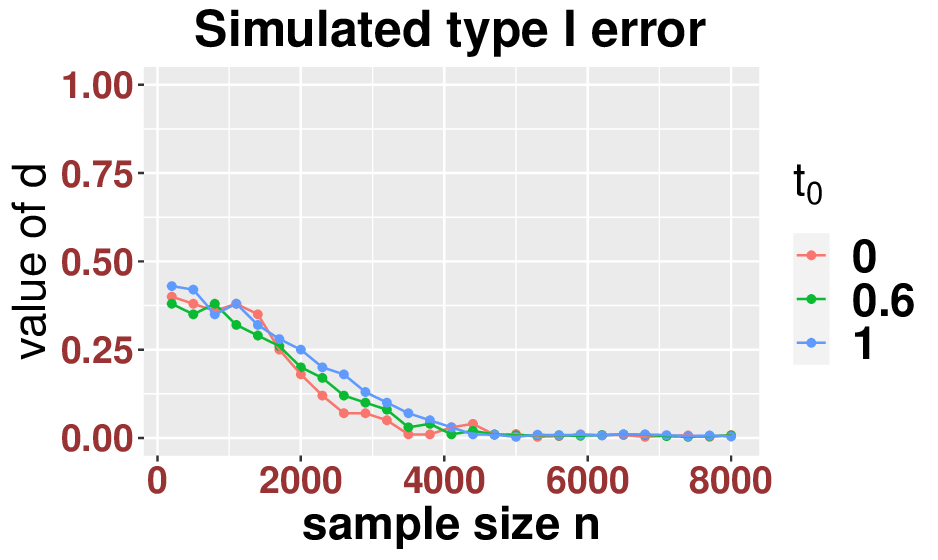}
\end{subfigure}
\begin{subfigure}{0.55\textwidth}
\includegraphics[width=7.8cm,height=4.8cm]{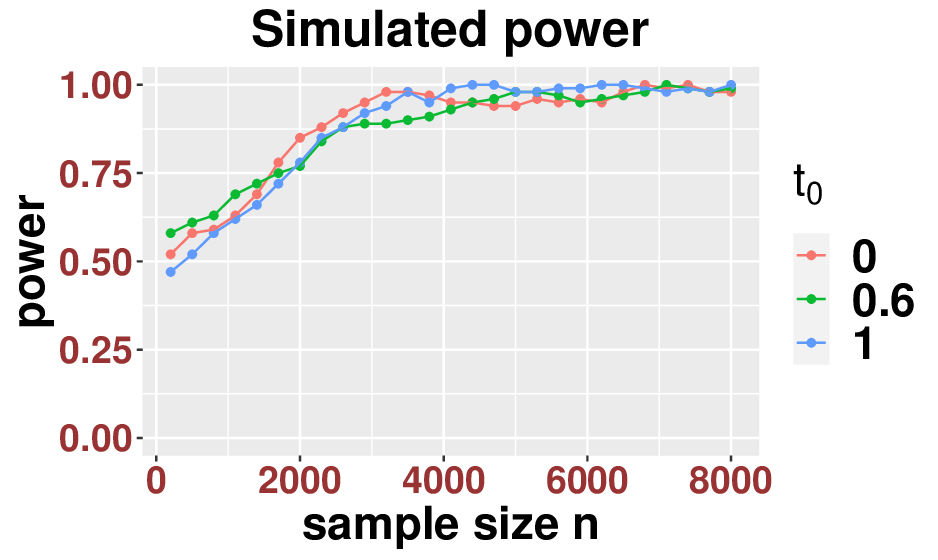}
\end{subfigure}
\caption{ {\footnotesize Simulated type I errors and power.  We consider three different types null hypothesis as in (\ref{eq_h0additional}) for $t_0=0, 0.6, 1,$ respectively. The alternative hypothesis is  (\ref{eq_h0additiona11}). The basis functions are the Fourier bases and the results are reported based on 1,000 simulations. }  }
\label{fig_additional1}
\end{figure}

Finally, to better understand how the sample size influence our type I and power, we provide the receiver operating characteristic curves (ROC) for various values of sample size $n=200, 500, 1000, 2000, 4000, 6000$ in Figure \ref{fig_roc} below.  ROC analysis is
commonly used in medical decision making, and in recent years has been used increasingly
in machine learning and data mining research. In the ROC curve, the $x$-axis is the Type I
error rate and the $y$-axis is the power of statistical test. In practice, researchers will choose the test with the largest area under an ROC curve (AUC). We can see from the simulations that the AUC increases with the sample size $n$, and when the sample size is relatively large, we obtain almost perfect AUC.


\begin{figure}[!ht]
\hspace*{-2.0cm}
\begin{subfigure}{0.32\textwidth}
\includegraphics[width=6.2cm,height=5cm]{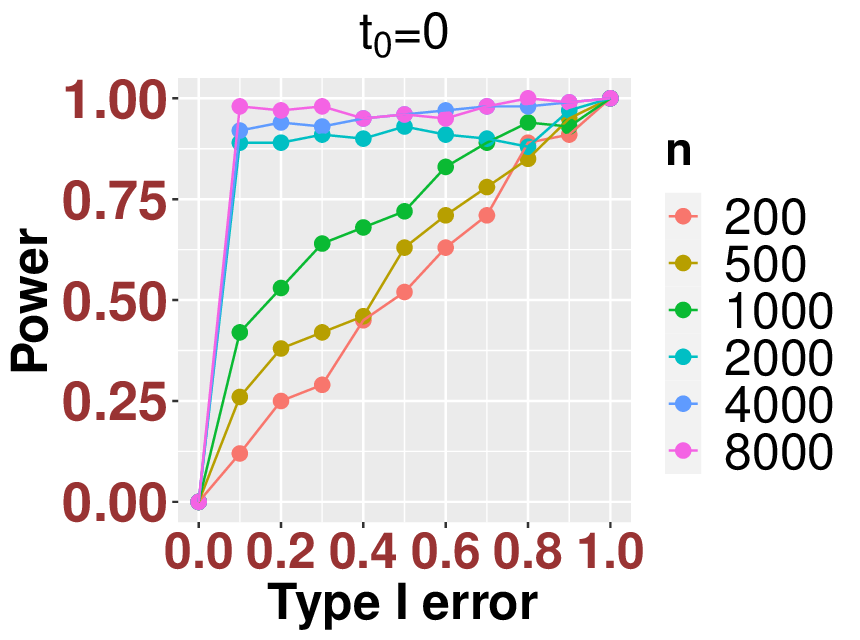}
\end{subfigure}
\begin{subfigure}{0.32\textwidth}
\includegraphics[width=6.2cm,height=5cm]{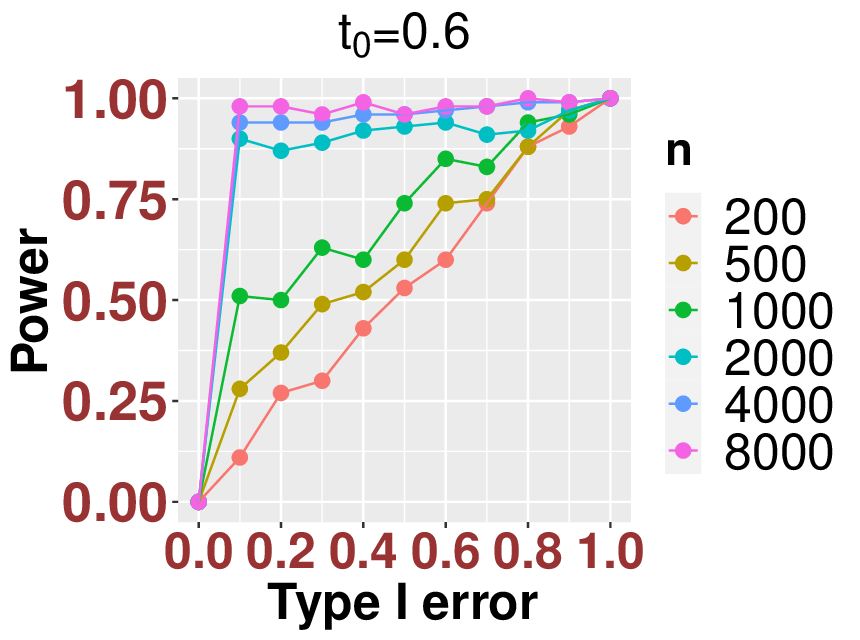}
\end{subfigure}
\begin{subfigure}{0.35\textwidth}
\includegraphics[width=6cm,height=5cm]{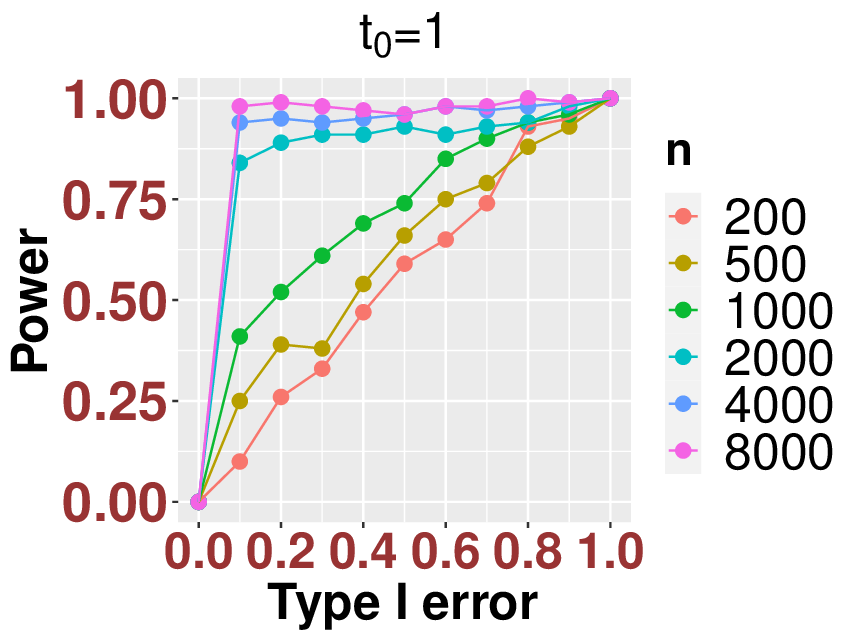}
\end{subfigure}
\caption{ {\footnotesize ROC curves for different values of $t_0.$ The setup is the same as in the caption of Figure \ref{fig_additional1}. } }
\label{fig_roc}
\end{figure}

\subsection{Additional real data analysis: global temperature data}\label{sec_suppl_anotherrealdata} In this subsection, we consider another real data analysis. We study the global temperature time series using the dataset Global component of Climate at a Glance (GCAG). As explained on the website of National Oceanic and Atmospheric Administration (NOAA) \footnote{\url{https://www.ncdc.noaa.gov/cag/global/data-info}}, GCAG comes from the Global Historical Climatology Network-Monthly (GHCN-M) Data Set and International Comprehensive Ocean-Atmosphere Data Set (ICOADS), which have data from 1880 to the present. These two datasets are blended into a single product to produce the combined global land and ocean temperature anomalies.
The term \textit{temperature anomaly} means a departure from a reference value or long-term average. 

 The available time series of global-scale temperature anomalies are calculated with respect to the 20th century average \cite{SRPL}, while the mapping tool displays global-scale temperature anomalies with respect to the 1981-2016 based period. ( see \url{https://datahub.io/core/global-temp#readme} for the dataset). This dataset is a global-scale climate diagnostic tool and provides a big picture overview of average global temperatures compared to a reference value.

\begin{figure}[H]
\centering
\includegraphics[width=10cm,height=6.5cm]{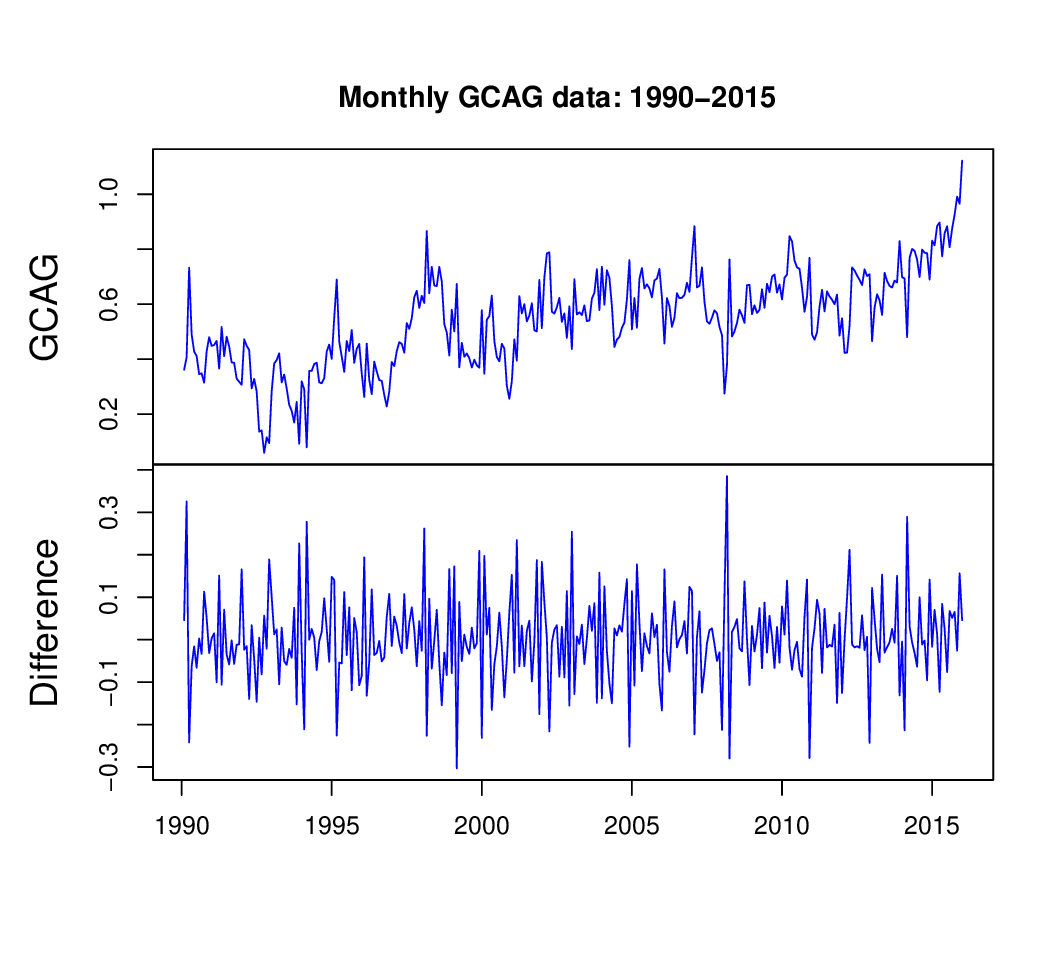}
\caption{Monthly (1990-2015) global temperature using data set GCAG. }
\label{annal_timeseries}
\end{figure}
We study the monthly time series from this dataset for the time period  1990-2015 (Figure \ref{annal_timeseries}). 
As indicated from the above figure, the  global temperature has an increasing trend and we consider its first order difference.


{
Then we apply the methodologies described in Sections \ref{sec:test} and \ref{sec_application} to study the time series.  We first employ the methods from Section \ref{sec:test} to test whether this time series is correlation stationary. For the sieve basis functions, we use the orthogonal wavelets constructed by (\ref{eq_meyerorthogonal}) with Daubechies-9 wavelet. The tuning parameters $b,c$ and $m$ are chosen according to Section \ref{sec:choiceparameter} which yields $b=5$, $J_n=3$ (i.e. $c=8$) and $m=10$. We apply the bootstrap procedure described in the end of Section \ref{sec_bootstrapping} to test the stationarity of the correlation and find that the $p$-value is $0.026$. We hence conclude that the prediction is unstable during this time period. 

}

{
Next, we use time series 1990-2015  as the training dataset to study the (rolling) forecasting performance over the year 2016, i.e., we do a one-step ahead prediction for each month of 2016 in a rolling manner and take the average of the square error.  We use the data-driven approach as described in Section \ref{sec:choiceparameter} to choose $b=6$ and $J_n=3.$ The MSE of our Sieve prediction 
is  $0.381.$  We compare this result with the methods mentioned in  Section \ref{sec_suppl_forecastingexamination} and record the results in Table \ref{table:tem}. We find that our prediction performs better than the other methods. Especially, we get a $16.6 \%$ improvement compared to simply fitting a stationary model using all the time series from 1990 to 2015 (SBLP).    

\begin{table}[H]
\center{
\begin{threeparttable}
\begin{tabular}{ccccccc}
\hline
Method & {\bf Sieve} & TTVAR & LSW & SNSTS & SBLP & PBLP \\ \hline
MSE  &  {\bf 0.381}       &    0.3913        &   0.3851            &      0.3969        &  0.45706           & 0.483  \\ 
\hline
\end{tabular}
\end{threeparttable}}
\caption{Comparison of prediction accuracy for GCAG. For SBLP,  we use all the time series from 1990 to 2015 to fit a stationary ARMA model. For PBLP, we used the most recent $b=6$ samples.
}\label{table:tem}
\end{table}
}
}

\section{A few further discussions, remarks and generalizations}\label{sec_appendix_one}

\subsection{Some assumptions and discussions}
First, we will need the following further assumption in the paper. 

\begin{assu}\label{assu_basis} We assume that the following assumptions hold true for the sieve basis functions and parameters: \\
(1). For any $k=1,2,\cdots, b_*, $ denote $\Sigma^k(t) \in \mathbb{R}^{k \times k}$ whose $(i,j)$-th entry is  $\Sigma^k_{ij}(t)=\gamma(t, |i-j|)$ where $\gamma(\cdot,\cdot)$ is defined in (\ref{eq_defncov}), we assume that the eigenvalues of 
$$ \int_0^1 \Sigma^k(t) \otimes \left( \mathbf{B}(t) \mathbf{B}^*(t) \right),$$
are bounded above and also away from zero by a universal constant $\kappa>0$. { Moreover, we assume that the derivatives of $\gamma(t,j)$ decay with $j$ as follows
\begin{equation*}
\sup_{t \in [0,1]} \sum_{j=0}^{\infty}|\gamma^{(d)}(t,j)|<\infty,
\end{equation*}
where $\gamma^{(d)}(t,j)$ is the d$th$ derivative of $\gamma(t,j)$ with respect to $t$. 
}
\\
(2).  There exist constants $\omega_1, \omega_2 \geq 0,$ for some constant $C>0,$ we have 
\begin{equation*}
\sup_t | \nabla \mathbf{B}(t) | \leq C n^{\omega_1} c^{\omega_2}.
\end{equation*}
(3). We assume that for $\tau$ defined in (\ref{eq_physcialbounbounbound}), $d$ defined in Assumption \ref{assu_smoothtrend} and $\mathfrak{a}$ defined in (\ref{eq_defnc}), there exists a large constant $C>2,$ such that
\begin{equation*}
\frac{C}{\tau}+\mathfrak{a}<1 \ \text{and} \ d \mathfrak{a}>2.
\end{equation*}
\end{assu}
{We mention that the above assumptions are mild and easy to check. First, { the first part of } (1) of Assumption \ref{assu_basis} guarantees the invertibility of the design matrix $Y$ and the existence of the OLS solution. It can be easily verified, for example, for the linear non-stationary process (\ref{ex_linear}), where { the first part of} (1) will be satisfied if $\sup_t \sum |a_j(t)|<1.$ { For the second part of (1), since $\phi_j(t)$ are defined via (\ref{eq_phismoothdefinition}), together with general Leibniz rule and implicit differentiation, it guarantees that the $d$th derivative of $\phi_j(t)$ is bounded (uniformly in $n$) so that the sieve approximation (\ref{eq_phiform}) always holds true. For the linear process (\ref{ex_linear}), the condition is satisfied once (\ref{eq_coeffdecay}) holds.} 
Second, (2) is a mild regularity condition on the sieve basis functions and is satisfied by many commonly  used basis functions. We refer the readers to \cite[Assumption 4]{CC} for further details. Finally, (3)  can be easily satisfied by choosing $C<\tau$
and $\mathfrak{a}$ accordingly. When the physical dependence is of exponential decay,
we only need  $d \mathfrak{a}>2.$ We refer the readers to \cite[Assumption 3.5]{DZ1} for more details. 
}

Second, the accuracy of the multiplier bootstrap in Section \ref{sec_bootstrapping} is determined by the closeness of its conditional covariance structure to that of $\Omega$. Following  \cite[Section 4.1.1]{ZZ1}, we shall use
\begin{equation}\label{eq_omegadefn}
\mathcal{L}(m)=\left| \left|\widehat{\Omega}-\Omega\right| \right|,
\end{equation}
where $\widehat{\Omega}$ is defined in (\ref{eq_widehatomega}), to quantify the latter closeness. The following theorem establishes the bound for $\mathcal{L}(m)$. Its proof  follows from (\ref{eq_lamdaomega}) below and the assumptions of Theorem \ref{thm_bootstrapping}. We omit the details here. 

\begin{thm}[Optimal choice of $m$] \label{thm_choicem} Under the assumptions of Theorem \ref{thm_bootstrapping}, we have 
\begin{equation*}
\mathcal{L}(m)=O_{\mathbb{P}} \left( b_* \zeta_c^2 \Big( \sqrt{\frac{m}{n}}+\frac{1}{m} \Big)  \right).
\end{equation*}
Consequently, the optimal choice of $m$ is of the order $O(n^{1/3}).$
\end{thm}
Note that compared to \cite[Theorem 4]{ZZ1}, the difference from Theorem \ref{thm_choicem} is that we get an extra factor $b_* \zeta_c^2$ due to the high dimensionality. For instance, when we use the Fourier basis, normalized Chebyshev orthogonal polynomials and orthogonal wavelets, we shall have that $b_* \zeta_c^2=p,$ which is the dimension of  $\bm{z}_i$ defined in (\ref{eq_zikronecker}). However, it will not influence the optimal choice of $m$. 

{ Finally, we show that $\{\bm{h}_i\}$ defined in (\ref{eq_zikronecker}) can be expressed using a physical representation and its physical dependence decays polynomially.  Recall that we assume $x_i$ has a physical representation as in (\ref{defn_model}) that $x_i=G(i/n, \mathcal{F}_i).$ 
\begin{lem}\label{lem_lemporductlocally}
Suppose Assumptions \ref{assu_pdc},  \ref{assu_smoothtrend}, \ref{assum_local} and \ref{assu_basis} hold true. Moreover, we assume that the physical dependence measure $\delta(j,q), q>2,$ in (\ref{eq_phygeneral1}) satisfies 
\begin{equation}\label{eq_physcialbounbounbound11}
\delta(j,q) \leq Cj^{-\tau}, \ j \geq 1, 
\end{equation}
for some constant $C>0$ and  $\tau>1$. Then for $\bm{h}_i=\bm{x}_i \epsilon_i \in \mathbb{R}^{b_*},$ we can find some measurable function $\mathbf{U}(\cdot,\cdot)=(u_1(\cdot,\cdot), \cdots, u_{b_*}(\cdot,\cdot))$ so that $\bm{h}_i$ admits a physical representation 
\begin{equation*}
\bm{h}_i=\mathbf{U}\left( \frac{i}{n}, \mathcal{F}_i \right).
\end{equation*}
Moreover, denote $\{\delta_{u_k}(j,q), 1 \leq k \leq b_*\}$ as the physical dependence measures of $\{u_k(\cdot, \cdot), 1\leq k \leq b_*\}.$ Then we have that
\begin{equation*}
\max_{1 \leq k \leq b_*} \delta_{u_k}(j,q) \leq C j^{-\tau},
\end{equation*}
for some universal constant $C>0.$ 
\end{lem}
\begin{proof}
For the first part of the result, according to \cite[Lemma 2.9]{DZ1}, $\{\epsilon_i\}$ admits a physical representation that
\begin{equation*}
\epsilon_i=H(i/n, \mathcal{F}_i),
\end{equation*}
and its physical dependence measure satisfies that 
\begin{equation}\label{eq_epsilonboundboundbound}
\delta_{\epsilon}(j, q) \leq Cj^{-\tau},
\end{equation}
for some universal constant $C>0.$ Recall that $x_i=G(i/n, \mathcal{F}_i).$ Therefore, for $1 \leq k \leq b_*,$ we can set $u_k(i/n, \mathcal{F}_i)=H(i/n, \mathcal{F}_i) G((i+1-k)/n, \mathcal{F}_{i+1-k}).$ This concludes the first part of the proof. 

For the second part of the proof, without loss of generality, we focus on $\delta_{u_1}.$ Note that 
\begin{equation*}
\delta_{u_1}(j,q)=\sup_{t \in [0,1]}\| H(t,\mathcal{F}_{0})G(t, \mathcal{F}_{0})-H(t,\mathcal{F}_{0,j})G(t, \mathcal{F}_{0,j}) \|_q
\end{equation*}
Then it is easy to see that the second part follows from the (\ref{eq_epsilonboundboundbound}) and the assumption (\ref{eq_physcialbounbounbound11}). 
\end{proof}
}

\subsection{Some remarks}\label{suppl_subsec_remark}
{ Six remarks are in order. }
\subsubsection{}\label{suppl_subsec_remark1}
{ First, we provide some remarks on the connection on the AR coefficients in equation (\ref{eq_arapproxiamtionnoncenter}) and the Cholesky decomposition following Section 2 of \cite{kang2020variable}. Without loss of generality, we consider the centered time series so that $\phi_{i0}=\mathbb{E} x_i \equiv 0,$ and for the general setting we refer the readers to \cite{pourahmadi1999joint}. We now write the time series $\{x_i\}_{i=1}^n$ into a column vector that $\bm{x}=(x_1, x_2, \cdots, x_n)^* \in \mathbb{R}^n$ and denote its covariance matrix as $\Sigma \in \mathbb{R}^{n \times n}.$ Consider (\ref{eq_arapproxiamtionnoncenter}) with the convention that $\epsilon_1 \equiv x_1$ and denote $D$ as an $n \times n$ diagonal matrix containing the variances of $\{\epsilon_i\}.$ By setting the $n \times n$ matrix $A$ as follows 
\begin{equation*}
A:=
\begin{pmatrix}
0 & 0 & 0 & \cdots & 0\\
\phi_{21} & 0 & 0& \cdots & 0\\
\phi_{31} & \phi_{32} & 0 &\cdots & 0\\ 
\vdots & \vdots & \ddots & \vdots & \vdots \\
\phi_{n1} & \phi_{n2} &\cdots & \phi_{n.n-1} & 0 
\end{pmatrix},
\end{equation*}
we can write $\bm{\epsilon}=(\epsilon_1, \cdots, \epsilon_n)^* \in \mathbb{R}^n $ as 
\begin{equation*}
\bm{\epsilon}=(I-A)\bm{x}. 
\end{equation*}
This leads to that 
\begin{equation*}
D=(I-A) \Sigma (I-A)^*.
\end{equation*}
Consequently, the above repressions reduces the challenge of modeling a covariance matrix or precision matrix into dealing with regression problems \cite{DZ1, kang2020variable, pourahmadi1999joint}.   
  }
  
  \subsubsection{}\label{suppl_subsec_remark2}
  { Second, we now explain that testing the AR coefficients is easier than testing the correlations directly. {On the one hand, in the literature, there exist some works on  testing covariance stationarity of a time series using techniques from the spectral domain. See, for instance,  \cite{DPV,DR, GN, EP2010}. Nevertheless, testing covariance stationary and correlation stationary are two different problems. Specifically, in order to adopt their methods for the purpose of testing correlation stationarity we observe that the time-varying marginal variance has to be estimated and removed from the time series first. However, it is generally unknown whether the errors introduced in such estimation would influence the finite sample and asymptotic behaviour of the tests. Furthermore, estimating the marginal variance usually involves the difficult choice of a smoothing parameter. One major advantage of our test when used as a test of correlation stationarity is that it is totally free from the marginal variance as the latter quantity is absorbed into the errors of the AR approximation and hence is independent of the AR approximation coefficients. Therefore the test statistic is much more robust to time-varying marginal variances than the existing methods. On the other hand, in another relevant work  \cite{zhao2015inference}, the author infers the autocorrelation function from the time-domain utilizing $\rho(t,j)=\gamma(t,j)/\gamma(t,0).$ Again, the inference procedure therein rely on estimating and inferring the marginal variance function $\gamma(t,0)$ which is a complicated task and could lead to numerical instability in finite samples .

}
\subsubsection{}\label{suppl_subsec_remark3}
{
Third, we make a remark on the connection between the two definitions of locally stationarity, i.e., Definition \ref{defn_locallystationary} and equation (\ref{eq_definition11111}). In fact, our required 
Definition \ref{defn_locallystationary} is a little bit more general in the sense that it requires relatively weaker assumptions compared to (\ref{eq_definition11111}).  In other words, equation (\ref{eq_definition11111}) implies Definition \ref{defn_locallystationary} in the sense that  if a locally stationary time series follows (\ref{eq_definition11111}), it will satisfy Definition \ref{defn_locallystationary}. Recall $h_i(u)$ from (\ref{eq_definition11111}) and denote $\gamma(u,\cdot)$ as the autocovariance function of $\{h_i(u)\}.$ Then we can write that 
\begin{equation}\label{eq_gammuij}
\gamma(u,|i-j|)=\operatorname{Cov}(h_i(u), h_j(u)).
\end{equation}
{\color{red} Moreover}, we have that
\begin{align}\label{eq_covarianceij}
\operatorname{Cov}(x_i, x_j) & =\operatorname{Cov}(h_i(t_i)+x_i-h_i(t_i), h_j(t_i)+x_j-h_j(t_i)) \notag \\
&=\operatorname{Cov}(t_i, |i-j|)+\sum_{k=1}^3 \mathrm{P}_k,
\end{align}
where $\mathrm{P}_k's$ are defined as 
\begin{equation*}
\mathrm{P}_1=\operatorname{Cov}(h_i(t_i), x_j-h_j(t_i)),
\end{equation*}
\begin{equation*}
\mathrm{P}_2=\operatorname{Cov}(x_i-h_i(t_i),h_j(t_i)),
\end{equation*}
\begin{equation*}
\mathrm{P}_3=\operatorname{Cov}(x_i-h_i(t_i), x_j-h_j(t_i)). 
\end{equation*}
It suffices to control $\mathrm{P}_k, k=1,2,3.$

 Using (\ref{eq_definition11111}) with  $u=t_i,$ we readily see that 
\begin{equation*}
|x_i-h_i(t_i)| \leq  \frac{U_i(t_i)}{n}, \ \text{and} \ |x_j-h_j(t_i)| \leq \left( \frac{|i-j|+1}{n} \right) U_j(t_i)   \ a.s.  
\end{equation*}
Consequently, we find that 
\begin{equation}\label{eq_errorbound}
\mathrm{P}_1=O\left( \frac{|i-j|+1}{n} \right), \ \mathrm{P}_2=O\left( \frac{1}{n} \right), \ \mathrm{P}_3=O\left( \frac{|i-j|+1}{n^2} \right).
\end{equation}
Combining (\ref{eq_gammuij}) (with $u=t_i$), (\ref{eq_covarianceij}) and (\ref{eq_errorbound}), we see that (\ref{eq_definition11111}) implies Definition \ref{defn_locallystationary}. }

\subsubsection{}\label{suppl_subsec_remark4}
{
Finally, we provide a remark on the local power in light of the work \cite{paparoditis2016local}.
 Before proceeding to compare the local power properties under different settings, we first summarize the results of \cite{paparoditis2016local}. In the aforementioned paper, the authors investigate some local power properties of frequency domain-based tests for stationarity for a general class of local alternatives. In particular, the authors consider the time-varying linear Gaussian process 
\begin{equation}\label{eq_timeseriesprocess}
X_{i,n}=\sum_{j=-\infty}^{\infty} a_{n}(i/n,j) \epsilon_{t-j},
\end{equation}
where $\{\epsilon_i\}$ are i.i.d. standard Gaussian random variables, and the sequence of functions $a_n(\cdot,l)$ are
twice continuously differentiable and satisfy certain decay and regularity conditions (see equation (3.2) of \cite{paparoditis2016local}). Let $f(u, \lambda)$ be the local spectral density of $X_{t,n}$ in (\ref{eq_timeseriesprocess}). The authors are interested in understanding the following hypothesis testing problem
\begin{align*}
& \mathbf{H}_0: \ f(u, \lambda)=g(\lambda), \ \text{a.e.}, \ g(\lambda)=\int_0^1 f(u,\lambda) \mathrm{d} u, \\
& \mathbf{H}_a: \ f(u,\lambda) \neq g(\lambda) \ \text{on a set} \ A \subset [0,1] \ \text{with positive Lebesgue measure}. 
\end{align*}  
Note that under $\mathbf{H}_0,$ $X_{t,n}$ is a stationary time series under some minor regularity conditions. Since $\mathbf{H}_a$ is too general, in order to study the local power properties, the authors consider two specific local alternatives in terms of the coefficients of $X_{t,n}.$  The first kind belongs to Pitman-type and is global in time that 
\begin{equation}\label{eq_firstalter}
\mathbf{H}_{a1}: \ a_n(u, l)=a_0(l)(1+c_n b(u,l)), \ n \in \mathbb{N},
\end{equation}  
where $c_n=n^{-\kappa}$ for some $\kappa>0$ and $b(\cdot,l): [0,1] \rightarrow \mathbb{R}$ are twice continuously differentiable functions satisfying some regularity conditions. The second kind is time localized that  
\begin{equation}\label{eq_secondalter}
\mathbf{H}_{a2}: \ a_n(u,l)=a_0(l)(1+c_n b((u-u_0)/\gamma_n,l)),
\end{equation}
where $u_0 \in (0,1)$ is some fixed time point, $c_n=n^{-\kappa}$ and $\gamma_n=n^{-\zeta}$ for some $\kappa, \zeta>0,$ and $b(\cdot, l)$ is a function satisfying some regularity conditions. Notice that under $\mathbf{H}_{a2},$ when $n$ is sufficiently large, it will become more concentrated around the time point $u_0.$ Then in \cite{paparoditis2016local}, the authors considered three different frequency-domain based tests as in \cite{DPV,paparoditis2009testing,preuss2013test}  and developed a framework for investing the local power properties of the three methods.

With the above discussion, we first point out several major differences in the setup between \cite{paparoditis2016local} and our current paper. First, in \cite{paparoditis2016local}, the authors are concerned with the frequency-domain methods whereas ours is a time-domain based approach. Second, in \cite{paparoditis2016local}, the authors focus on testing the covariance stationarity of the time series while our current paper studies the correlation stationarity. Note that many time series are correlation stationarity but not covariance stationary, for example, the time-varying white noise process. Third, in \cite{paparoditis2016local}, probably for technical simplicity, the temporal decay condition (i.e., (3.2) of \cite{paparoditis2016local}) is chosen in a fixed manner that for some fixed constant $C>0$
\begin{equation}\label{eq_decayregime}
\sup_{u \in [0,1]} \sum_j|j||a_n(u,j)| \leq C. 
\end{equation} 
All the results are then derived based on this assumption, especially the transition in terms of $\kappa$ and $\zeta$; see Theorems 3.1-3.6 in \cite{paparoditis2016local}. In contrast, in our paper, our temporal decay is more general and flexible. Particularly, in terms of (\ref{eq_timeseriesprocess}), we assume that 
\begin{equation*}
\sup_{u \in [0,1]} |a_n(u,j)| \leq C j^{-\tau}.
\end{equation*}
Consequently, our results in Proposition \ref{prop_power} are adaptive to the temporal decay in terms of the parameter $\tau.$ (Note that Proposition \ref{prop_power} is adaptive to $b_*$ and $b_*$ is essentially related to $\tau$ via equation (\ref{eq_choiceofb}).) Our Proposition \ref{prop_power} demonstrates that when $\tau$ is larger, i.e., the temporal decay is faster, our test can achieve asymptotic power one for weaker alternatives. Since \cite{paparoditis2016local} focuses on a specific regime (\ref{eq_decayregime}), when the temporal decays faster, the results derived in Sections 3.2 and 3.3 of \cite{paparoditis2016local} are not necessarily sharp. Fourth, similar arguments apply to the smoothness of the coefficients $a_n(u,j).$ In \cite{paparoditis2016local}, the authors only consider two continuously differentiable. However, our results in Proposition  \ref{prop_power} of the revised manuscript is adaptive to the smoothness via the value of $c.$ Especially, equation (\ref{eq_phiform}) of the revised manuscript implies that $c$ can be chosen that $c \asymp n^{1/d}.$ That is to say, when the functions are smoother, only a smaller value of $c$ is required. Therefore, when the coefficients are smoother, the results in \cite{paparoditis2016local} may not be sharp.   Finally, for technical convenience,  in \cite{paparoditis2016local}, a linear Gaussian process is studied but our current approach can be applied to nonlinear and non-Gaussian time series.

We can see that there exist significant differences in the settings and scope between \cite{paparoditis2016local} and our current paper. Even though a direct comparison between our current method and \cite{paparoditis2016local} is unfair in general,  we make the following efforts to compare and unify the understanding of tests of stationarity using (some examples of) the linear time-varying process.  For simplicity and to avoid distraction, we focus on the global in time local alternatives $\mathbf{H}_{a1}.$ Similar arguments and discussions apply to $\mathbf{H}_{a2}$ after necessary modifications.    

First, we consider a time-varying AR(1) process whose causal representation can be written into (\ref{eq_timeseriesprocess}). More specifically, we consider that $x_i \equiv x_{i,n}$
\begin{equation*}
x_i=\phi(i/n) x_{i-1}+\epsilon_i,
\end{equation*} 
where $\{\epsilon_i\}$ is a time-varying white noise process in general. We now consider the alternative that 
\begin{equation}\label{eq_commonalternative}
\mathbf{H}_a': \ \phi(i/n)=\phi(1+c_n b(i/n)),
\end{equation}
for some $c_n=o(1)$ and constant $|\phi|<1$. In what follows, we will show how the local alternatives of (\ref{eq_firstalter}) and (\ref{eq_secondalter}) relate and differ to that in our paper in the setting of (\ref{eq_commonalternative}). On the one hand, using (\ref{eq_commonalternative}), the local alternative of our current paper reads as   
\begin{equation}\label{eq_ccccc}
\mathbf{H}_a: \ \int_0^1 (\phi(t)-\bar{\phi})^2 \mathrm{d}t = \phi^2 c_n^2 \int_0^1 (b(t)-\int_0^1 b(t))^2 \mathrm{d} t.  
\end{equation}
Assuming that $\int_0^1 (b(t)-\int_0^1 b(t))^2 \mathrm{d} t$ is bounded from below, for example, $b(t)=\cos(2 \pi t),$ then our Proposition \ref{prop_power} holds once $c_n^2>C \sqrt{b_* c}/n$ for some large constant $C>0.$ Especially in the setting of \cite{paparoditis2016local} with $c_n=n^{-\kappa},$ the phase transition happens when $n^{-2\kappa} >C \sqrt{b_*c}/n.$ Considering the setting of \cite{paparoditis2016local} that we shall choose $b_* \asymp n^{1/2}$ and $c \asymp n^{1/2},$ we find that the phase transition happens when $\kappa=1/4.$  Moreover, when the temporal relation decays faster and the functions are infinitely differentiable, we can choose $b_*, c \asymp \log n$  so that the phase transition happens when $\kappa=1/2$. In general, it is adaptive to the temporal decay rate and the smoothness of the functions.  On the other hand,  as $\sup_t |\phi(t)|<1$, using the method of matching coefficients, we readily see that $x_i$ can be rewritten as 
\begin{equation}\label{eq_ar1causal}
x_i=\sum_{j=0}^{\infty} \phi(i/n)^j \epsilon_{i-j}.
\end{equation} 
Note that (\ref{eq_ar1causal}) is an example of (\ref{eq_timeseriesprocess}) by setting $a_n(i/n,j) \equiv 0, j<0$ and $a_n(i/n,j)=\phi(i/n)^j.$
Then in terms of the alternative (\ref{eq_commonalternative}),  $\mathbf{H}_{a1}$ in (\ref{eq_firstalter}) reads as 
\begin{equation}\label{eq_transferassumption}
\mathbf{H}_{a1}: a(i/n,j)=\phi^j(1+c_n b(i/n))^j.  
\end{equation} 
Heuristically, when $j$ becomes large, $\phi^j$ will be negligible so that we only need to focus on small values of $j \leq C \log n.$ In this setting, as $c_n \asymp n^{-\kappa},$ we have that when $n$ is sufficiently large, $(1+c_n b(i/n))^j \asymp 1+j c_n b(i/n).$ Consequently, we can approximate $\mathbf{H}_{a1}$ using
\begin{equation*}
\mathbf{H}_{a1}': a(i/n,j)=\phi^j(1+j c_n b(i/n)).
\end{equation*}
Therefore, the results of Section 3.2 of \cite{paparoditis2016local} apply. Especially, we notice that the transition of Theorem 3.2 therein is $\kappa=1/4$ and Theorem 3.3 is $\kappa=1/2$. For the transition in Theorem 3.1, when our $b_*$ and $c$ take some specific values, we can also match the transition. { In summary, under the time-varying AR(1) model which is a special case of the setting  in \cite{paparoditis2016local}, we find our time-domain method which is used to test correlation stationarity has the same phase transition as the frequency method proposed by \cite{DPV} which is used to test covariance stationarity. Especially, when $\{\epsilon_i\}$ are stationary white noise as in \cite{paparoditis2016local}, these two tests are asymptotically equivalent in terms of power. Moreover, if the temporal decays faster and the functions are smoother, our method has better performance and is asymptotically equivalent to the frequency based test in \cite{preuss2013test}.}


Second, we consider a time-varying MA(1) process which is also a special case of (\ref{eq_timeseriesprocess}). We point out that it is meaningless to compare the MA(0) process which is the time-varying white noise. The reason is because regardless of whether it is covariance stationary, it is always correlation stationarity which is our null hypothesis. For the MA(1), for simplicity, we consider that
\begin{equation*}
x_i=\theta(i/n)\epsilon_{i-1}+\epsilon_i, 
\end{equation*}
where $\epsilon_i$ are white noise process and $\sup_t |\theta(t)|<1.$ In this setting, we consider the local alternative $\mathbf{H}_{a1}$ as in (\ref{eq_firstalter}) that
\begin{equation}\label{eq_ma1ha1}
\mathbf{H}_{a1}: \theta(i/n)=\theta(1+c_nb(i/n)),
\end{equation}
for some $c_n=o(1)$ and constant $|\theta|<1.$ On the one hand, under such a setting, the results in Section 3.2 of \cite{paparoditis2016local} apply. On the other hand, as $\sup_t |\theta(t)|<1,$ using the method of matching coefficients, we find that $x_i$ can be rewritten as
\begin{equation*}
x_i=\sum_{j=1}^{\infty} \left(-\theta\left(\frac{i}{n}\right)\right)^j x_{i-j}+\epsilon_i.
\end{equation*} 
The above form is consistent with the AR approximation used in our paper. By a discussion similar to the paragraph below (\ref{eq_transferassumption}), in terms of (\ref{eq_ma1ha1}), it suffices to consider $j \leq C \log n$  so that we have $\left(-\theta\left(\frac{i}{n}\right)\right)^j \approx (-\theta)^j(1+c_n j b(i/n)).$ In other words, $b_* \asymp \log n.$ Consequently, in terms of (\ref{eq_ma1ha1}),  we can approximate our alternative similar to (\ref{eq_ccccc}) using
\begin{equation*}
 \sum_{j=1}^{\infty} \theta^{2j} j^2  c_n^2 \int_0^1 (b(t)-\int_0^1 b(t) \mathrm{d} t)^2 \mathrm{d} t.
\end{equation*}     
Assuming that $\int_0^1 (b(t)-\int_0^1 b(t))^2 \mathrm{d} t$ is bounded from below, for example, $b(t)=\cos(2 \pi t),$ using the setup $c \asymp n^{-\kappa},$ our local alternative can be approximated as  
\begin{equation*}
\mathbf{H}_{a}':  c_n^2 >C \sqrt{c\log n }/n.
\end{equation*}
Assuming that the functions are twice continuously differentiable that $c \asymp n^{1/2}$, we see that the phase transition happens when $\kappa=3/8.$ This matches the result of Theorem 3.1 of \cite{paparoditis2016local} for some properly chosen $\delta$ and $\rho$ (see the definitions therein). Moreover, when the functions are infinitely differentiable, we have $c \asymp \log n$ so that the phase transition happens $\kappa=1/2$ which matches the results of Theorem 3.3. When the smoothness of the functions varies, we can also match the results of Theorem 3.2. { In summary, under the time-varying MA(1) model which is a special case of the setting in \cite{paparoditis2016local}, we find our time-domain method which is used to test correlation stationarity has the same phase transition as the frequency method proposed by \cite{paparoditis2009testing} which is used to test covariance stationarity, provided the parameters in \cite{paparoditis2009testing} are properly chosen. Especially, when $\{\epsilon_i\}$ are stationary white noise as in \cite{paparoditis2016local}, these two tests are asymptotically equivalent in terms of power. Moreover, if the temporal decays faster and the functions are smoother, our method has better performance and is asymptotically equivalent to the frequency based test in \cite{preuss2013test}.} 

Based on the analysis of the above two examples, we find that different methods have their own advantage in testing stationarity. In the setting when the marginal variance is stationary (i.e., correlation stationarity is the same as covariance stationarity), under the exact assumption of \cite{paparoditis2016local}, for the time-varying AR(1) model, our proposed method has the same local power properties as \cite{DPV}; and for the time-varying MA(1),  our proposed method has the same local power properties as \cite{paparoditis2009testing} provided the parameters are properly chosen. Moreover, if the functions are more smooth, our method will have the same power properties as \cite{preuss2013test}. Nevertheless, we want to emphasize again, our method is proposed for correlation stationarity testing. That is to say, even when the marginal variance is time-varying (which is clearly covariance non-stationary), we can have correlation stationary. Finally, for general time-varying AR($p$) model and MA($q$) (even MA$(\infty)),$ we can apply the same ideas here to study the comparison. The computation will be more tedious by working on the causal representation of the AR model and invertible representation of the MA model. To     to fully understand the similarity and differences, we need to work on another paper and this will be the future work.

}
}

{
\subsubsection{}\label{secsec_arch} In this subsection, we discuss the locally stationary ARCH model considered in \cite{dahlhaus2006statistical}  where the authors considered the following model. For some smooth functions $a_i(t), 0 \leq i \leq p: [0,1] \rightarrow \mathbb{R}^+,$ they consider the time series $x_i \equiv x_{i,n}, 1 \leq i \leq n,$ that  
\begin{equation}\label{eq_locallystationarymodel}
x_i=\left(a_0 \left(i/n \right)+a_1(i/n)x^2_{i-1}+\cdots+a_p(i/n)x^2_{i-p} \right)^{1/2}\epsilon_i,
\end{equation}     
where $\{\epsilon_i\}$ are i.i.d. random variables with $\mathbb{E} \epsilon_i=0$ and $\mathbb{E} \epsilon_i^2=1.$  
{In terms of the squared time series $x_i^2$, \eqref{eq_locallystationarymodel} leads to 
\begin{equation}\label{eq_locallystationarymodel11}
x^2_i=\left( a_0 \left(i/n \right)+a_1(i/n)x^2_{i-1}+\cdots+a_p(i/n)x^2_{i-p} \right)\epsilon_i^2.
\end{equation}    

Using a slight extension of the proofs of the results in Section 2.4.2 of Mayer, Z{\" a}hle and Zhou (2020) \cite{mayer2020functional}, it can be shown that under some mild conditions,
\begin{equation*}
x_i=G(i/n,{\cal F}_i)+O_{\mathbb{P}}(1/n),
\end{equation*}
where $G(i/n,{\cal F}_i)$ is a locally stationary time series as defined in Example \ref{eexample_linear}. Furthermore, $G(i/n,{\cal F}_i)$ is a white noise process, i.e., a time-varying MA(0) process. Therefore, if the driving noise process $\{z_t\}$ in (\ref{ex_linear}) is a white noise process, our expression (\ref{ex_linear}) covers the time-varying ARCH model in  \cite{dahlhaus2006statistical} (approximately). As a side note, if the driving noise process $\{z_t\}$ has to be i.i.d., then expression (\ref{ex_linear}) does not cover the time-varying ARCH model in general.

On the other hand, for the squared time series (\ref{eq_locallystationarymodel11}), denote $y_i=x_i^2,$ then we can rewrite it as follows
\begin{equation}\label{eq_locallystationaryrepresentationarch}
y_i=a_0(i/n)+\sum_{j=1}^p a_j(i/n)y_{i-j}+\pi_i,
\end{equation} 
where $\pi_i$ is defined as 
\begin{equation*}
\pi_i=(\epsilon_i^2-1)\left( a_0(i/n)+\sum_{j=1}^p a_j(i/n)y_{i-j}\right). 
\end{equation*}
Recall that $\mathbb{E} \epsilon^2_i=1.$ Following the same argument as above, $y_i=G^2(i/n,{\cal F}_i)+O_{\mathbb{P}}(1/n)$ and $\pi_i$ is (approximately) a locally stationary white noise process.
That is to say, (\ref{eq_locallystationarymodel11}) is approximately a white-noise-driven locally stationary AR($p$) model. Moreover, according to \cite{zhou2013inference}, under mild conditions, (\ref{eq_locallystationaryrepresentationarch}) can be well approximated by a time-varying MA($\infty$) process. Therefore, we can conclude that (\ref{ex_linear}) also covers (\ref{eq_locallystationarymodel11}) asymptotically. In summary, when $n$ is sufficiently large, under suitable condition,  (\ref{ex_linear}) covers the locally stationary ARCH($p$) model in terms of both (\ref{eq_locallystationarymodel}) and (\ref{eq_locallystationarymodel11}).  The rigorous justification is out of the scope of the current paper and we will pursue this direction in the future works. 

 Moreover, in terms of testing,  if the  mechanism of $x_i$ is locally stationary (G)ARCH, then our test $T$ can be used to test the constancy of the ARCH coefficients when it is applied to the squared time series $x_i^2$. In particular,
for (\ref{eq_locallystationarymodel11}), it can be regarded (approximately) as a white-noise-driven locally stationary AR($p$) process where the AR coefficients are exactly the ARCH coefficients of $x_i$. Hence, our statistics $T$ can be utilized to test the constancy of the functions $a_k(\cdot), 0 \leq k \leq p,$ asymptotically, i.e., the stationarity of the ARCH model. Again, the rigorous mathematical treatment of the approximation errors shall be pursued in our future works.

 
}
}

\subsubsection{}\label{rem_collectionofremark}
Two remarks are in order. First, we point out that
as can be seen from the proof of Theorem \ref{thm_bootstrapping}, especially the discussion between (\ref{eq_lamdaomega}) and (\ref{eq_hatUp}), when conditional on the data,  the covariance matrix of $\Phi$ ($\widehat{\Phi}$) can be explicitly computed as in (\ref{eq_defnphi}) which could serve as a plug-in estimator for $\Omega$. People can then estimate the quantities $f_1, f_2$ in Proposition \ref{prop_normal} using those plug-in estimators; see (\ref{eq_festimatorsuppl})  for more details.  Instead, our Algorithm \ref{alg:boostrapping} directly mimic the distribution of $T$ without using the plug-in estimator for the purpose of faster convergence and more accurate finite-sample performance. In Section \ref{sec_numericalplguin}, we use extensive numerical simulations to illustrate the superior empirical performance of our Algorithm \ref{alg:boostrapping} compared to the plug-in approach. Second, we point that that other resampling methods such as the AR sieve bootstrap may also work for the implementation of $T$. We refer to \cite{buhlmann2002bootstraps,kreiss2012bootstrap,kreiss2011} for  reviews of bootstrap methods for time series. Even though some spectral domain bootstrap methods have been developed for locally stationary linear process, for example \cite{kreiss2015bootstrapping}, most of the time domain bootstrap techniques have been developed or justified only for stationary time series. The generalization of these methods to quadratic forms of locally stationary time series is highly nontrivial and will be studied in the future; see Section \ref{sec_suppl_generalization3} for more discussions.          
\subsection{Some generalizations}\label{sec_suppl_generalization}
{ Some arguments on potential generalizations are recorded in order.}
\subsubsection{}\label{sec_suppl_generalization1}
{ First, we make a remark on the $h$-step ahead prediction for $h \leq h_0,$ where $h_0$ is some fixed positive integer. Since $h$ is fixed and $b \equiv b(n)$ diverges with $n$, when $n$ (and $b$) is sufficiently large, for $b+1 \leq i \leq n,$ we consider the $h$-step ahead best linear prediction $\widehat{x}_{i,h}$ of $x_i$ which utilizes all its predecessors up to $x_{i-h},$ i.e., 
\begin{equation*}
\hat{x}_{i,h}=\phi_{i0,h}+\sum_{j=h}^{i-1} \phi_{ij,h} x_{i-j}, \ i=b+1,\cdots, n.  
\end{equation*}
Denote $\epsilon_{i,h}=x_i-\hat{x}_{i,h}.$ Similar to (\ref{eq_arapproxiamtionnoncenter}), we can rewrite 
\begin{equation*}
x_i=\phi_{i0,h}+\sum_{j=h}^{i-1} \phi_{ij,h} x_{i-j}+\epsilon_{i,h}, \ i=b+1,\cdots, n.
\end{equation*}
On the one hand, as $h$ is fixed, we can check that the results of Section \ref{sec:nonzeromeandiscussion} hold for $\{\phi_{ij,h}\}$ after some minor modification. On the other hand, when $\{x_i\}$ is a locally stationary time series satisfying the assumptions of Section \ref{sec:arappoximate}, we can show that the results of Section \ref{sec:arappoximate} still apply with some notational changes. Especially, similar to (\ref{eq_phismoothdefinition}), for the smooth function $\bm{\phi}_h(t)=(\phi_{1,h}(t), \cdots, \phi_{b,h}(t))^* \mathbb{R}^b$ such that 
\begin{equation}\label{eq_yulewalkerh}
\bm{\phi}_h(t)=\Gamma_h(t)^{-1} \bm{\gamma}_h(t),
\end{equation}
where $\Gamma_h(t)$ and $\bm{\gamma}_h(t)$ are defined as follows that whose entries satisfy 
\begin{equation*}
\Gamma_{h,ij}(t)=\gamma(t,|i-j|), \ \bm{\gamma}_i(t)=\gamma(t,i), \ i,j=h+1, h+2, \cdots, h+b,
\end{equation*}
we can show that 
\begin{equation*}
x_i=\phi_{0,h}(i/n)+\sum_{j=h}^b\phi_{j,h}(i/n)x_{i-j}+o_{\ell_2}(1). 
\end{equation*} 
Consequently, for the $h$-step ahead prediction, similar to (\ref{eq_forecast}) for the one-step ahead prediction, for sufficiently large $n,$ we shall use the following linear predictor  
\begin{equation}\label{eq_truecoefficientestimator}
\widehat{x}_{n+h}^b=\phi_{0,h}(1)+\sum_{j=h}^b\phi_{j,h}(1)x_{n+1-j}.
\end{equation} 
Moreover, as $h$ is fixed, we can prove a result similar to Theorem \ref{thm_prediction} and show $\widehat{x}_{n+h}^b$ is an asymptotic linear optimal predictor. 

Next, due to the smoothness of $\{\phi_{j,h}\}_{j=0}^b,$ we can use the method of sieves to estimate them as in (\ref{eq_phiform}), i.e., 
\begin{equation*}
\phi_{j,h}(i/n)=\sum_{k=1}^c a_{jk,h} \alpha_k(i/n)+o(1), \ 0 \leq j \leq b, \ i>b,
\end{equation*}  
where we recall that $\{\alpha_k\}$ are some basis functions on $[0,1]$ and $c$ is the number of basis functions. Consequently, it suffices to estimate the coefficients $a_{jk,h}'s$ using OLS as in (\ref{eq_choeq}). Denote the OLS estimates as $\hat{a}_{jk,h}.$ We can then estimate $\phi_{j,h}(t)$ using
\begin{equation*}
\hat{\phi}_{j,h}(t)=\sum_{k=1}^c \hat{a}_{jk,h} \alpha_k(t). 
\end{equation*}  
In view of (\ref{eq_truecoefficientestimator}), similar to (\ref{eq_forcaestequation}), we can forecast $x_{n+h}$ using 
\begin{equation*}
\widehat{\mathsf{x}}_{n+h}^b=\hat{\phi}_{0,h}(1)+\sum_{j=h}^b \widehat{\phi}_{j,h}(1)x_{n+1-j}. 
\end{equation*}  
In addition, we can study the MSE of the forecast and prove a consistent result similar to Theorem \ref{thm_finalresult}. 

\subsubsection{} \label{sec_suppl_generalization2}
Second, we briefly discuss how to generalize our arguments in Section \ref{sec:test} from locally stationary time series to the piecewise locally stationary time series as introduced in \cite{dette2019change,wu2019multiscale, ZZ1}. Consider a locally stationary time series with possible abrupt changes following \cite{wu2019multiscale} 
\begin{equation}\label{eq_plsmodel}
x_i=\nu(i/n)+y_i,
\end{equation} 
where $\nu(t)$ is a piece-wise smooth function with $\mathrm{p}$ jump points $0<d_1<d_2<\cdots<d_{\mathrm{p}}<1,$ and $\{y_i\}$ is a centered piece-wise locally stationary process \cite{dette2019change, ZZ1} defined as follows.

\begin{defn}[Piece-wise locally stationary processes] Let $\{\eta_i\}_{i \in \mathbb{Z}}$ be a sequence of i.i.d. random variables and $\mathcal{F}_i=(\eta_s, s \leq i).$ The mean-zero sequence $\{y_i\}$ is called piece-wise locally stationary (PLS) with $\mathrm{q}$ abrupt change points if there exist constants $0=c_0<c_1<c_2<\cdots<c_{\mathrm{q}}<c_{\mathrm{q}+1}=1$ and some measurable functions $G_j, 0 \leq j \leq \mathrm{q},$  such that 
\begin{equation*}
y_i=G_j(i/n, \mathcal{F}_i), \ c_j<i/n \leq c_{j+1}, \ 0 \leq j \leq \mathrm{q}, 
\end{equation*}
where 
\begin{equation*}
\| G_j(t, \mathcal{F}_0)-G_j(s, \mathcal{F}_0) \|_\ell \leq C|t-s|,
\end{equation*}
for all $t,s \in (c_j, c_{j+1}],  0 \leq j \leq \mathrm{q},$ for some finite constant $\ell>2$ and some finite constant $C>0.$
\end{defn} 

Note that when $\mathrm{p}=\mathrm{q}=0,$ the model (\ref{eq_plsmodel}) reduces to the locally stationary time series considered in Section \ref{sec:test}. When $\mathrm{p}+\mathrm{q} \neq 0,$ we need to detect these change points before applying our current methodology. We define the abrupt change points time index sets as $\mathcal{C}:=\mathcal{C}_1 \bigcup \mathcal{C}_2$ where   
\begin{equation*}
\mathcal{C}_1:=\{d_1, d_2, \cdots, d_{\mathrm{p}}\}, \ \mathcal{C}_2:=\{c_1, c_2, \cdots, c_{\mathrm{q}}\}. 
\end{equation*}
Without loss of generality, we assume that all the abrupt change points are distinct. For simplicity, we assume that both $\mathrm{p}$ and $\mathrm{q}$ are finite. On the one hand, the detection of the change points $\mathcal{C}_1$ has been studied in \cite{wu2019multiscale, ZZ1} under the model (\ref{eq_plsmodel}).  On the other hand, as discussed in \cite{wu2019multiscale}, under  (\ref{eq_plsmodel}), the two-stage detection method proposed in \cite{wu2019multiscale} can also been applied to detect the change points in $\mathcal{C}_2$. In this regard, we can apply the method of  \cite{wu2019multiscale} to detect the 
change points set $\mathcal{C}=\mathcal{C}_1 \bigcup \mathcal{C}_2.$  For simplicity, we denote
\begin{equation*}
\mathcal{C}:=\{\mathsf{s}_1< \cdots< \mathsf{s}_{\mathrm{p}+\mathrm{q}}\},
\end{equation*}
and set $\mathsf{s}_0=0$ and $\mathsf{s}_{\mathrm{p}+\mathrm{q}+1}=1$ for convenience. Then for $x_i$ restricted in each interval that $i \in (\mathsf{s}_j, \mathsf{s}_{j+1}],  0\leq j \leq \mathrm{p}+\mathrm{q},$ it is a locally stationary time series. Therefore we can apply our inferential theory of Section \ref{sec:test} to each of these intervals. In particular, we will be able to show that $x_i$ can be approximated by a locally stationary AR process on each interval $(\mathsf{s}_j, \mathsf{s}_{j+1}],  0\leq j \leq \mathrm{p}+\mathrm{q}$ so that overall the time series $x_i, 1 \leq i \leq n,$ can be well approximated by a piece-wise locally stationary AR process. Moreover, we can test the constancy of the AR approximation coefficients on each of these intervals. Finally, in terms of forecasting, we can utilize the last interval, i.e., $(\mathsf{s}_{\mathrm{p}+\mathrm{q}},1].$ Since this is not the main focus of the current paper, we will study such a generalization in the future works.

\subsubsection{}\label{sec_suppl_generalization3}
Third, we briefly discuss the ideas and main challenges of locally stationary AR sieve bootstrap. The rigorous justification and development will be left as future works since it is not the main focus of the current paper.

In the literature, AR sieve bootstrap is only developed and fully justified for stationary process \cite{kreiss2011} and has not been modified to fit for locally stationary time series. To study the generalization of stationary AR sieve bootstrap to non-stationary time series will require a rather substantial discussion. In fact, except for Gaussian time series, directly using the pseudo time series generated from the fitted AR($b_*$) model may fail to approximate the distribution of $nT.$ The main reason is because the standard AR sieve bootstrap is only able to replicate the covariance information of the underlying time series. However, according to Proposition \ref{prop_normal}, in order to apply the distribution results, we need to know the long run covariance matrix of a high dimensional locally stationary time series $\{\bm{h}_i\} \in \mathbb{R}^{b_*}$ as defined in (\ref{eq_zikronecker}).  Note that $\bm{h}_i=\bm{x}_i \epsilon_i$ depends on both  the univariate time series and the residual, and is a high dimensional locally stationary time series in general. Since the distribution of $nT$ depends on a quadratic form of $\{\bm{h}_i\},$ it actually relies on the first fourth cumulants of the underlying time series.

In this regard, in order to apply the idea of AR sieve bootstrap for our $\mathcal{L}^2$ test, we conjecture that we shall work with $\{\bm{h}_i\}$ instead of $\{x_i\}.$ In order to establish the AR bootstrap for $\{\bm{h}_i\}$, unlike in the current paper we focus on univariate time series,  we need to establish the AR approximation theory for high dimensional locally stationary time series $\{\bm{h}_i\},$ i.e., for some time-varying matrix coefficients $A_j(t) \in \mathbb{R}^{b_* \times b_*}$ and some slowly divergent value $\mathrm{b}_{**},$ we conjecture that 
\begin{equation}\label{eq_highdimensionalAR}
\bm{h}_i=A_0(i/n)+\sum_{j=1}^{\mathrm{b}_{**}} A_j(i/n) \bm{h}_{i-j}+\bm{e}_i+o(1),
\end{equation}     
where $\{\bm{e}_i\}$ is the white noise vector process. Moreover, the covariance matrix of $\{\bm{e}_i\}$ is smoothly time-varying. Nevertheless, the probabilistic investigation into the above high-dimensional and non-stationary AR approximation is very difficult and we shall actively investigate this line of research in our future endeavours. 


In summary, we believe that the standard AR sieve bootstrap cannot be applied directly to our testing problem. We conjecture that our test can be implemented by an AR-sieve bootstrap procedure if AR approximation results can be established for high dimensional and locally stationary time series. We point out that in the literature, the authors of \cite{kreiss2011} studied the AR sieve bootstrap for stationary processes and provided some very general and deep results.  The generalization to locally stationary time series is nontrivial and should be studied separately in another project.

Moreover, we point out that resampling for temporally dependent data typically requires one tuning parameter to account for the temporal dependence.  For example, the block bootstrap and subsampling will need to introduce the block size parameter. For our multiplier bootstrap, the parameter $m$ is introduced for the same purpose.  We notice that for the AR sieve bootstrap in (\ref{eq_highdimensionalAR}), one also needs to choose an order $\mathrm{b}_{**}$ for the AR approximation. Note that $\mathrm{b}_{**}$ is typically required to be different from the AR approximation order $b_*$ of the original time series $\{x_i\}$. Please also note that the role of $\mathrm{b}_{**} $ is again for adjusting the bootstrap to the temporal dependence. Therefore, we believe that, introducing a tuning parameter to account for temporal dependence is quite typical for resampling methods for time series.

}

\section{Technical proofs}\label{sec_mainproof}
This section is devoted to the technical proofs of the paper. { For the reader's convenience, we offer a brief description of the proof strategies before providing the actual technical detail. }

%

\subsection{Proofs of the main results of Section \ref{sec:preliminary}}

{ In this subsection, we provide the technical proof for the results regarding AR approximation theory established in Section \ref{sec:preliminary}. }
%

{ In what follows, we first prove Theorem \ref{lem_phibound}. It contains two parts. The first part (\ref{eq_phibound1}) is to show that the AR coefficients $\phi_{ij}$ decays polynomially fast with $j$ when $n$ is sufficiently large. The starting point of the proof is the Yule-Walker's equation representation  for $\phi_{ij}, 1 \leq j \leq i-1$ as in (\ref{eq_system1}). Such a representation is valid due to the UPDC as in Assumption \ref{assu_pdc}. Consequently, in order to study $\phi_{ij},$ it suffices to control the entries of the $j$th row of  $\Gamma_i^{-1}$ and all the entries of $\bm{\gamma}_i.$ Unfortunately, a direct control for the entries of $\Gamma_i^{-1}$ is not available. Instead, we need to firstly find an approximation for $\Gamma_i.$ The main motivation is from a result in modern operator theory (i.e., Lemma \ref{lem_band}) which states that the inverse of a banded matrix can also be approximated by another banded-like matrix. Note that under the short-range temporal decay condition Assumption \ref{assum_shortrange}, for each large $j,$ we can find a banded matrix $\Gamma_{i}^s \equiv \Gamma_i^s(j)$ as in (\ref{eq_bandedmatrixform}) so that $\Gamma_i^{-1}$ and $(\Gamma_i^s)^{-1}$ are close in the sense of (\ref{eq_bandedapprox}), which can be used to study the Yule-Walker's equations. Armed with these ingredients, we can construct an approximation for $\phi_{ij}$ via the Yule-Walker representation (\ref{eq_expansiondefn}), denoted as $\phi_{ij}^s.$ On the one hand, $\phi_{ij}$ and $\phi_{ij}^s$ are close in the sense of (\ref{eq_bbbbbbb111111}). On the other hand, we point out that $\phi_{ij}^s$ is constructed via Yule-Walker's equation using the inverse of the banded matrix $\Gamma_i^s.$  Therefore, according to Lemma \ref{lem_band}, the entries of (the $j$th row of) $(\Gamma_i^s)^{-1}$ can be effectively controlled as in (\ref{eq_bandin}) so that together with Assumption \ref{assum_shortrange}  $\phi_{ij}^s$ can be estimated as in (\ref{eq_phijscontrol}). This concludes the first part of the proof.  For the second part of the proof in (\ref{eq_phibound2}), note that the approximation $\{\phi_{ij}^b\}_{j \geq 1}$ can be expressed as the solution of a linear system via Yule-Walker's equation as in (\ref{eq_system2}), whereas the original coefficients $\{\phi_{ij}\}_{j \geq 1}$ are defined via another linear system via Yule-Walker's equation as in  (\ref{eq_system1}). In this regard,  (\ref{eq_system2}) can be viewed as a perturbation of  (\ref{eq_system1}) so that the error analysis for solutions of perturbation of linear systems (i.e., Lemma \ref{lem_nuem}) can be applied. The discussion for $j=0$ is straightforward using their definitions as in (\ref{eq_errorcontrol}) once we obtain the results for $j \geq 1.$  { We point out that the error rate $b^{-(\tau-1)}$ is obtained by studying a $b$-banded perturbed Yule-Walker's equation via the control of linear system; see (\ref{eq_deltagammaone}) below for more details. Moreover, the slower convergence of the trends { $|\phi_{i0,n}-\phi_{i0,n}^b|$} is mainly technical due to the use of Cauchy-Schwarz inequality; see (\ref{eq_errorcontrol}) below for more details.}
}

{
\begin{proof}[\bf Proof of Theorem \ref{lem_phibound}] We start with the proof of (\ref{eq_phibound1}). 
Till the end of the proof, we focus our discussion on each fixed $j.$ { Recall from (\ref{eq_newdefnphi}) that $\bm{\phi_i}=(\phi_{i1}, \cdots, \phi_{i,i-1})^* \in \mathbb{R}^{i-1}.$} By the Yule-Walker's equation, we have 
\begin{equation}\label{eq_system1}
\bm{\phi_i}=\Gamma_i^{-1} \bm{\gamma}_i,
\end{equation}
where $\Gamma_i=\text{Cov}(\bm{x}_{i-1}, \bm{x}_{i-1})$ and $\bm{\gamma}_i=\text{Cov}(\bm{x}_{i}, x_i) $ with $\bm{x}_{i-1}=(x_{i-1}, \cdots, x_1)^*.$
Note that for some constants $C_1, C_2>0,$
\begin{equation}\label{eq_boundedproof}
| \bm{\phi_i} | \leq \frac{1}{\lambda_{\min}(\Gamma_i)} |  \bm{\gamma}_i | \leq C_1 \kappa \sum_{k=1}^{i-1} k^{-\tau} \leq C_2, 
\end{equation} 
where in the second inequality we used the UPDC condition in Assumption \ref{assu_pdc} and Assumption \ref{assum_shortrange} that $\tau>1.$

Note that when $j=O(1),$ the result holds immediately according to (\ref{eq_boundedproof}). We next focus our discussion on the case when  $j$ diverges with $n.$
Since $i>j,$ $i$ also diverges with $n.$ We denote the $(i-1) \times (i-1)$ symmetric banded matrix $\Gamma_i^s \equiv \Gamma_{i}^s(j)$  by
\begin{equation}\label{eq_bandedmatrixform}
(\Gamma_i^s)_{kl}=
\begin{cases} 
(\Gamma_i)_{kl}, & |k-l| \leq \frac{j}{K \log j}; \\
0, & \text{otherwise}.
\end{cases}
\end{equation}
{Here $K>0$ is some large constant. By Lemma \ref{lem_disc}, Assumption \ref{assum_shortrange} and the UPDC condition in Assumption \ref{assu_pdc},}
we have for some constant $C>0,$  $$\lambda_{\min}(\Gamma_i^s) \geq \kappa-Cj^{1-\tau} (K \log j)^{\tau-1},$$ for all $i>j.$  Similarly, we can show that $\lambda_{\max}(\Gamma_i^s) \leq C$ for some constant $C>0.$  Since $n$ is sufficiently large and $j$ diverges with $n$, the above arguments
show that the support of the spectrum of $\Gamma_i^s$ is bounded from both above and below by some constants. Therefore, by Lemma \ref{lem_band}, we conclude that for some $\delta \in (0,1)$ and some constant $C>0,$ we have 
\begin{equation}\label{eq_bandin}
\left| (\Gamma_i^s)^{-1}_{kl}\right| \leq C \delta^{(K|k-l| \log j)/j}.
\end{equation} 
By Cauchy-Schwarz inequality and Lemma \ref{lem_disc}, when $n$ is large enough, for some constant $C>0,$ we have that  
\begin{align}\label{eq_bandedapprox}
 \left| \Gamma_i^{-1}\bm{\gamma}_i-(\Gamma_i^s)^{-1} \bm{\gamma}_i \right|  & \leq  
\| \Gamma_i-\Gamma_i^s \| \| \Gamma_i^{-1} \| \| (\Gamma_i^s)^{-1} \| | \bm{\gamma}_i | \nonumber \\
 & \leq C j^{1-\tau} (K \log j)^{\tau-1},
\end{align}
where we used (\ref{eq_polynomialdecay}), the UPDC in Assumption \ref{assu_pdc} and the conclusion $\lambda_{\min}(\Gamma_i^s) \geq C_1,$ for some constant $C_1>0.$   
 Denote $\bm{\phi}_i^s=(\phi_{i1}^s, \cdots, \phi_{i,i-1}^s)$ such that 
\begin{equation}\label{eq_expansiondefn}
\bm{\phi}_i^s=(\Gamma_i^s)^{-1} \bm{\gamma}_i.
\end{equation}
Then we get immediately from (\ref{eq_bandedapprox}) that 
\begin{equation}\label{eq_bbbbbbb111111}
|\phi_{ij}-\phi_{ij}^s| \leq  C j^{1-\tau} (K \log j)^{\tau-1}.
\end{equation}
Hence, it suffices to control $\phi_{ij}^s.$ {By (\ref{eq_expansiondefn}),} we note that $\phi_{ij}^s=\sum_{k=1}^{i-1} (\Gamma^s_i)_{jk}^{-1} \gamma_{ik},$ where we recall that $\gamma_{ik}=\text{Cov}(x_i, x_{i-k}).$ By (\ref{eq_polynomialdecay}) and (\ref{eq_bandin}), we have that for some constants $C, C_1>0$ 
\begin{align}\label{eq_phijscontrol}
|\phi_{ij}^s|  \leq C \sum_{k=1}^{i-1}\delta^{(K|k-j| \log j)/j} k^{-\tau} & =C \left( \sum_{k=1}^{j-1} \delta^{(K(j-k) \log j)/j} k^{-\tau}+ \sum_{k=j}^{i-1} \delta^{(K(k-j) \log j)/j} k^{-\tau} \right) \nonumber \\
& \leq  C_1 \left(\sum_{k=1}^{j-1} \delta^{K \log j (j-k)/j} k^{-\tau} +j^{1-\tau} \right).
\end{align}
where in the second inequality we used the fact that 
$\delta^{(K(k-j) \log j)/j}$ is bounded for $k \geq j$. Furthermore, to control the first summation of the right-hand side of (\ref{eq_phijscontrol}), since $j$ diverges with $n,$ we see that 
\begin{equation*}
\sum_{k=1}^{j-1} \delta^{K \log j (j-k)/j} k^{-\tau}\asymp \sum_{k=1}^{j-1} j^{-K(1-k/j)} k^{-\tau}. 
\end{equation*}
Let $f(k)=j^{-K(1-k/j)} k^{-\tau}.$ By an elementary derivative argument, it is easy to see that $f(k)$ is decreasing between $1$ and $\tau j/(K \log j)$ and increasing between $\tau j/(K \log j)$ and $j-1.$  As a result, since $j^{-K(1-k/j)}$ is bounded when $k<j,$ for some constants $C, C_1>0,$ we have  
\begin{align*}
\sum_{k=1}^{j-1} j^{-K(1-k/j)} k^{-\tau} & \leq \left(\frac{\tau j}{K \log j} \right) j^{-K(1-1/j)}+C\sum_{k=\tau j/(K \log j)}^{j-1} k^{-\tau} \\
& \leq C_1 (j/\log j)^{-\tau+1},
\end{align*}   
where in the second inequality we used the fact that $K$ is a large constant and $j$ diverges. 
Together with (\ref{eq_bbbbbbb111111}), we conclude our proof of (\ref{eq_phibound1}). 

Then we proceed to prove the first equation of (\ref{eq_phibound2}) using Lemma \ref{lem_nuem}. For the convenience of our discussion, we denote the $(k,l)$-entry of $\Gamma_i$ as $\Gamma_i(k,l).$
For $i>b,$ we denote the $(i-1) \times (i-1)$ block matrix $\Gamma_i^b$ and the block vector $\bm{\gamma}_i^b \in \mathbb{R}^{i-1}$ via
\begin{equation*}
\Gamma_{i}^b=
\begin{bmatrix}
\operatorname{Cov}(\bm{x}_i^b, \bm{x}_i^b) & \bm{E}_1 \\
\bm{E}_3 & \bm{E}_2
\end{bmatrix}, \ 
\bm{\gamma}_{i}^b=(\operatorname{Cov}(\bm{x}_i^b, x_i), \bm{0}),
\end{equation*}
where $\bm{x}_i^b=(x_{i-1}, \cdots, x_{i-b})^*$ and $\bm{E}_i, i=1,2,$ are defined as
\begin{equation}\label{eq_defne1e2}
\bm{E}_1=\operatorname{Cov}( \bm{x}_i^m , \bm{x}^b_i) \in \mathbb{R}^{b \times (i-b-1)}, \  \bm{E}_2=\operatorname{Cov}(\bm{x}_i^m, \bm{x}_i^m) \in \mathbb{R}^{(i-b-1) \times (i-b-1)},
\end{equation}
and $\bm{x}_i^m=(x_{i-b-1}, \cdots, x_1)^*.$ Moreover, $\bm{E}_3=(\bm{E}_3(k,l)) \in \mathbb{R}^{(i-b-1) \times b}$ is denoted as { 
\begin{equation}\label{eq_e3}
\bm{E}_3=\bm{E}_1^*.
\end{equation} }
%
{ Recall from (\ref{eq_newdefnphi}) that $\bm{\phi}_i^b=(\phi^b_{i1}, \cdots, \phi^b_{ib}, \bm{0})^* \in \mathbb{R}^{i-1}.$} We have that 
\begin{equation}\label{eq_system2}
\Gamma_i^b\bm{\phi}_i^b=\bm{\gamma_i}^b-\Delta \bm{\gamma}_i, 
\end{equation}
where $ \Delta\bm{\gamma}_i$ is defined as 
\begin{equation*}
\Delta \bm{\gamma}_i=(\bm{E}_3 \widetilde{\bm{\phi}}_i^b, \bm{0}), \ \widetilde{\bm{\phi}}_i^b=(\phi_{i1}^b, \cdots, \phi_{ib}^b)^*.
\end{equation*}
{ Therefore, it suffices to provide an upper bound for 
\begin{equation}\label{eq_differenbound}
| \bm{\phi}_i-\bm{\phi}_i^b |.
\end{equation}} 
Now we employ Lemma \ref{lem_nuem} with $A=\Gamma_i, \Delta A=\Gamma_i^b-\Gamma_i, x=\bm{\phi}_i, \Delta x=\bm{\phi}_i^b-\bm{\phi}_i, v=\bm{\gamma}_i, \Delta v=\bm{\gamma}_i^b-\bm{\gamma}_i-\Delta \bm{\gamma}_i$ to the systems (\ref{eq_system1}) and (\ref{eq_system2}). By the UPDC in Assumption \ref{assu_pdc}, for some constant $C>0,$ we find that $\kappa(A) \leq C.$ By Lemma \ref{lem_disc} and (\ref{eq_polynomialdecay}), we find that for some constant $C>0,$ we have 
\begin{equation*}
\| \Delta A \|\leq (b/\varsigma)^{-\tau+1} \leq C b^{-\tau+1}.
\end{equation*}
Moreover, note that
\begin{equation*}
| \Delta v | \leq | \bm{\gamma}_i^b-\bm{\gamma}_i |+| \Delta \bm{\gamma}_i |.
\end{equation*}
The first term of the right-hand side of the above equation can be bounded by $C b^{-\tau+1}$ using (\ref{eq_polynomialdecay}). For the second term, by a discussion similar to (\ref{eq_phibound1}), we find that  
\begin{equation}\label{eq_bbb}
|\phi_{ij}^b| \leq C ((\log j+1)/j)^{\tau-1}.
\end{equation}
 Using the definition of $\bm{E}_3$ in (\ref{eq_e3}), we claim that for some constants $C>0$ {
\begin{align}\label{eq_deltagammaone}
| \Delta \bm{\gamma}_i |  \leq C b^{-\tau+1} (\log b)^{\tau-1}.  
\end{align}
Combining the above discussion, we have that 
\begin{equation}\label{eq_l2norm}
| \bm{\phi}_i-\bm{\phi}_i^b | \leq C b^{-\tau+1} (\log b)^{\tau-1}. 
\end{equation} } 
{ To see (\ref{eq_deltagammaone}), using (\ref{eq_e3}) and  (\ref{eq_bbb}),  we notice that for $1 \leq k \leq i-b-1$
\begin{align*}
\left(\bm{E}_3 \widetilde{\bm{\phi}}_i^b \right)_{k} & =O\left( \sum_{j=1}^b (k+b-j+1)^{-\tau} j^{-\tau+1} (\log j)^{\tau-1} \right) \\
&=O\left( (\log b)^{\tau-1} \sum_{j=1}^b (k+b-j+1)^{-\tau} j^{-\tau+1}  \right).
\end{align*}
Moreover, since $\tau>2,$ we can control 
\begin{align*}
 \sum_{j=1}^b (k+b-j+1)^{-\tau} j^{-\tau+1} & =\frac{1}{(k+b+1)^{\tau}} \sum_{j=1}^b \frac{(k+b+1-j+j)^{\tau}}{(k+b-j+1)^{\tau} j^{\tau-1}}   \\
 & \leq \frac{C}{(k+b+1)^\tau} \left( 1+ \sum_{j=1}^b \frac{k+b+1}{(k+b+1-j)^{\tau}} \right) \\
 & \leq \frac{C}{(k+b+1)^\tau}  \left(1+\frac{k+b+1}{(k+1)^{\tau-1}} \right),
\end{align*}
where  $C>0$ is some constant and in the second step we used the elementary inequality that for $\mathfrak{a}, \mathfrak{b}>0,$ $(\mathfrak{a}+\mathfrak{b})^{\tau} \leq 2^{\tau-1}(\mathfrak{a}^\tau+\mathfrak{b}^\tau)$ for the denominator.  Consequently, by the definition of $\Delta \bm{\gamma}_i,$  we have that 
 \begin{align*}
 | \Delta \bm{\gamma}_i |_2^2 & =O\left( (\log b)^{2\tau-2} \left[ \sum_{k=1}^{i-b-1} \frac{1}{(k+b+1)^{2\tau}}+\sum_{k=1}^{i-b-1} \frac{1}{(k+1)^{2\tau-2} (k+b+1)^{2\tau-2}} \right] \right) \\
 & =O\left( (\log b)^{2\tau-2} \left[ b^{-2\tau+1}+b^{-2\tau+2} \right] \right)=O((\log b)^{2\tau-2} b^{-2\tau+2} ),
 \end{align*}
 where in the third step we used the assumption $\tau>2$ and
 \begin{align*}
 \sum_{k=1}^{i-b-1} \frac{1}{(k+1)^{2\tau-2} (k+b+1)^{2\tau-2}} & \leq \frac{1}{(b+2)^{2\tau-2}} \sum_{k=1}^{i-b-1} \frac{1}{(k+1)^{2\tau-2}} \\
& =O\left( b^{-2\tau+2} \right).
 \end{align*}
}
This finishes our proof of (\ref{eq_deltagammaone}) and hence the first equation of (\ref{eq_phibound2}). 

Finally, we prove the second equation of (\ref{eq_phibound2}). Note that
\begin{equation}\label{eq_noncenteredintercept}
\phi_{i0}=\mu_i-\sum_{j=1}^{i-1} \phi_{ij} \mu_{i-j}, \ \ \phi_{i0}^b=\mu_i-\sum_{j=1}^b\phi_{ij}^b \mu_{i-j},
\end{equation}
where $\mu_i=\mathbb{E}x_i, i=1,2,\cdots, n,$ is  the sequence of trends of $\{x_i\}.$  We have 
\begin{equation}\label{eq_errorcontrol}
\phi_{i0}-\phi_{i0}^b=\sum_{j=1}^b(\phi_{ij}^b-\phi_{ij})\mu_{i-j}-\sum_{j=b+1}^{i-1} \phi_{ij} \mu_{i-j}.
\end{equation}
{ Under Assumption \ref{assum_shortrange}, the first term of the right-hand side of the above equation is bounded by $C (\log b)^{\tau-1} b^{-(\tau-1.5)} $ using (\ref{eq_l2norm}) and Cauchy-Schwarz inequality and the second term can be bounded by  {$C b^{-(\tau-2)} (\log b)^{\tau}$ using (\ref{eq_phibound1}) when $b$ is sufficiently large. } This concludes our proof. }


%
\end{proof}
}


{ Once Theorem \ref{lem_phibound} is established, we can prove Theorem \ref{thm_arrepresent} by decomposing $x_i$ as in (\ref{eq_xidecompose}). The last (residual) term on the right-hand side of (\ref{eq_xidecompose}) can be controlled using Theorem \ref{lem_phibound} as in (\ref{eq_controlll22}).  }

\begin{proof}[\bf Proof of Theorem \ref{thm_arrepresent}] 
We start with the first part.  First of all, when $i \leq
 b,$ it holds by setting $\phi_{ij}$ to be the coefficients of best linear prediction. When $i>b, $  by (\ref{eq_arapproxiamtionnoncenter}), we decompose that 
 \begin{equation}\label{eq_xidecompose}
 x_i=\phi_{i0}+\sum_{j=1}^b \phi_{ij}x_{i-j}+\epsilon_i+\sum_{j=b+1}^{i-1} \phi_{ij} x_{i-j}.
 \end{equation}
By (\ref{eq_polynomialdecay}) of the main article, we have that
\begin{equation}\label{eq_controlll22}
\mathbb{E}\left|\sum_{j=b+1}^{i-1} \phi_{ij} x_{i-j} \right|^2 \leq C \sum_k \sum_j \phi_{i,j} \phi_{i,j+k} k^{-\tau}.
\end{equation} 
Together with Theorem \ref{lem_phibound}, we find that 
\begin{equation*}
\sum_{j=b+1}^{i-1} \phi_{ij} x_{i-j}=O_{\ell^2}(b^{-(\tau-1.5)} (\log b)^{\tau-1} ). 
\end{equation*}

This concludes our proof of the first part. 

Next, we prove the second part. Recall (\ref{defn_xistart}).
Clearly, $\{x_i^*\}$ is an AR($b$) process when $i>b.$ For $i=b+1,$ we have that
\begin{equation*}
x_i-x_i^*=0.
\end{equation*}
Suppose (\ref{eq_induction}) holds true for $k>b+1,$ then for $k+1,$ we have 
\begin{align*}
x_{k+1}-x_{k+1}^* & =\sum_{j=1}^b \phi_{ij} (x_{k+1-j}-x^*_{k+1-j})+\sum_{j=b+1}^{k} \phi_{ij} x_{k+1-j} \\
& =O_{\ell^2}(b^{-(\tau-1.5)} (\log b)^{\tau-1} ), 
\end{align*}

where in the second step we used induction and Theorem \ref{lem_phibound}.
\end{proof}

{ Then we prove Proposition \ref{prop_pdc}. First of all, due to the short-range dependence assumption, i.e., Assumption \ref{assum_shortrange}, we can show that the eigenvalues of $\operatorname{Cov}(x_1, \cdots, x_n)$ is sufficiently close to those of a banded matrix $\Sigma^{d_n}$ as in (\ref{eq_bandband}), especially the smallest eigenvalues are close as in (\ref{eq_boundlargeusesmall}). Here $d_n$ is some parameter  which controls the bandedness and will be chosen in the proof. For the sufficiency part, the key ingredient is Lemma \ref{lem_bandedbound}. It proves that if the spectral density is bounded from below, then for a length $d_n$  subsequence of the time series, it must satisfy the UPDC. In fact, when $d_n$ is small, due to the local stationarity assumption, the subsequence of the time series behaves like stationary time series so that Herglotz's theorem (i.e., Lemma \ref{lem_spectralbound}) implies UPDC. Moreover, the covariance matrix of the subsequence is close to the non-zero entries of the banded matrix $\Sigma^{d_n}.$ Together with (\ref{eq_boundlargeusesmall}), we can prove the sufficient part. The proof of necessity is similar and contains two steps. In the first step,  using the short-range dependence assumption,  following the classic theory of spectral density function, we construct a function $f_n(t, \omega)$ in (\ref{defn_fndefn}) which is sufficiently close to $f(t, \omega)$ when $n$ is sufficiently large.  In the second step, we will show that $f_n(t,w)$ is bounded from below by a fixed positive constant. To do so, we construct another function $g_n(t,\omega)$ defined in (\ref{eq_gndefinition}) using the banded matrix $\Sigma^{d_n}$ which is bounded from below by a positive constant by UPDC and the short-range assumption. Moreover, $f_n(t,\omega)$ and $g_n(t,\omega)$ are sufficiently close when $n$ is sufficiently large as in (\ref{eq_gf}).

  }

{
\begin{proof}[\bf Proof of Proposition \ref{prop_pdc}]
Denote the covariance matrix of $(x_1, \cdots, x_n)$ as $\Sigma \equiv \Sigma_n.$  For a given truncation level { $d_n \asymp n^f, \ 0<f<\frac{1}{2},$ } we define the banded matrix $\Sigma^{d_n}$ such that 
\begin{equation}\label{eq_bandband}
\Sigma_{ij}^{d_n}= 
\begin{cases}
\Sigma_{ij}, & \text{if} \ |i-j| \leq d_n; \\
0 , & \text{otherwise}. 
\end{cases}
\end{equation} 
Throughout the proof, we let $\lambda_n$ be the smallest eigenvalue of $\Sigma$ and $\mu_n$ be that of $\Sigma^{d_n}.$ {Under Assumption \ref{assum_shortrange}, when $n$ is large enough,} by Lemma \ref{lem_disc}, we have 
\begin{equation}\label{eq_boundlargeusesmall}
\lambda_n=\mu_n+o(1),
\end{equation}
{where we used the assumption that $\tau>1$ in (\ref{eq_polynomialdecay}).} Therefore, it is equivalent to study the UPDC for $\Sigma^{d_n}.$ We now consider a longer time series $\{x_i\}_{i=-d_n}^{n+d_n},$ where we use the convention $x_i=G(0, \mathcal{F}_i)$ if $i<0$ and $x_i=G(1, \mathcal{F}_i)$ if $i>n.$ We will need the following lemma to prove the sufficiency.
\begin{lem} \label{lem_bandedbound} Let $\Sigma_i^{d_n}$ be the covariance matrix of $(x_i, x_{i+1}, \cdots, x_{i+d_n}).$ Then for  all $-d_n \leq i \leq n,$ let $\lambda_{d_n}(\Sigma_i^{d_n})$ be the smallest eigenvalue of $\Sigma_i^{d_n}.$ Then if the spectral density (\ref{eq_spetraldensity}) is bounded from below, we have that for some constant $\varsigma>0,$
\begin{equation*}
\lambda_{d_n} (\Sigma_i^{d_n}) \geq \varsigma>0, \  \text{for all} \ i. 
\end{equation*}
\end{lem}
\begin{proof}
Without loss of generality, we set $i=0.$ Consider the stationary process $\{x_i^0\}$ such that $\gamma(0,\cdot)$ is its autocovariance function.  By Lemma \ref{lem_spectralbound}, when the spectral density is bounded below, we find that 
\begin{equation}\label{eq_boundone}
\lambda_{d_n}(\text{Cov}(x_i^0, \cdots, x_{d_n}^0)) \geq \varsigma>0,
\end{equation}
for any $d_n.$ Moreover, when $1 \leq i, j \leq d_n,$ for some constant $C>0,$ we have 
\begin{equation*}
\left| \text{Cov}(x_i, x_j)-\text{Cov} (x_i^0, x_j^0) \right| \leq C \min \left( \frac{\max(i,j)}{n}, |i-j|^{-\tau} 
\right),
\end{equation*}
{where the first bound comes from the (\ref{eq_covdefn}) and the Lipschitz continuity of $\gamma$ in $t$,  and the second bound is due to (\ref{eq_polynomialdecay}).} As a consequence, by Lemma \ref{lem_disc}, we find that
\begin{equation*}
|\lambda_{d_n}(\Sigma_i^{d_n})-\lambda_{d_n}(\text{Cov}(x_i^0, \cdots, x_{d_n}^0))| \leq C \frac{d_n^2}{n}.
\end{equation*}
Together with (\ref{eq_boundone}), we finish the proof.
\end{proof}
With the above preparation, we proceed with the final proof. We start with the sufficiency part. For any non-zero vector $\bm{a}=(a_1, \cdots, a_{i+2d_n})^* \in \mathbb{R}^{i+2d_n},\ i=-d_n, \cdots, n,$ denote 
\begin{equation*}
F(\bm{a}, i):=\sum_{k=1}^{i+d_n} \sum_{l=1}^{i+d_n} a_k (\Sigma_i^{d_n})_{k,l} a_l.
\end{equation*}  
By Lemma \ref{lem_bandedbound} and { the structure of $\Sigma^{d_n}$}, we find that 
\begin{equation}\label{eq_fbound}
F(\bm{a}, i) \geq \varsigma \sum_{l=i}^{i+d_n} a_l^2.  
\end{equation}
Now we let the first and last $d_n$ entries of $\bm{a}$ be zeros. Then using a discussion similar to (\ref{eq_fbound}), we find that 
\begin{equation}\label{eq_finalbound2}
\frac{1}{d_n} \sum_{i=-d_n}^n F(\bm{a}, i) \geq \varsigma \sum_{l=1}^n a_l^2. 
\end{equation}
Furthermore, by Lemma \ref{lem_disc}, it is easy to see that for some constant $C>0,$
\begin{equation*}
\left| \frac{1}{d_n} \sum_{i=-d_n}^n F(\bm{a}, i)-\sum_{k=1}^n \sum_{l=1}^n a_k \Sigma^{d_n}_{kl} a_l \right| \leq \frac{C}{d_n} \sum_{k=1}^n a_k^2. 
\end{equation*}
Together with (\ref{eq_finalbound2}), we find that 
\begin{equation*}
\sum_{k=1}^n \sum_{l=1}^n a_k \Sigma^{d_n}_{kl} a_l \geq \frac{\varsigma}{2} \sum_{l=1}^n a_l^2,
\end{equation*}
when $n$ is large enough. This shows that $\Sigma^{d_n}$ satisfies PDC and hence finishes the proof of the sufficient part. 

Next we briefly discuss the proof of necessity. We make use of the structure of $\Sigma^{d_n}.$ For any given $t_i:=\frac{i}{n}$ and $\omega,$ denote {
\begin{align}\label{defn_fndefn}
f_{n}(t_i, \omega) & =\frac{1}{2 \pi n} \sum_{k,l=1}^n e^{-\mathrm{i} k \omega} \gamma(t_i, k-l)  e^{\mathrm{i} l \omega} \\
&=\frac{1}{2 \pi } \sum_{|h|<n} \left(1-\frac{|h|}{n} \right) e^{-\mathrm{i} h \omega} \gamma(t_i,h).
\end{align} }
{ Using the short-range dependence Assumption \ref{assum_shortrange},} it is easy to check that (for instance see a similar discussion as in \cite[Corollary 4.3.2]{BD}) { for sufficiently large $n$}
\begin{equation}\label{eq_ffi}
f_n(t_i, \omega)= f(t_i, \omega)+o(1).
\end{equation} 
{ Consequently, it suffices to show that $f_n(t_i,\omega)$ is bounded from below by a constant for sufficiently large $n$. To achieve this, we introduce a spectral density function defined by the banded matrix $\Sigma^{d_n},$ i.e.,} 
\begin{equation}\label{eq_gndefinition}
g_n(t_i,\omega)=\frac{1}{2 \pi n} \sum_{k,l=1}^n e^{-\mathrm{i} k \omega}  \Sigma^{d_n}_{i, |k-l|} e^{\mathrm{i} l \omega},
\end{equation}
{ where we used the notation that $\Sigma^{d_n}_{i, |k-l|}$ is the $(i, |k-l|)$ entry of the banded matrix $\Sigma^{d_n}$. Recall that $d_n \asymp n^{f}$ for $0<f<1/2.$} { Using the definitions of $f_n(t_i,\omega)$ and $g_n(t_i, \omega)$}, we find that { for some constants $C_1, C_2, C_3>0,$ when $n$ is sufficiently large
\begin{align*}
\left| g_n(t_i, \omega)-f_n(t_i,\omega) \right| & \leq  \frac{C_1}{n} \sum_{k,l=1}^n \left|\gamma(t_i, k-l)-\Sigma_{i,|k-l|}^{d_n} \right| \notag \\ 
& \leq C_2 \left[ \frac{1}{n}\left( \sum_{k=1}^n \sum_{|k-l|>d_n} |\gamma(t_i, k-l) |\right)+ \frac{1}{n}\left( \sum_{k=1}^n \sum_{|k-l| \leq d_n} |\gamma(t_i, k-l)-\operatorname{Cov}(x_k, x_l)|\right) \right] \notag \\
& \leq C_3 \left( n^{f(-\tau+1)}+n^{-f+1} \right),  
\end{align*}
where in the second step we used the definition of $\Sigma^{d_n}$ and in third step we used the assumption (\ref{eq_covdefn}) and Assumption \ref{assum_shortrange} (or equivalently (\ref{eq_definedmodified})).
Consequently, as $\tau>1$ and $0<f<1/2,$ when $n$ is sufficiently large we have that 
\begin{equation}\label{eq_gf}
g_n(t_i, \omega)=f_n(t_i, \omega)+o(1). 
\end{equation}
Using the structure of $\Sigma^{d_n}$, the assumption that $\Sigma$ satisfies UPDC and (\ref{eq_boundlargeusesmall})}, we find that for some constant $\kappa>0,$
\begin{equation*}
g_n(t_i, \omega) \geq \kappa. 
\end{equation*} 
In light of  (\ref{eq_ffi}) and (\ref{eq_gf}), we find that $f(t_i, \omega) \geq \kappa.$ Finally, we can conclude our proof using the continuity of $f(t, \omega)$ in $t$. This concludes our proof.   
\end{proof}

}

%

{ Finally, we prove Theorem \ref{thm_locallynonzero}. The proof ideas are similar to those used earlier. For example, the key part of the proof of (\ref{eq_controlarcoeff1}) makes use of the Yule-Walker representation (\ref{eq_phismoothdefinition}) and  (\ref{eq_yueyueeee2222}). The actual control, as in the discussion between (\ref{eq_trianaleinequality}) and (\ref{eq_bcontrol}), rely on Cauchy-Schwarz inequality with UPDC  Assumption \ref{assu_pdc} and the local stationarity condition (\ref{eq_covdefn}).  }

{
\begin{proof}[ \bf Proof of Theorem \ref{thm_locallynonzero}] For the first statement regarding the smoothness of $\phi_j(t),$ under Assumption \ref{assu_smoothtrend}, the case $1 \leq j \leq b$ follows from Lemma 3.1 of \cite{DZ1}. When $j=0,$ i.e., $\phi_0(t)$ defined in (\ref{eq_phi0tdefn}), the smoothness can be easily proved using
term by term differentiation, Assumption \ref{assu_smoothtrend} and the results $\phi_j(t) \in C^d([0,1]), 1 \leq j \leq b$.

{
We then prove (\ref{eq_controlarcoeff1}). Since $\eta_i$'s in the filtration $\mathcal{F}_i$ are i.i.d., in light of the definition of $\phi_j(t),$ we can equivalently write $\phi_j(i/n), 1 \leq j \leq b$ in the following way. Recall ${\bm{\phi}}(\frac{i}{n})=(\phi_1(\frac{i}{n}), \cdots, \phi_b(\frac{i}{n}))^*$ via $\bm{\phi}(\frac{i}{n})=\Gamma(i/n)^{-1} \bm{\gamma}(i/n)$ as in (\ref{eq_phismoothdefinition}). 
In view of (\ref{eq_phibound2}), similar to the discussion of (\ref{eq_differenbound}),  it suffices to offer an upper bound for $| \bm{\phi}^b_i-\bm{\phi}(\frac{i}{n}) |.$ Note that $\bm{\phi}_i^b=(\phi_{i1}^b, \cdots, \phi_{ib}^b)^* \in \mathbb{R}^b$ is governed by the following Yule-Walker's equation 
\begin{equation}\label{eq_yueyueeee2222}
\bm{\phi}_i^b=\Gamma_{ib}^{-1} \bm{\gamma}_{ib},
\end{equation} 
where $\Gamma_{ib}=\text{Cov}(\bm{x}_i^b, \bm{x}_i^b), \bm{\gamma}_{ib}=\text{Cov}(\bm{x}_i^b, x_i)$ with $\bm{x}_i^b=(x_{i-1}, \cdots, x_{i-b})^*.$ Consequently, we have that 
\begin{align}\label{eq_trianaleinequality}
|  \bm{\phi}^b_i-\bm{\phi}(\frac{i}{n}) | & \leq \| \Gamma_{ib}^{-1} \| | \bm{\gamma}_{ib} -\bm{\gamma}(i/n) | \nonumber \\
&+\| \Gamma(i/n)^{-1} \| \| \Gamma_{ib}^{-1} \| \| \Gamma(i/n)-\Gamma_{ib} \| | \bm{\gamma}(i/n) |.
\end{align}
By (\ref{eq_covdefn}) and Cauchy-Schwarz inequality, it is easy to see that for some constant $C>0$
\begin{equation*}
|  \bm{\gamma}_{ib} -\bm{\gamma}(i/n) | \leq C \frac{b^{1.5}}{n}.
\end{equation*}
Similarly, together with Lemma \ref{lem_disc}, we see that
\begin{equation*}
\| \Gamma(i/n)-\Gamma_{ib} \| \leq C\frac{b^2}{n}.
\end{equation*}
Under Assumption \ref{assu_pdc}, combing with (\ref{eq_trianaleinequality}), we conclude that 
\begin{equation}\label{eq_bcontrol}
|  \bm{\phi}^b_i-\bm{\phi}(\frac{i}{n})| \leq C\frac{b^2}{n}.
\end{equation}
This completes our proof for (\ref{eq_controlarcoeff1}). 

For the proof of (\ref{eq_controlarcoeff2}), (\ref{eq_phi0tdefn}) implies that 
\begin{equation*}
\phi_0(\frac{i}{n})=\mu(\frac{i}{n})-\sum_{j=1}^{b} \phi_j(\frac{i}{n}) { \mu(\frac{i}{n})}.
\end{equation*}
By Assumption \ref{assu_smoothtrend} that $|\mu(i/n)-\mu((i-j)/n)| \leq C_1 b/n$ when $|i-j| \leq b$ for some constant $C_1>0,$  (\ref{eq_noncenteredintercept}) and  (\ref{eq_bcontrol}), we find that for some constant $C>0,$
\begin{equation*}
\left| \phi_0(\frac{i}{n})-\phi_{i0}^b \right| \leq C\frac{b^{2.5}}{n},
\end{equation*}
where we used Cauchy-Schwarz inequality. Then we can prove (\ref{eq_controlarcoeff2}) using (\ref{eq_phibound2}).

Further, invoking (\ref{eq_arapproxiamtionnoncenter}),  (\ref{eq_choleskylocal}) follows from (\ref{eq_controlarcoeff1}), (\ref{eq_controlarcoeff2}) and (\ref{assum_moment}). Finally, the proof of (\ref{cor_localboundmse}) is similar to those of Theorem \ref{thm_arrepresent} using Theorems \ref{thm_locallynonzero} and \ref{thm_arrepresent} and we omit further details here. 


}



\end{proof}
}

\subsection{Proofs of the main results of Section \ref{sec:test}} { In this subsection, we provide the technical proof for the results regarding the inferential theory for locally stationary time series established in Section \ref{sec:test}.} We point out that throughout this subsection, we use the choice of $b_*$ as in (\ref{eq_choiceofb}). { First of all, we prove Theorem \ref{lem_reducedtest}. For the proof of (\ref{eq:zzz}), we need to construct $\varrho(j).$ In fact, by the established approximation theory, we find that $\operatorname{Corr}(x_i, x_{i+j})$ can be well approximated by a locally stationary AR process as in (\ref{eq_originalcontrol}). Moreover, the ACF of the locally AR process is sufficiently close to that of a stationary AR process as defined in (\ref{eq_stationaryzidefinition}). { For the proof of (\ref{eq_errortermone}), it relies on a perturbation argument with Lemma \ref{lem_nuem}. The staring point is the  Yule-Walker's representation (\ref{eq_acfdefinitionyulewalker}) of $\bm{\phi}(i/n).$ Under the null hypothesis that the time series is correlation stationary, we can show that it is close to another system (\ref{eq_pertubed}) which results in time invariant solution. Moreover, the error can be controlled using the locally stationary assumption (\ref{eq_covdefn}) and the UPDC in Assumption \ref{assu_pdc}. } This completes the proof.  

}



%
%
%
%
%
%

\begin{proof}[\bf Proof of Theorem \ref{lem_reducedtest}]
By Assumptions \ref{assum_shortrange} and \ref{assum_local}, we find that there exists some constant $C>0,$ such that {
\begin{equation*}
\sup_i \left| \operatorname{Corr}(x_i, x_{i+j}) \right| \leq C |j|^{-\tau}. 
\end{equation*} }
Therefore, we only consider the correlation when $|j| \leq b_*.$ Indeed, due to Assumption \ref{assum_local}, for some constant $C>0,$ we have 
{
\begin{equation*}
\sup_{1 \leq i \leq b_*} \left| \operatorname{Corr}(x_i, x_{i+j})-\operatorname{Corr}(x_{b_*+i}, x_{b_*+j+i}) \right| \leq C \frac{b_*}{n}.
\end{equation*}  
}
Therefore, it suffices to test the stationarity  for the correlation of $x_i$ and $x_{i+j},$ where $i>b_*$ and $|j| \leq b_*.$ { Without loss of generality, in what follows, we assume $j \geq 0$  and focus on the setting $0 \leq j \leq b_*.$}


{ We start with the proof of (\ref{eq:zzz}). Recall (\ref{eq_phismoothdefinition}) and the physical representation (\ref{eq_defncov}). }
{ We first notice that when $i>b_*$ and $j \leq b_*,$ by Theorem \ref{thm_locallynonzero}, we can see that $x_i$ can be well approximated by a locally stationary AR process, i.e., 
\begin{equation*}
x_i-\phi_{i0}=x_i^{**}-\phi_0(i/n)+O_{\ell_2}\left((\log b)^{\tau-1} b^{-(\tau-1)}+\frac{b^2}{n}\right),
\end{equation*} 
where $x_i^{**}$ is defined in (\ref{eq_xiss}). This implies that
\begin{equation}\label{eq_originalcontrol}
\operatorname{Corr}(x_i, x_{i+j})= \operatorname{Corr}(x^{**}_i, x^{**}_{i+j})+O\left((\log b)^{\tau-1} b^{-(\tau-1)}+\frac{b^2}{n}\right).
\end{equation}
Now under the assumption of $\mathbf{H}_0$ that $\phi_k(\cdot), 1 \leq k \leq b_*$ are constant functions, we can hence write 
\begin{equation*}
x_i^{**}-\phi_0(i/n)=\sum_{j=1}^b \phi_j x_{i-j}^{**}+ \sigma_i \widetilde{\epsilon}_i, 
\end{equation*}
where $\sigma_i^2$ is the variance of $\epsilon_i$ and $\{\widetilde{\epsilon}_i\}$ is a stationary white noise process with mean zero and variance one. This implies that 
\begin{equation*}
\frac{x_i^{**}-\phi_0(i/n)}{\sigma_i}=\sum_{j=1}^b \phi_j \frac{x_{i-j}^{**}}{\sigma_i}+ \widetilde{\epsilon}_i. 
\end{equation*}
By (\ref{eq_vvvffff}), we can see that for some smooth function $\varphi(\cdot)$ 
\begin{equation}\label{eq_sigmaclose}
\sigma_i=\varphi(i/n)+O\left( (\log b)^{\tau-1} b^{-(\tau-1.5)}+\frac{b^2}{n} \right).
\end{equation}
This implies that for $k>0$
\begin{equation}\label{eq_differencecontrol}
\sigma_i=\sigma_{i\pm k}+O\left((\log b)^{\tau-1} b^{-(\tau-1.5)}+\frac{b^2}{n}+\frac{k}{n} \right). 
\end{equation}
Now we denote 
\begin{equation*}
y_i=\frac{x_i^{**}-\phi_0(i/n)}{\sigma_i}. 
\end{equation*}
On the one hand, by definition, we have that  $\operatorname{Corr}(x_i^{**}, x_j^{**})=\operatorname{Corr}(y_i, y_j).$ On the other hand, we see from (\ref{eq_differencecontrol}) and Cauchy-Schwarz inequality that
\begin{equation*}
y_i=\sum_{i=1}^b \phi_j y_{i-j}+\widetilde{\epsilon}_i+O\left((\log b)^{\tau-1} b^{-(\tau-2)}+\frac{b^{2.5}}{n} \right).
\end{equation*}
Let $z_i$ be defined as 
\begin{equation}\label{eq_stationaryzidefinition}
z_i=\sum_{i=1}^b \phi_j z_{i-j}+\widetilde{\epsilon}_i,
\end{equation}
which is clearly a stationary time series whose ACF is denoted as $\varrho$. It is easy to see that 
\begin{equation*}
\operatorname{Corr}(z_i, z_{i+j})= \operatorname{Corr}(y_i, y_{i+j})+O\left((\log b)^{\tau-1} b^{-(\tau-2)}+\frac{b^{2.5}}{n} \right).
\end{equation*}
Combining with (\ref{eq_originalcontrol}), we can  conclude our proof.}

{ Then we prove (\ref{eq_errortermone}). Using Yule-Walker's equation, by setting $t=i/n,$ we can write 
\begin{equation}\label{eq_acfdefinitionyulewalker}
\Gamma(i/n) \bm{\phi}(i/n)=\gamma(i/n). 
\end{equation}
Furthermore, under the null hypothesis  $\mathbf{H}'_0,$ that $\{x_i\}$ is correlation stationary, we can define $\bm{\phi} \equiv \bm{\phi}^{b_*}$ independent of $i/n$ according to 
\begin{equation*}
\mathrm{P} \bm{\phi}=\bm{\rho},
\end{equation*}
where $\mathrm{P}$ is the correlation matrix of $\bm{x}_{i-1}^{b_*}=(x_{i-1}, \cdots, x_{i-b_*})^*$ and $\bm{\rho}$ is the correlation vector of $x_i$ and $\bm{x}^{b_*}_{i-1}.$ We can multiply $\operatorname{Var}(x_i)$ on both sides of the above equation and rewrite it as
\begin{equation}\label{eq_pertubed}
\operatorname{Var}(x_i) \mathrm{P} \bm{\phi}=\operatorname{Var}(x_i) \bm{\rho}. 
\end{equation}
In general, (\ref{eq_pertubed}) can be regarded as a perturbed system of (\ref{eq_acfdefinitionyulewalker}), and when $\operatorname{Var}(x_i)$ is independent of $i,$ (\ref{eq_pertubed}) is identical to (\ref{eq_acfdefinitionyulewalker}). Now we write (\ref{eq_pertubed}) as
\begin{equation*}
\left(\Gamma(i/n)+\operatorname{Var}(x_i) \mathrm{P}-\Gamma(i/n) \right)\left( \bm{\phi}(i/n)+\bm{\phi}-\bm{\phi}(i/n) \right)=\left( \bm{\gamma}(i/n)+\operatorname{Var}(x_i) \bm{\rho}-\bm{\gamma}(i/n) \right).
\end{equation*}
In order to control $\| \bm{\phi}-\bm{\phi}(i/n) \|,$ we apply the perturbation theory, i.e., Lemma \ref{lem_nuem}. Note that for the k$th$ entry, $1 \leq k \leq b_*$ of  $\operatorname{Var}(x_i) \bm{\rho}-\bm{\gamma}(i/n),$ it reads as 
\begin{equation*}
\operatorname{Var}(x_i) \rho(k)-\sqrt{\gamma(i/n,0) \gamma((i-k)/n,0)} \rho(k),
\end{equation*}
where we recall the definition of $\gamma(\cdot, \cdot)$ in (\ref{eq_covdefn}).  Moreover, by the assumption of (\ref{eq_covdefn}), it is easy to see that  
\begin{equation*}
\operatorname{Var}(x_i)=\sqrt{\gamma(i/n,0) \gamma((i-k)/n,0)}+O\left( \frac{k+1}{n} \right). 
\end{equation*}
Therefore, we can conclude that 
\begin{equation*}
\left | \operatorname{Var}(x_i) \bm{\rho}-\bm{\gamma}(i/n) \right|=O\left( \frac{b^2_*}{n} \right).
\end{equation*}
Similarly, using Lemma \ref{lem_disc}, we can show that
\begin{equation*}
\left\| \operatorname{Var}(x_i) \mathrm{P}-\Gamma(i/n) \right\|=O\left( \frac{b^2_*}{n} \right). 
\end{equation*}  
By the UPDC in Assumption \ref{assu_pdc} and  a discussion similar to (\ref{eq_differenbound}), we can complete our proof using Lemma \ref{lem_nuem}. 
}


 
\end{proof}

{ Then we prove Lemma \ref{lem_reducedquadratic} following a straightforward algebraic computation (\ref{eq_tform}) with the OLS representation as in (\ref{beta_est})}

\begin{proof}[\bf Proof of Lemma \ref{lem_reducedquadratic}] 
Under the null assumption $\mathbf{H}_0$ that $\phi_j(t)$ are identical in $t$,  we have
\begin{equation*}
(\widehat{\phi}_j(t)-\overline{\widehat{\phi}}_j)^2=\left(\widehat{\phi}_j(t)-\phi_j(t)-\left(\int_0^1 (\widehat{\phi}_j(s)-\phi_j(s) )ds\right)\right)^2.
\end{equation*}
By (\ref{eq_phiform}), we can  write 
\begin{equation}\label{eq_tform}
T=\sum_{j=1}^{b_*} \left( \bm{\beta}_j^*-\widehat{\bm{\beta}}_j^* \right) W \Big( \bm{\beta}_j-\widehat{\bm{\beta}}_j \Big)+O(b_* c^{-d}), \ W=\Big(I-\bar{B} \bar{B}^*\Big),
\end{equation}
where $\bm{\beta}_j \in \mathbb{R}^c$ satisfies that $\bm{\beta}_{jk}=\bm{\beta}_{jc+k}, \ 1 \leq k \leq c.$ 
It is well-known that the OLS estimator satisfies 
\begin{equation}\label{beta_est}
\widehat{\bm{\beta}}=\bm{\beta}+\left(\frac{Y^*Y}{n} \right)^{-1}\frac{Y^* \bm{\epsilon}}{n}, \ \bm{\epsilon}=(\epsilon_{b_*+1}, \cdots, \epsilon_n)^*.
\end{equation}
 By (\ref{beta_est}), we find that $nT$ is a quadratic form in terms of $\frac{1}{\sqrt{n}} \sum_{i=b_*+1}^n \bm{z}_i.$ 
We find that 
\begin{equation*}
nT=\mathbf{X}^* \left( \frac{Y^*Y}{n} \right)^{-1} \mathbf{I}_{b_* c} \mathbf{W} \left( \frac{Y^*Y}{n} \right)^{-1} \mathbf{X}+O_{\mathbb{P}}(b_* c^{-d}).
\end{equation*}
By (2) of Lemma \ref{lem_coll} and (1) and (3) of Assumption \ref{assu_basis}, we can conclude our proof.  

\end{proof}

{ Next, we prove the Gaussian approximation result Theorem \ref{thm_gaussian} utilizing the Gaussian approximation result on convex sets for $m$-dependent sequence (i.e., Lemma \ref{lem_xf}). The starting point is that since our statistic is a quadratic form, the control of $\mathcal{K}(\mathbf{X}, \mathbf{Y})$ is reduced to proving a Gaussian approximation result on convex sets as in (\ref{eq_defncalk}). We point out that such a result has been proved for bounded and $m$-dependent sequence in \cite[Theorem 2.1]{FX}, i.e., Lemma \ref{lem_xf}. Even though our observations are not exactly bounded and $m$-dependent,  thanks to the short-range dependence assumption and the established concentration inequalities as summarized in Lemma \ref{lem_con}, we can find such a bounded and $m$-dependent approximation $\overline{\mathbf{X}}^M$ as in the lines below (\ref{eq_setsetsetsetsetset}) so that $\operatorname{Cov}(\mathbf{X})$ is close to $\operatorname{Cov}(\overline{\mathbf{X}}^M)$ as in (\ref{eq_covariance}). This will enable us to apply Lemma \ref{lem_xf}. For the actual proof, we can control $\mathcal{K}(\mathbf{X}, \mathbf{Y})$ using triangle inequality, i.e., 
\begin{equation*}
\mathcal{K}(\mathbf{X}, \mathbf{Y}) \leq \mathcal{K}(\mathbf{X}, \mathbf{X}^M)+\mathcal{K}(\mathbf{X}^M, \overline{\mathbf{X}}^M)+\mathcal{K}(\overline{\mathbf{X}}^M, \mathbf{Y}),
\end{equation*} 
where $\mathbf{X}^M$ is the $m$-approximation without truncation. First, $\mathcal{K}(\overline{\mathbf{X}}^M, \mathbf{Y})$ can be controlled using Lemma \ref{lem_xf} as $\overline{\mathbf{X}}^M$ is bounded and $m$-dependent. The result is recorded in (\ref{eq_onepartcontrol}). Second, $\mathcal{K}(\mathbf{X}, \mathbf{X}^M)$ can be controlled using the $m$-dependent approximation result for locally stationary time series as in (\ref{eq_controlcontrolthird}) using Lemma \ref{lem_con}. Finally, $\mathcal{K}(\mathbf{X}^M, \overline{\mathbf{X}}^M)$ can be controlled using the concentration inequalities in Lemma \ref{lem_con} since $\overline{\mathbf{X}}^M$ is the truncated version of $\mathbf{X}^M.$
  }
\begin{proof}[\bf Proof of Theorem \ref{thm_gaussian}]  
Denote 
\begin{equation*}
A_x:=\Big\{ \mathbf{W} \in \mathbb{R}^p: \mathbf{W}^* \Gamma \mathbf{W} \leq x \Big\}.
\end{equation*}
It is easy to check that  $A_x$  is convex as $\Gamma$ is positive semi-definite. By definition, we have 
\begin{align}\label{eq_defncalk}
\mathcal{K}(\mathbf{X}, \mathbf{Y})=\sup_x \Big| \mathbb{P} \Big( \mathbf{X} \in A_x \Big)-\mathbb{P} \Big( \mathbf{Y} \in A_x \Big) \Big|\le \sup_{A \in \mathcal{A}} \Big| \mathbb{P} \Big( \mathbf{X} \in A \Big)-\mathbb{P} \Big( \mathbf{Y} \in A \Big) \Big|,
\end{align}
where  $\mathcal{A}$ is the collection of all convex sets in $\mathbb{R}^p$. Given a large constant $M \equiv M(n),$ denote
\begin{equation*}
\bm{h}_i^M= \mathbb{E}(\bm{h}_i | \eta_{i-M}, \cdots, \eta_i), \ i=b_*+1, \cdots,n,
\end{equation*}
and $\bm{z}_i^M=\bm{h}_i^M \otimes \mathbf{B}(\frac{i}{n})=(z_{i1}^M, \cdots, z_{ip}^M)^*.$ Recall $p=(b_*+1)c.$ Then we can define $\mathbf{X}^M$ accordingly and then $\mathbf{Y}^M$ can be defined similarly. Note that  in Lemma \ref{lem_xf}, we have $n_1=n_2=n_3=M.$ Next we provide a truncation for the $M$-dependent sequence.      Now we choose $M_z$ for $\gamma \in (0,1),$ such that
\begin{equation*}
\mathbb{P} \Big( \max_{b_*+1 \leq i \leq n} \max_{1 \leq j \leq p} |z^M_{ij}| \geq M_z\Big) \leq \gamma.
\end{equation*}
Denote the set
\begin{equation}\label{eq_setsetsetsetsetset}
\mathcal{B}(M_z):=\Big\{ \max_{b_*+1 \leq i \leq n} \max_{1 \leq j \leq p} |z^M_{ij}| \leq M_z\Big\},
\end{equation}
and $\mathbf{X}=(X_1, \cdots, X_p).$ Similarly, we can define its $M$-dependent approximation as $\mathbf{X}^M$ and truncated version as $\overline{\mathbf{X}}^M.$ We decompose the probability by  
\begin{align}
\mathcal{K}(\mathbf{X}^M, \mathbf{Y})&=
\mathcal{K}(\mathbf{X}^M, \mathbf{Y} \cap \mathcal{B}(M_z)) +\mathcal{K}(\mathbf{X}^M, \mathbf{Y} \cap \mathcal{B}^c(M_z)) \nonumber \\
& \leq \mathcal{K}(\overline{\mathbf{X}}^M, \mathbf{Y})+C \gamma, \label{eq_finalshow}
\end{align}
where $C>0$ is some constant. 
%
Note that on $\mathcal{B}(M_z),$
\begin{equation*}
\Big|\frac{1}{\sqrt{n}}\bm{z}^M_i \Big|=\frac{1}{\sqrt{n}}\Big| \bm{h}^M_i \otimes \mathbf{B}(\frac{i}{n}) \Big| \leq \frac{C \sqrt{p} M_z}{\sqrt{n}}. 
\end{equation*}
Denote $\widetilde{\mathbf{Y}}^M$ as the Gaussian random vector with the same covariance structure with $\overline{\mathbf{X}}^M$ whose Gaussian part is the same as $\mathbf{Y}.$ By Lemma \ref{lem_xf}, we conclude that 
\begin{equation*}
\mathcal{K}(\overline{\mathbf{X}}^M, \widetilde{\mathbf{Y}}^M) \leq C p^{\frac{7}{4}} n^{-1/2}M_z^3 M^2.
\end{equation*}
In light of (\ref{eq_finalshow}), it suffices to control the difference of the covariance matrices between $\overline{\mathbf{X}}^M$ and $\mathbf{X}.$
First, we show that the covariance matrices between $\overline{\mathbf{X}}^M$ and $\mathbf{X}^M$ are close.  We emphasize that $\overline{X}_i^M \neq X_i \mathbf{1}(|X_i| \leq M_z).$ We need to conduct a more careful analysis. For $i=1,2,\cdots,p,$
\begin{equation}\label{eq_variancedecomposition}
\operatorname{Var}(\overline{X}_i^M)-\operatorname{Var}(X_i^M)=\mathbb{E}(\overline{X}_i^M)^2-\mathbb{E} (X_i^M)^2+(\mathbb{E}(\overline{X}_i^M-X_i^M))(\mathbb{E}(\overline{X}_i^M+X_i^M)).
\end{equation}
Note that
\begin{align*}
\left|\mathbb{E}(\overline{X}_i^M-X_i^M)\right|=\left|\frac{1}{\sqrt{n}} \sum_{k=b_*+1}^n \mathbb{E}(z_{ki}^M -\overline{z}_{ki}^M) \right|=\frac{1}{\sqrt{n}} \left| \sum_{k=b_*+1}^n \mathbb{E}z_{ki}^M \mathbf{1}(|z_{ki}^M|> M_z) \right| \leq \sqrt{n} \xi_c M_{z}^{-q+1},
\end{align*}
where we used the fact 
\begin{equation*}
\mathbf{1}\left(|z_{ki}^M|>M_z \right) \leq \frac{|z_{ki}^M|^{q-1}}{M_z^{q-1}},
\end{equation*}
 and Markov inequality. By Cauchy-Schwarz inequality, we can show analogously that for some constant $C>0$ 
\begin{equation*}
\left| \mathbb{E}(\overline{X}_i^M)^2-\mathbb{E} (X_i^M)^2 \right| \leq C \xi_c^{2} n M_z^{-(q-1)}.
\end{equation*}
This implies that for some constant $C>0$ 
\begin{equation*}
\left|\operatorname{Var}(\overline{X}_i^M)-\operatorname{Var}(X_i^M) \right| \leq C n \xi_c^2 M_z^{-(q-2)}.
\end{equation*}
Similarly, we can show that 
\begin{equation*}
\left|\operatorname{Cov}(\overline{X}_i^M,\overline{X}_j^M)-\operatorname{Cov}(X_i^M, X_j^M) \right| \leq C n \xi_c^2 M_z^{-(q-1)}.
\end{equation*}
Together with Lemma \ref{lem_disc}, we find that 
\begin{equation*}
\| \operatorname{Cov}(\mathbf{X}^M)- \operatorname{Cov}(\overline{\mathbf{X}}^M)  \|\leq C \xi_c^2 pn M_z^{-(q-2)}.
\end{equation*}

Second, we control the difference between $\mathbf{X}^M$ and $\mathbf{X}.$ By \cite[Lemma A.1]{LL} (or Lemma \ref{lem_con}), we have 
\begin{equation}\label{eq_bbbd}
\mathbb{E} \left(|X_j-X_j^M|^q \right)^{2/q} \leq C \Theta^2_{M,j,q}.
\end{equation}
By (\ref{eq_physcialbounbounbound}), we conclude that 
\begin{equation}\label{eq_bbbd1}
\Theta_{M,j,q} \leq C \xi_c M^{-\tau+1}.
\end{equation}
Consequently, by Jenson's inequality, we have that 
\begin{equation}\label{eq_mommentboundfinal}
\mathbb{E}|X_j-X_j^M| \leq C \xi_c M^{-\tau+1}, \ \mathbb{E}|X_j-X_j^M|^2 \leq C \xi_c^2 M^{-2\tau+2}. 
\end{equation}
Therefore, we have that for some constant $C>0,$
\begin{equation*}
\left|\operatorname{Var}(X_i^M)-\operatorname{Var}(X_i^M) \right| \leq C \xi_c M^{-\tau+1},
\end{equation*}
where we use a discussion similar to (\ref{eq_variancedecomposition}). Similarly, we can show that 
\begin{equation*}
\left|\operatorname{Cov}(X_i^M,X_j^M)-\operatorname{Cov}(X_i, X_j) \right| \leq C \xi_c M^{-\tau+1}.
\end{equation*}
Together with Lemma \ref{lem_disc},  we find that 
\begin{equation*}
\| \operatorname{Cov}(\mathbf{X})- \operatorname{Cov}(\mathbf{X}^M)  \|\leq C p\xi_c M^{-\tau+1}.
\end{equation*}
As a result, we conclude that 
\begin{equation}\label{eq_covariance}
\| \operatorname{Cov}(\mathbf{X})- \operatorname{Cov}(\overline{\mathbf{X}}^M)  \|\leq C(p\xi_c M^{-\tau+1}+p n \xi_c^2 M_z^{-(q-2)}).
\end{equation}
We decompose that
\begin{equation}\label{eq_ffffff}
\mathbb{P}(\mathbf{Y} \Gamma \mathbf{Y}^* \leq x)-\mathbf{P}(\widetilde{\mathbf{Y}}^M \Gamma (\widetilde{\mathbf{Y}}^M)^* \leq x)=\mathbb{P}(\mathbf{Y} \Gamma \mathbf{Y}^* \leq x)-\mathbb{P}(\mathbf{Y} \Gamma \mathbf{Y}^* \leq x+\mathcal{D}(\mathbf{Y}, \widetilde{\mathbf{Y}}^M)),
\end{equation}
where $\mathcal{D}(\mathbf{Y}, \widetilde{\mathbf{Y}})$ is defined as 
\begin{equation*}
\mathcal{D}(\mathbf{Y}, \widetilde{\mathbf{Y}}^M):=-\widetilde{\mathbf{Y}}^M \Gamma (\widetilde{\mathbf{Y}}^M)^*+\mathbf{Y} \Gamma \mathbf{Y}^* .
\end{equation*}
By (\ref{eq_covariance}), a decomposition similar to (\ref{eq_decompositionone}) below and Bernstein's inequality (see Example 2.11 of \cite{MR3967104}),  for some small constant $\delta>0,$ we have with $1-O(n^{-\delta})$ probability 
\begin{equation*}
\|  \mathcal{D}(\mathbf{Y}, \widetilde{\mathbf{Y}}^M) \|_2 \leq C p n^{\delta} \xi_c (p\xi_c M^{-\tau+1}+p\xi_c^2 n M_z^{-(q-2)}).
\end{equation*}
By Lemma \ref{lem_chisquare} and (\ref{eq_ffffff}), we find that with $1-O(n^{-\delta})$ probability 
\begin{equation*}
\mathcal{K}(\mathbf{Y}, \widetilde{\mathbf{Y}}^M) \leq C \left( p  \xi_c (p\xi_c M^{-\tau+1}+p\xi_c^2 n M_z^{-(q-2)}) \right)^{1/2}.
\end{equation*}
Therefore, using the definition of $\mathcal{K}(\cdot,\cdot)$ in (\ref{eq_defncalk}), we conclude that for some small constant $\delta>0,$
\begin{equation}\label{eq_onepartcontrol}
\mathcal{K}(\mathbf{X}^M, \mathbf{Y}) \leq C \left(\gamma+p^{\frac{7}{4}} n^{-1/2}M_z^3 M^2+\left( p\xi_c  (p\xi_c M^{-\tau+1}+p \xi_c^2 n M_z^{-(q-2)}) \right)^{1/2} \right)+n^{-\delta}.
\end{equation}
It is clear that we can choose $\gamma=O\Big(\log n \frac{\xi_c}{M_z}\Big).$ Finally, we control $\mathcal{K}(\mathbf{X}, \mathbf{X}^M)$ to finish our proof. We first introduce some notations. Denote the physical dependence measure for $z_{kl}$ as $\delta_{kl}^z(s,q)$ and 
\begin{equation*}
\theta_{k,j,q}=\sup_{k} \delta_{kl}^z(s,q), \ \Theta_{s,l,q}=\sum_{o=s}^{\infty} \theta_{o,l,q}.
\end{equation*}
By (\ref{eq_physcialbounbounbound}), we conclude that 
\begin{equation}\label{eq_boundboundbound}
\sup_{1 \leq l \leq p} \Theta_{s,l,q}<\xi_c, \ \sum_{s=1}^{\infty} \sup_{1 \leq l \leq p}s \theta_{s,l,3}<\xi_c.  
\end{equation}
Denote the set
\begin{equation*}
\mathcal{I}(\Delta_M):=\Big \{ \max_{1 \leq j \leq p} \Big| X_j-X_j^{(M)} \Big| \leq \Delta_M \Big\}.
\end{equation*}
We claim that for arbitrary small $\delta>0,$ we can decompose the probability by { 
\begin{align}\label{eq_decompositionkey}
\mathcal{K}(\mathbf{X}^M, \mathbf{X})&=
\mathcal{K}(\mathbf{X}^M, \mathbf{X} \cap \mathcal{I}(\Delta_M)) +\mathcal{K}(\mathbf{X}^M, \mathbf{X}\cap \mathcal{I}^c(\Delta_M))  \\
& \leq C \Big( \sqrt{p \Delta_M \xi_c n^{\delta}} +n^{-\delta}+\mathbb{P}(\mathcal{I}^c(\Delta_M)) \Big), \nonumber
\end{align}
where we use the definition of $\mathcal{K}(\cdot, \cdot)$ to control the second term of the right-hand side of (\ref{eq_decompositionkey}). For the first term, note that
\begin{equation*}
\mathbb{P}\left( (\mathbf{X}^M)^* \Gamma \mathbf{X}^M \leq x \right)-\mathbb{P} \left( \mathbf{X} \Gamma \mathbf{X} \leq x \right)=\mathbb{P}\left( (\mathbf{X}^M)^* \Gamma \mathbf{X}^M \leq x \right)-\mathbb{P}\left( (\mathbf{X}^M)^* \Gamma \mathbf{X}^M \leq x+\mathcal{D}(\mathbf{X}^M, \mathbf{X}) \right),
\end{equation*}
where $\mathcal{D}(\mathbf{X}^M, \mathbf{M})$ is defined as
\begin{equation*}
\mathcal{D}(\mathbf{X}^M,\mathbf{X})=-\mathbf{X}^* \Gamma \mathbf{X}+(\mathbf{X}^M)^* \Gamma \mathbf{X}^M .
\end{equation*}
Further, we have
\begin{align}\label{eq_decompositionone}
\| \mathcal{D}(\mathbf{X}^M, \mathbf{M}) \|_2 \leq \|(\mathbf{X}^M)^* \Gamma (\mathbf{X}^M-\mathbf{X}) \|_2+\| (\mathbf{X}^M-\mathbf{X})^* \Gamma \mathbf{X}\|_2. 
\end{align}
Recall (\ref{eq_bbbd}) and (\ref{eq_bbbd1}). Restricted on $\mathcal{I}(\Delta_M),$ by Cauchy-Schwarz inequality, the fact $\Gamma$ is bounded, Lemma \ref{lem_con} with (\ref{eq_boundboundbound}), we find that for some constant $C>0,$ 
\begin{equation*}
\| \mathcal{D}(\mathbf{X}^M, \mathbf{X}) \|_2 \leq  C \sqrt{p} \xi_c (\sqrt{p} \Delta_M)=C p \Delta_M \xi_c.
\end{equation*}
Therefore, conditional on $\mathcal{I}(\Delta_M),$ for some constant $C>0,$ we have 
\begin{align*}
\left|\mathbb{P}\left( (\mathbf{X}^M)^* \Gamma \mathbf{X}^M \leq x \right)-\mathbb{P} \left( \mathbf{X} \Gamma \mathbf{X} \leq x \right)  \right|&  \leq C n^{-\delta} \\
& +\left|\mathbb{P}\left( (\mathbf{X}^M)^* \Gamma \mathbf{X}^M \leq x  \right)-\mathbb{P}\left( (\mathbf{X}^M)^* \Gamma \mathbf{X}^M \leq x+n^{\delta} p \Delta_M \xi_c  \right) \right|. 
\end{align*}
Moreover, we have 
\begin{align*}
\left|\mathbb{P}\left( (\mathbf{X}^M)^* \Gamma \mathbf{X}^M \leq x  \right) -\mathbb{P}\left( (\mathbf{X}^M)^* \Gamma \mathbf{X}^M \leq x+n^{\delta} p \Delta_M \xi_c  \right) \right| \leq & 2 \mathcal{K}(\mathbf{X}^M, \mathbf{Y})+ \mathbb{P}\left( \mathbf{Y}^* \Gamma \mathbf{Y} \leq x+n^{\delta} p \Delta_M \xi_c  \right) \\
& -\mathbb{P}\left( (\mathbf{Y}^* \Gamma \mathbf{Y} \leq x  \right).
\end{align*}
Since $\Gamma$ is positive definite and bounded, by Lemma \ref{lem_chisquare} and the rotation invariance property of Gaussian random vectors, we obtain the bound for the first term of the right-hand side of (\ref{eq_decompositionkey}).}  Next, by Markov inequality and a simple union bound, we have that 
\begin{equation*}
\mathbb{P}(\mathcal{I}^c(\Delta_M)) \leq C \sum_{j=1}^p \frac{\Theta_{M,j,q}^q}{\Delta_M^q}.
\end{equation*}
Consequently, we can control 
\begin{equation*}
\mathcal{K}(\mathbf{X}^M, \mathbf{X}) \leq C \Big( \sqrt{p \Delta_M \xi_c n^{\delta}} +n^{-\delta}+p \xi_c  M^{-q\tau+1}/\Delta^q_M \Big).
\end{equation*}
By optimizing $\Delta_M,$ we conclude that
\begin{equation}\label{eq_controlcontrolthird}
\mathcal{K}(\mathbf{X}^M, \mathbf{X}) \leq C \Big( M^{\frac{-q\tau+1}{2q+1}}\xi_c^{(q+1)/(2q+1)} p^{\frac{q+1}{2q+1}} n^{\frac{\delta q}{2q+1}} +n^{-\delta}\Big),
\end{equation}
 This finishes our proof using triangle inequality.
\end{proof}

{
\begin{rem}
The current convergence rate of Theorem \ref{thm_gaussian} depends on $\xi_c$ in such a way that the rate is slower for basis functions with larger $\xi_c$. For example, for wavelets,
\begin{equation*}
\xi_c=\sup_{1 \leq i\leq c} \sup_t |\alpha_i(t)|=2^{\log c/2}=O(c^{1/2}).
\end{equation*}
We believe that this is an artifact of the proof. Specifically, observe that the maximum magnitude of $ \bm{z}_i$ over $i$ is determined by $\xi_c$. Furthermore, the truncation effect as well as the dependence measures of $\{\bm{z}_i\}$ are associated with the magnitude of $\{\bm{z}_i\}$. In particular, our proof relies on controlling the truncation effect as well as the dependence measures utilizing the $l_\infty$ norm of the basis functions. As pointed out by one referee, utilizing other norms such as the $l_1$ or $l_2$ norm of the basis functions could improve the convergence rates of Theorem \ref{thm_gaussian}. As of now it is unclear to us how to use other norms of the basis functions to control the truncation effects and the dependence measures. We will pursue this direction in the future works.

\end{rem}
}

{ Armed with Theorem \ref{thm_gaussian}, we can prove Proposition \ref{prop_normal} using Lindeberg's central limit theorem. }

\begin{proof}[\bf Proof of Proposition \ref{prop_normal}] Denote $r=\text{Rank}( \Omega^{1/2} \Gamma \Omega^{1/2})$ and the eigenvalues of $\Omega^{1/2} \Gamma \Omega^{1/2}$ as $d_1 \geq d_2>\cdots\geq d_r. $ Under (1) of Assumption \ref{assu_basis}, the definition of $\mathbf{W}$ and the fact that 
\begin{equation*}
\lambda_{\min}(A) \lambda_{\min}(B) \leq \lambda_{\min}(AB) \leq \lambda_{\max}(AB) \leq \lambda_{\max}(A) \lambda_{\max}(B),
\end{equation*}
for any given positive semi-definite matrices $A$ and $B$, we conclude that $d_i=O(1),\ i=1,2,\cdots,r.$ For the basis functions we used, we have that $r=O(b_* c).$ Therefore, we have 
\begin{equation*}
\frac{d_1}{f_2} \rightarrow 0. 
\end{equation*}
Hence, by Theorem \ref{thm_gaussian} and Lindeberg's central limit theorem, we finish our proof. 

\end{proof}

{ Once Proposition \ref{prop_normal} is proved, using some straightforward decompositions (\ref{eq_ntttt1}) and (\ref{eq_nttttt2}) below, we can prove Proposition \ref{prop_power}.}

\begin{proof}[\bf Proof of Proposition \ref{prop_power}] 
 Denote the statistic $\mathcal{T}$ as
\begin{equation}\label{eq_ntttt1}
\mathcal{T}:=\sum_{j=1}^{b_*} \int_0^1 \Big( \widehat{\phi}_j(t)-\phi_j(t)- \Big( \int_0^1 \widehat{\phi}_j(s)-\phi_j(s) ds \Big) \Big)^2 dt.
\end{equation}
One one hand, by Proposition \ref{prop_normal}, we have that 
\begin{equation*}
\frac{n \mathcal{T}-f_1}{f_2} \Rightarrow \mathcal{N}(0,2).
\end{equation*}
On the other hand, by an elementary computation, we have 
\begin{equation}\label{eq_nttttt2}
n\mathcal{T}=nT+n \sum_{j=1}^{b_*} \int_0^1 \Big(\phi_j(t)-\bar{\phi}_j \Big)^2dt-2n \sum_{j=1}^{b_*} \int_0^1 \Big( \phi_j(t)-\bar{\phi}_j \Big)\Big( \widehat{\phi}_j(t)-\bar{\widehat{\phi}}_j \Big) dt. 
\end{equation}
Furthermore, we can rewrite the above equation as 
\begin{equation*}
n\mathcal{T}=nT-n\sum_{j=1}^{b_*} \int_0^1 \left( \phi_j(t)-\bar{\phi}_j \right)^2 dt+2n \sum_{j=1}^{b_*} \int_0^1 \left( \phi_j(t)-\bar{\phi}_j \right) \left( \phi_j(t)-\widehat{\phi}_j(t)-(\bar{\phi}_j-\bar{\widehat{\phi}}_j) \right) dt.
\end{equation*}
By (\ref{eq_phiform}), we find that
\begin{equation*}
\int_0^1 \left( \phi_j(t)-\bar{\phi}_j \right) \left( \phi_j(t)-\widehat{\phi}_j(t)-(\bar{\phi}_j-\bar{\widehat{\phi}}_j) \right)dt=\bm{\beta}_j^*\widehat{B}(\bm{\beta}_j-\widehat{\bm{\beta}}_j)+O(b_* c^{-d}),
\end{equation*}
where $\widehat{B}$ is defined as 
\begin{equation*}
\widehat{B}=\int_0^1 (\mathbf{B}(t)-\bar{B})(\mathbf{B}(t)-\bar{B})^*dt. 
\end{equation*}
It is easy to see that $\| \widehat{B} \|=O(1).$ Therefore, under the alternative hypothesis $\mathbf{H}_a,$ we find that 
\begin{equation*}
\int_0^1 \left( \phi_j(t)-\bar{\phi}_j \right) \left( \phi_j(t)-\widehat{\phi}_j(t)-(\bar{\phi}_j-\bar{\widehat{\phi}}_j) \right)dt=O_{\mathbb{P}} \left( \sqrt{\log n}\frac{(b_* c)^{1/4}}{n} \right),
\end{equation*}
where we use Theorem \ref{thm_finalresult} and Assumption \ref{assu_basis}. This concludes our proof of part one. For part two, it follows directly
from part one.   
\end{proof}

{ Finally, we justify the validity of our proposed multiplier bootstrap method, i.e., Theorem \ref{thm_bootstrapping}. In particular, we will show the asymptotic normality for $\widehat{\mathcal{T}}.$ The proof consists of two steps. In the first step, we work on $\mathcal{T}$ as in (\ref{eq_sampleversiont}). The difference between $\mathcal{T}$ and $\widehat{\mathcal{T}}$ is that in (\ref{eq_sampleversiont}) we use $\Phi$ as in (\ref{eq_defnphi}) which is defined using the white noise $\{\epsilon_i\}$ instead of the residual. By definition, $\Phi$ is Gaussian. According to Proposition \ref{prop_normal}, it suffices to show that conditional on the data, $\operatorname{Cov}(\Phi)$ is close to $\Omega.$ This is accomplished in  Lemmas \ref{lem_a2} and \ref{lem_a3} and concluded in (\ref{eq_lamdaomega}). In the second step, we can prove the results for $\widehat{\mathcal{T}}$ by showing the closeness of $\Phi$ and $\widehat{\Phi}.$  }


\begin{proof}[\bf Proof of Theorem \ref{thm_bootstrapping}] 
We divide our proofs into two steps. In the first step, we show that the result holds for $\mathcal{T}$ defined as
\begin{equation}\label{eq_sampleversiont}
\mathcal{T}:=\Phi^* \widehat{\Gamma} \Phi,
\end{equation}
In the second step, we control the closeness between $\mathcal{T}$ and $\widehat{\mathcal{T}}$ defined in (\ref{eq_defnmathcalt}). We start with the first step  following the proof strategy of \cite[Theorem 3]{ZZ1}. Denote
\begin{equation*}
\Lambda=\frac{1}{(n-m-b_*)} \sum_{i=b_*+1}^{n-m} \Upsilon_{i,m} \Upsilon_{i,m}^*,
\end{equation*}
where we use
\begin{equation*}
\Upsilon_{i,m}=\frac{1}{\sqrt{m}}H_i \otimes \mathbf{B}(\frac{i}{n}), \ H_i=\Big(\sum_{j=i}^{i+m} \bm{h}_j \Big) .
\end{equation*}
We first propose and prove the following Lemmas \ref{lem_a2} and \ref{lem_a3}.

\begin{lem}\label{lem_a2} Under the assumptions of Theorem \ref{thm_bootstrapping}, we have that for all  $b_*+1 \leq i \leq n-m$
\begin{equation*}
 \left| \left| \Upsilon_{i,m} \Upsilon_{i,m}^*-\mathbb{E} \Big( \Upsilon_{i,m} \Upsilon_{i,m}^* \Big) \right| \right|=O_{\mathbb{P}} \Big( b_* \zeta^2_c \sqrt{m} \Big).
\end{equation*}
\end{lem}
\begin{proof}
Using the basic property of Kronecker product, we find 
\begin{equation*}
\Upsilon_{i,m} \Upsilon_{i,m}^*=\frac{1}{m} \left[ H_i H_i^*  \right] \otimes \left[  \mathbf{B}(\frac{i}{n})  \mathbf{B}^*(\frac{i}{n})\right].
\end{equation*}
As a consequence, we have that 
\begin{equation}\label{eq_upsionone}
 \left| \left| \Upsilon_{i,m} \Upsilon_{i,m}^*-\mathbb{E} \Big( \Upsilon_{i,m} \Upsilon_{i,m}^* \Big) \right| \right| \leq 
  \left| \left| H_i H_i^*-\mathbb{E} \Big( H_i H_i^* \Big) \right| \right| \frac{\zeta^2_c}{m},
\end{equation}
where we use the property of the spectrum of Kronecker product and the fact  $  \mathbf{B}(\frac{i}{n})  \mathbf{B}^*(\frac{i}{n})$ is a rank-one matrix. Now we focus on studying the first entry of $H_i H_i^*,$ which is of the form $w=\Big(\sum_{j=i}^{i+m} x_{j-1}\epsilon_j \Big)^2.$ We first study its physical dependence measure.  Note that $w$ is $\mathcal{F}_{i+m}$ measurable and can be written as $f_i(\mathcal{F}_{i+m}).$ Denote $w(l)=f_i(\mathcal{F}_{i+m,l}).$ By (\ref{eq_physcialbounbounbound}) and Lemma \ref{lem_con}, we conclude that 
\begin{equation}\label{eq_priorbound}
 \left | \left | \sum_{j=i}^{i+m} x_{j-1} \epsilon_j \right| \right|_q=O(\sqrt{m}).
\end{equation} 
Recall that by Jensen's inequality, if $x \in L^q, q>4$, we have 
\begin{equation}\label{eq_jensonbound}
\mathbb{E}|x|^2 \leq (\mathbb{E}|x|^q)^{2/q}.
\end{equation} 
Therefore, by (\ref{eq_priorbound}), (\ref{eq_jensonbound}) and Minkowski's inequality, we have 
\begin{equation*}
||w-w(l)||_2=O(\sqrt{m}) \Big( \sum_{j=l-m}^l \delta(j,q) \Big).
\end{equation*}
By Lemma \ref{lem_con} and (\ref{eq_physcialbounbounbound}), we have 
\begin{equation*}
||w-\mathbb{E}w||_2=O(m^{3/2}).
\end{equation*}
Therefore, by (\ref{eq_upsionone}) and Lemma \ref{lem_disc}, we conclude our proof.


\end{proof}
Using a discussion similar to the lemma above and by (\ref{eq_assumimply}), it is easy to conclude that 
\begin{equation}\label{eq_lambdaconvergence}
|| \Lambda-\mathbb{E}(\Lambda) ||=O_{\mathbb{P}} \Big( b_*\zeta_c^2 \sqrt{m/n} \Big).
\end{equation}

Next, we show that the covariance of a stationary time series can be used  to approximate $\mathbb{E} \Big(H_jH_j^* \Big),$ where the stationary time series can closely preserve the long-run covariance matrix (\ref{eq_longrunh}). Recall (\ref{eq_defnh}).  Denote the stationary time series as
\begin{equation*}
\widetilde{\bm{h}}_{i,j}=\mathbf{U}(\frac{i}{n}, \mathcal{F}_{j}), \ i \leq j \leq i+m.
\end{equation*}
Correspondingly, we can define 
\begin{equation*}
\widetilde{\Upsilon}_{i,m}=\frac{1}{\sqrt{m}}\widetilde{H}_i \otimes \mathbf{B}(\frac{i}{n}), \ \widetilde{H}_i=\sum_{j=i}^{i+m} \widetilde{\bm{h}}_{i,j}.
\end{equation*} 
\begin{lem}\label{lem_a3} Under the assumptions of Theorem \ref{thm_bootstrapping}, we have that for all $b_*+1 \leq i \leq n-m$ 
\begin{equation*}
\left| \left| \mathbb{E} \Big( \Upsilon_{i,m} \Upsilon_{i,m}^*\Big)-\mathbb{E} \Big(\widetilde{\Upsilon}_{i,m} \widetilde{\Upsilon}_{i,m}^* \Big) \right| \right|=O \Big( \Big( \frac{mb_*^2}{n} \Big)^{1-2/\tau}  b_*\zeta_c^2 \Big).
\end{equation*}
\end{lem}
\begin{proof}
Similar to (\ref{eq_upsionone}), we have 
\begin{equation*}
 \left| \left| \mathbb{E} \Big( \Upsilon_{i,m} \Upsilon_{i,m}^*\Big)-\mathbb{E} \Big(\widetilde{\Upsilon}_{i,m} \widetilde{\Upsilon}_{i,m}^* \Big) \right| \right|\leq  \left| \left| \mathbb{E}(\widetilde{H}_i \widetilde{H}_i^*)-\mathbb{E} ( H_i H_i^* ) \right| \right| \frac{\zeta^2_c}{m}.
\end{equation*}
We also focus on studying the first entry of $ \widetilde{H}_i \widetilde{H}_i^*- H_i H_i^*,$ which is of the form $\Big(\sum_{j=i}^{i+m} \widetilde{x}_{j-1}\widetilde{\epsilon}_j \Big)^2-\Big(\sum_{j=i}^{i+m} x_{j-1}\epsilon_j \Big)^2.$ We first observe that
\begin{align*}
\left| \left| \sum_{j=i}^{i+m} \Big( \widetilde{x}_{j-1} \widetilde{\epsilon}_j-x_{j-1} \epsilon_j \Big) \right| \right|_2=O\left( \left| \left| \sum_{j=i}^{i+m} x_{j-1}(\widetilde{\epsilon}_{j}-\epsilon_{j}) \right| \right|_2 \right).
\end{align*}
Hence, by Lemma \ref{lem_con} and Assumption \ref{assum_local}, we have 
\begin{equation*}
\left| \left| \sum_{j=i}^{i+m} \Big( \widetilde{x}_{j-1} \widetilde{\epsilon}_j-x_{j-1} \epsilon_j \Big) \right| \right|_2=O\Big( \sqrt{m} \sum_{j=0}^{\infty}\min \{\frac{m}{n}, \delta(j,2)\} \Big)=O \Big( \sqrt{m} \Big( \frac{m}{n} \Big)^{1-2/\tau} \Big),
\end{equation*}
where we use the fact $\delta(j,2) \leq \delta(j,q).$ Hence, by (\ref{eq_jensonbound}) and Minkowski's inequality,  we have that 
\begin{equation*}
 \left| \left| \Big(\sum_{j=i}^{i+m} \widetilde{x}_{j-1}\widetilde{\epsilon}_j \Big)^2-\Big(\sum_{j=i}^{i+m} x_{j-1}\epsilon_j \Big)^2 \right| \right|_2 =O\Big(m \Big( \frac{m}{n} \Big)^{1-2/\tau}  \Big).
\end{equation*}
This concludes our proof using Lemma \ref{lem_disc}. 
\end{proof}
Furthermore, by \cite[Lemma 4]{ZZ1} and a discussion similar to (\ref{eq_upsionone}), we have 
 \begin{equation*}
\left| \left| \mathbb{E} \Big(\widetilde{\Upsilon}_{i,m} \widetilde{\Upsilon}_{i,m}^* \Big)-\Omega(\frac{i}{n}) \otimes \left(\mathbf{B}(\frac{i}{n}) \mathbf{B}(\frac{i}{n})^*\right) \right| \right|=O\Big(\frac{b_* \zeta_c^2}{m} \Big). 
 \end{equation*}
Hence, by Assumption \ref{assu_smoothtrend} and  \cite[Theorem 1.1]{HT}, we have
\begin{equation*}
\left\|\frac{1}{n-m-b_*} \sum_{i=b_*+1}^{n-m} \mathbb{E} \Big(\widetilde{\Upsilon}_{i,m} \widetilde{\Upsilon}_{i,m}^* \Big)-\int_0^1 \Omega(t) \otimes \left(\mathbf{B}(t) \mathbf{B}(t)^*\right)dt \right\|=O\Big(\frac{b_*\zeta_c^2}{m}+\frac{1}{(n-m-b_*)^2}\Big).
\end{equation*}
We now come back to our proof of Theorem \ref{thm_bootstrapping}. Under (\ref{eq_assumimply}), by Lemmas \ref{lem_con}, \ref{lem_a2} and \ref{lem_a3}, we have that 
\begin{equation}\label{eq_lamdaomega}
||\Lambda-\Omega||=O_{\mathbb{P}}\Big( \theta(m) \Big), \ \theta(m)=b_* \zeta_c^2 \left( \sqrt{\frac{m}{n}}+\frac{1}{\sqrt{n}}\Big(\frac{mb_*^2}{n} \Big)^{1-2/\tau}+\frac{1}{m} \right).
\end{equation}
It is easy to check that as $\tau>4,$
\begin{equation*}
\frac{1}{\sqrt{n}}\Big(\frac{m}{n} \Big)^{1-2/\tau} \leq \frac{1}{m},
\end{equation*}
where we use the assumption that $m \ll n.$ By definition, conditional on the data, $\Phi$ is normally distributed. Hence, we may write
\begin{equation*}
\Phi \equiv \Lambda^{1/2} \mathbf{G},
\end{equation*}
where $\mathbf{G} \sim \mathcal{N} (0, I_p)$ and $\equiv$ means that they have the same distribution. Define $r=\text{Rank}(\Lambda^{1/2} \widehat{\Gamma} \Lambda^{1/2})$ and the eigenvalues of $\Lambda^{1/2} \widehat{\Gamma} \Lambda^{1/2}$ as $\lambda_1 \geq \lambda_2 \geq \cdots \geq \lambda_r>0.$ By (\ref{eq_lambdaconvergence}) and Assumption \ref{assu_basis}, it is easy to see that  $\lambda_i=O(1)$ when conditional on the data. Therefore, by Lindeberg's central limit theorem, we have 
\begin{equation*}
\frac{\mathbf{G}^* \Lambda^{1/2} \widehat{\Gamma} \Lambda^{1/2} \mathbf{G}-\sum_{i=1}^r \lambda_i}{(\sum_{i=1}^r \lambda_i^2)^{1/2}} \Rightarrow \mathcal{N}(0,2). 
\end{equation*}
Recall that $d_1 \geq d_2 \geq \cdots \geq d_r>0$ are the eigenvalues of $\Omega^{1/2} \Gamma \Omega^{1/2}$ and note $d_i=O(1).$  Recall that $r=O(b_* c)$ and denote the set $\mathcal{A} \equiv \mathcal{A}_n$ as 
\begin{align}\label{eq_constructionset}
\mathcal{A} \equiv \mathcal{A}_n:=\Big\{|\sum_{i=1}^r (\lambda_i-d_i)| \leq b_n \sqrt{b_* c}, \ |\sum_{i=1}^r (\lambda_i^2-d_i^2)| \leq c_n \sqrt{b_* c} \Big\},
\end{align}
where $b_n, c_n=o(1).$ On the event $\mathcal{A},$ we have that
\begin{align} \label{eq_gapproximation}
\frac{\mathbf{G}^*  \Lambda^{1/2} \widehat{\Gamma} \Lambda^{1/2} \mathbf{G}-f_1}{f_2}&=\frac{\mathbf{G}^*  \Lambda^{1/2} \widehat{\Gamma} \Lambda^{1/2} \mathbf{G}-\sum_{i=1}^r \lambda_i+\sum_{i=1}^r \lambda_i-f_1}{(\sum_{i=1}^r \lambda_i^2)^{1/2}} \left(\frac{(\sum_{i=1}^r \lambda_i^2)^{1/2}}{f_2} \right) \nonumber \\
& = \frac{\mathbf{G}^* \Lambda^{1/2} \widehat{\Gamma} \Lambda^{1/2} \mathbf{G}-\sum_{i=1}^r \lambda_i}{(\sum_{i=1}^r \lambda_i^2)^{1/2}}+o(1). 
\end{align}
Therefore, we have shown that Theorem \ref{thm_bootstrapping} holds true on the event $\mathcal{A}.$ Under (\ref{eq_assumimply}), using a discussion similar to (\ref{eq_lamdaomega}) \footnote{The operator norm and the difference of trace share the same order as we apply Lemma \ref{lem_disc}.}  and  (2) of Lemma \ref{lem_coll}, we find that
\begin{equation*}
|| \widehat{\Sigma}-\Sigma ||=O_{\mathbb{P}} \Big( \frac{\zeta_c \log n}{\sqrt{n}} \Big).
\end{equation*}   
Consequently, we have that
$$\mathbb{P}(\mathcal{A})=1-o(1).$$ Hence, we can conclude our proof for $\mathcal{T}$ using Theorem \ref{thm_gaussian}. 

For the second step, by Theorems \ref{lem_phibound} and  a discussion similar to \cite[Theorem 3.7 and Corollary 3.8]{DZ1}, we conclude that
\begin{equation*}
\sup_{i>b} |\epsilon_i-\widehat{\epsilon}^b_i|=O_{\mathbb{P}}(\vartheta(n)), \ \vartheta_n= n^{1/q} \Big( b_* \zeta_c \sqrt{\frac{\log n}{n}}+n^{-d\mathfrak{a}} \Big).
\end{equation*}
Denote $\widehat{\Upsilon}_{i,m}$ by replacing $\bm{h}_{i}$ with $\widehat{\bm{h}}_i,$ i.e.,
\begin{equation}\label{eq_hatUp}
\widehat{\Upsilon}_{i,m}:=\frac{1}{\sqrt{m}} \widehat{H}_i \otimes \mathbf{B}(\frac{i}{n}), \ \widehat{H}_i=\sum_{j=i}^{i+m} \widehat{\bm{h}}_j.
\end{equation}
 By a discussion similar to Lemma \ref{lem_a3}, we conclude that
\begin{equation*}
\sup_{b_*+1 \leq i \leq n-m} \left| \left|  \Upsilon_{i,m} \Upsilon_{i,m}^*-\widehat{\Upsilon}_{i,m} \widehat{\Upsilon}_{i,m}^*  \right| \right|=O_{\mathbb{P}}(b_* \zeta_c^2 \vartheta_n).
\end{equation*}
Hence, we have 
\begin{equation*}
||\Lambda-\widehat{\Lambda} ||=O_{\mathbb{P}}\Big(\frac{1}{\sqrt{n}} b_* \zeta_c^2 \vartheta_n \Big).
\end{equation*}
Using a discussion similar to (\ref{eq_gapproximation}), we can conclude our proof.   

%
%

\end{proof}

%
\subsection{Proofs of the main results of Section \ref{sec_application}} { In this subsection, we prove the results related to optimal forecasting as in Section \ref{sec_application}. Since the forecasting part is an application of our AR approximation theory established in Section \ref{sec:arappoximate}, the proof of Theorem \ref{thm_prediction} is relatively straightforward using the established results. Moreover, Theorem \ref{thm_finalresult} is a slight generalization of  \cite[Theorem 3.7 and Corollary 3.8]{DZ1} and we will focus on explaining the differences. }

\begin{proof}[\bf Proof of Theorem \ref{thm_prediction}] 
Note that by adding and subtracting $\widehat{x}_{n+1},$ we have 
\begin{align*}
\mathbb{E}(x_{n+1}-\widehat{x}_{n+1}^b)^2=\mathbb{E}(x_{n+1}-\widehat{x}_{n+1})^2+\mathbb{E}(\widehat{x}_{n+1}-\widehat{x}_{n+1}^b)^2+2\mathbb{E}(x_{n+1}-\widehat{x}_{n+1})(\widehat{x}_{n+1}-\widehat{x}_{n+1}^b).
\end{align*}
It suffices to control the second and third terms of the above equations. First, 
\begin{equation}\label{eq_expandifference}
\widehat{x}_{n+1}-\widehat{x}_{n+1}^b=\sum_{j=1}^b(\phi_{nj}-\phi_j(1))x_{n+1-j}+\sum_{j=b+1}^n
\phi_{nj} x_{n+1-j}.
\end{equation}
Therefore, by Theorem \ref{thm_locallynonzero}, (\ref{assum_moment}) and (\ref{eq_phibound1}), we find that there exists some constant $C>0$ such that 
{
\begin{equation}\label{eq_aaa}
\mathbb{E}(\widehat{x}_{n+1}-\widehat{x}_{n+1}^b)^2 \leq C \left((\log b)^{\tau}  b^{-(\tau-2)}+\frac{b^{2.5}}{n} \right)^2.
\end{equation}}
Second,  since $\widehat{x}_{n+1}$ is the best linear forecasting based on $\{x_1, \cdots, x_n\}$, then $x_{n+1}-\widehat{x}_{n+1}$ is uncorrelated with  any linear combination of $\{x_1, \cdots, x_n\}$. Together with (\ref{eq_expandifference}), we readily obtain that 
\begin{equation*}
\mathbb{E}(x_{n+1}-\widehat{x}_{n+1})(\widehat{x}_{n+1}-\widehat{x}_{n+1}^b)=0.
\end{equation*}
This completes our proof. 
\end{proof}

\begin{proof}[\bf Proof of Theorem \ref{thm_finalresult}]
By a discussion similar to \cite[Theorem 3.7 and Corollary 3.8]{DZ1}, we find that
\begin{equation}\label{eq_coefficientbound}
\sup_{i>b, 0 \leq j \leq b} \left| \varphi_j(\frac{i}{n})-\widehat{\varphi}_j(\frac{i}{n}) \right|=O_{\mathbb{P}}\left( b\zeta_c \sqrt{\frac{\log n}{n}}+bc^{-d} \right). 
\end{equation}
In fact, the only difference of the proof is that our design matrix $Y $ is the $(n-b) \times (b+1)c$ rectangular matrix whose $i$-th row is $\bm{x}_{i} \otimes \mathbf{B}(\frac{i}{n}).$  Here $\bm{x}_{i}=(1,x_{i-1}, \cdots, x_{i-b}) \in \mathbb{R}^{b+1},$ $\mathbf{B}(i/n)=(\alpha_1(\frac{i}{n}), \cdots, \alpha_c(\frac{i}{n})) \in \mathbb{R}^c$ and $\otimes$ is the Kronecker product. Then it is easy to see that the proof follows from (\ref{eq_coefficientbound}), (\ref{eq_vvvffff}) and the smoothness of $\varphi(\cdot)$.  Together with (\ref{eq_vvvffff}), we can conclude our proof. 

\end{proof}

\section{Choices of tuning parameters}\label{sec:choiceparameter}

In this section, we discuss how to choose the parameters. As we have seen from (\ref{eq_forecast}) and (\ref{eq_phiform}), we need to choose two important parameters in order to get an accurate prediction: $b$ and $c.$ We use a data-driven procedure proposed in \cite{bishop2013pattern} to choose such parameters. 

For a given integer $l,$ say $l=\lfloor 3 \log_2 n \rfloor,$ we divide the time series into two parts: the training part $\{x_i\}_{i=1}^{n-l}$ and the validation part $\{x_i\}_{i=n-l+1}^n.$  With some preliminary initial pair $(b,c)$, we propose a sequence of candidate pairs  $(b_i, c_j), \ i=1,2,\cdots, u, \ j=1,2,\cdots, v,$ in an appropriate neighbourhood of $(b,c)$ where $u, v$ are some given integers. For each pair of the choices $(b_i, c_j),$  we fit a time-varying AR($b_i$) model (i.e., $b=b_i$ in (\ref{eq_forecast})) with $c_j$ sieve basis expansion using the training data set.  Then using the fitted model, we forecast the time series in the validation part of the time series.  Let $\widehat x_{n-l+1,ij}, \cdots, \widehat x_{n,ij}$ be the forecast of $x_{n-l+1},..., x_n,$ respectively using the parameter pair $(b_i, c_j)$. Then we choose the pair $(b_{i_0},c_{j_0})$ with the minimum sample MSE of forecast, i.e.,
 \begin{equation*} 
({i_0},{j_0}):= \argmin_{((i,j): 1 \leq i \leq u, 1 \leq j \leq v)} \frac{1}{l}\sum_{k=n-l+1}^n (x_k-\widehat x_{k,ij})^2.
 \end{equation*}

{
Then we discuss how to choose $m$ for practical implementation. In \cite{ZZ1}, the author used the minimum volatility (MV) method to choose the window size $m$ for the scalar covariance function. The MV method does not depend on the specific form of the underlying time series dependence structure and hence is robust to misspecification of
the latter structure \cite{politis1999subsampling}. The MV method utilizes the fact that the covariance structure of $\widehat{\Omega}$ becomes stable when the
block size $m$ is in an appropriate range, where $\widehat{\Omega}=E[\Phi\Phi^*|(x_1,\cdots,x_n)]$ is defined as 
{ 
\begin{equation}\label{eq_widehatomega}
\widehat{\Omega}:=\frac{1}{(n-m-b+1)m} \sum_{i=b+1}^{n-m} \Big[ \Big(\sum_{j=i}^{i+m} \bm{h}_i \Big) \otimes \Big( \mathbf{B}(\frac{i}{n}) \Big) \Big] \times \Big[ \Big(\sum_{j=i}^{i+m} \bm{h}_i \Big) \otimes \Big( \mathbf{B}(\frac{i}{n}) \Big) \Big]^*.
\end{equation}
}    Therefore, it desires to minimize the standard errors of the latter covariance structure in a suitable range of candidate $m$'s.

In detail, for a give large value $m_{n_0}$ and a neighborhood control parameter $h_0>0,$  we can choose a sequence of window sizes $m_{-h_0+1}<\cdots<m_1< m_2<\cdots<m_{n_0}<\cdots<m_{n_0+h_0}$  and obtain $\widehat{\Omega}_{m_j}$ by replacing $m$ with $m_j$ in (\ref{eq_defnphi}), $j=-h_0+1,2, \cdots, n_0+h_0.$ For each $m_j, j=1,2,\cdots, m_{n_0},$ we calculate the matrix norm error of $\widehat{\Omega}_{m_j}$ in the $h_0$-neighborhood, i.e., 
\begin{equation*}
\mathsf{se}(m_j):=\mathsf{se}(\{ \widehat{\Omega}_{m_{j+k}}\}_{k=-h_0}^{h_0})=\left[\frac{1}{2h_0} \sum_{k=-h_0}^{h_0} \| \overline{\widehat{\Omega}}_{m_j}-\widehat{\Omega}_{m_j+k} \|^2 \right]^{1/2},
\end{equation*}
where $\overline{\widehat{\Omega}}_{m_j}=\sum_{k=-h_0}^{h_0} \widehat{\Omega}_{m_j+k} /(2h_0+1).$
Therefore, we choose the estimate of $m$ using 
\begin{equation*}
\widehat{m}:=\argmin_{m_1 \leq m \leq m_{n_0}} \mathsf{se}(m).
\end{equation*}
Note that in \cite{ZZ1} the author used $h_0=3$ and we also adopt this choice in the current paper. 
}

{ \section{Additional remarks}\label{sec_suppl_remarks}
In this section, we provide a few more remarks. First, we explain a little bit more on the connection of Rosenblatt transform and (\ref{eq_xi}). We start with bivariate dependence. For a pair of jointly distributed random variables $(X,Y),$ let $F_{Y|X}$ be the conditional distribution function of $Y$ given $X.$ For $u \in (0,1),$ denote the conditional quantile function as
\begin{equation*}
G(x,u)=\inf\{y \in \mathbb{R}: F_{Y| X}(y| x) \geq u\}.
\end{equation*}  
Under some suitable conditions on $F_{Y|X},$ for some random variable $U \sim \text{Uniform}[0,1]$ independent of $X,$ people can conclude that $(X,Y)$ has the same distribution as $(X, G(X,U)).$ Consequently, we can write $Y=G(X,U).$ The above idea can be generalized to study multivariate dependence of the random vector $(X_1, \cdots, X_n). $ Denote $\mathbf{X}_m=(X_1, \cdots, X_m)$ for $m \leq n.$ For some measurable function $G_n$ and $U[0,1]$ random variable $U_n$ independent of $\mathbf{X}_n,$  in terms of distribution, we can write $X_n=G_n(\mathbf{X}_{n-1}, U_n).$ Iterating this process, as in equation (12) of \cite{wu2010new}, we have that for a sequence of i.i.d. $U[0,1]$ random variables $U_i, 1 \leq i \leq n,$ which are independent of $\mathbf{X}_n,$ and some measurable functions $H_1, \cdots, H_n,$ we have that 
\begin{equation} \label{eq_distributionequivalence}
\begin{pmatrix}
X_1 \\
X_2  \\
\vdots \\
X_n
\end{pmatrix} \cong
\begin{pmatrix}
H_1(\mathbf{U}_1) \\
H_2(\mathbf{U}_2) \\
\vdots \\
H_n(\mathbf{U}_n)
\end{pmatrix},
\end{equation}
where $\cong$ means equal in distribution and $\mathbf{U}_m=(U_1, \cdots, U_m), m \leq n.$  

Based on the above summary, we can see that for any time series $\{x_i\}, x_i \equiv x_{i,n},$ regardless of the stationarity, we can always rewrite it using some physical representation using Rosenblatt transform. This is the main advantage of this transform that guarantees the existence of the physical representation. Even though the physical representation is not unique so that Rosenblatt transform may not offer the most convenient  choice and the physical representation may not be explicit in general even for the linear stationary process, our established theory and methodology only require the existence of the physical representation form. This indicates that our method and theory are quite general and do not need to reply on specific structural assumptions of the underlying time series.

\section{Additional simulation results}
%

\subsection{Finite sample numerical comparison of (\ref{eq_simplified}) and (\ref{eq_defnntori})}\label{sec_differentTcomparison} In this subsection, we conduct some numerical simulations to compare the finite sample performance of  (\ref{eq_simplified}) and (\ref{eq_defnntori}) using the Fourier and Legendre basis functions. For these two specific bases, they satisfy (\ref{eq_basispropertyhaha}) so that (\ref{eq_defnntori}) can be reduced to (\ref{eq_simplified}). 

For comparison, we follow the setup of Section \ref{sec_poer} to compare the performance of these two equivalent expressions for the two basis functions. First, the finite sample accuracy of the two statistics under the null hypothesis (\ref{eq_testingcases}) are recorded in Tables \ref{table_cttfourier} and \ref{table_cttpoly} for Fourier basis and Legendre basis respectively using the models from Section \ref{simu_intro}. We conclude that even though (\ref{eq_simplified}) and (\ref{eq_defnntori}) are equivalent, (\ref{eq_simplified}) seems to be more accurate overall, especially when the sample size $n$ is smaller. Second, we study the power of the tests under the alternative (\ref{eq_testingcasesalternative}). Analogously, we find from Tables \ref{table_cttfourierpower} and \ref{table_cttpolypower} that (\ref{eq_simplified}) seems to have a better finite sample performance overall when the sample size $n$ is smaller. 

\begin{table}[ht]
\begin{center}
\setlength\arrayrulewidth{1pt}
\renewcommand{\arraystretch}{1.3}
{\fontsize{9}{9}\selectfont 
\begin{tabular}{|c|ccccc|ccccc|}
\hline
      & \multicolumn{5}{c|}{$\alpha=0.1$}                                                                                                                       & \multicolumn{5}{c|}{$\alpha=0.05$}                                                                                                                        \\ \hline
Statistics/Model & \multicolumn{1}{c|}{1} & \multicolumn{1}{c|}{2} & \multicolumn{1}{c|}{3} & \multicolumn{1}{c|}{4} & \multicolumn{1}{c|}{5}  & \multicolumn{1}{c|}{1} & \multicolumn{1}{c|}{2} & \multicolumn{1}{c|}{3} & \multicolumn{1}{c|}{4} & \multicolumn{1}{c|}{5} \\ 
\hline
     & \multicolumn{10}{c|}{$n$=256}                                                                                                                                                                                                                                                                                          \\
   \hline
(\ref{eq_defnntori})    &          0.132  & 0.11                          &                          0.12 &         0.13                  &                          0.11 &                    0.067   & 0.07    & 0.06   &                                     0.04 &       0.06                    \\
(\ref{eq_simplified})    &     0.128       & 0.089                          &                          0.114 &   0.125                        &           0.09                &              0.06    & 0.059         &  0.041  &                                     0.061 &                            0.041 \\
\hline
      & \multicolumn{10}{c|}{$n$=512}                                                                                                                                                                                                                                                                                         \\
       \hline
(\ref{eq_defnntori})     &            0.09      & 0.13                 &                          0.11 &      0.13                     &  0.127  &                         0.05 & 0.06 & 0.067   &    0.068                                  &              0.069                 \\
(\ref{eq_simplified})     &    0.09    & 0.091    &                  0.113                                &          0.126                 &                       0.108   & 0.046   & 0.049                       & 0.065    &        0.058                              &                            0.043 \\
 \hline
\end{tabular}
}
\end{center}
\caption{Comparison of simulated type I errors using the setup (\ref{eq_testingcases}) with Fourier bases. The results are reported based on 1,000 simulations.  
}
\label{table_cttfourier}
\end{table}
  
\begin{table}[ht]
\begin{center}
\setlength\arrayrulewidth{1pt}
\renewcommand{\arraystretch}{1.3}
{\fontsize{9}{9}\selectfont 
\begin{tabular}{|c|ccccc|ccccc|}
\hline
      & \multicolumn{5}{c|}{$\alpha=0.1$}                                                                                                                       & \multicolumn{5}{c|}{$\alpha=0.05$}                                                                                                                        \\ \hline
Statistics/Model & \multicolumn{1}{c|}{1} & \multicolumn{1}{c|}{2} & \multicolumn{1}{c|}{3} & \multicolumn{1}{c|}{4} & \multicolumn{1}{c|}{5}  & \multicolumn{1}{c|}{1} & \multicolumn{1}{c|}{2} & \multicolumn{1}{c|}{3} & \multicolumn{1}{c|}{4} & \multicolumn{1}{c|}{5} \\ 
\hline
     & \multicolumn{10}{c|}{$n$=256}                                                                                                                                                                                                                                                                                          \\
   \hline
(\ref{eq_defnntori})   &     0.091       & 0.136                          &                          0.13 &   0.12                        &           0.13                &              0.06    & 0.059         &  0.041  &                                     0.07 &                            0.07   \\
(\ref{eq_simplified})    &     0.093       & 0.118                          &                          0.128 &   0.087                        &           0.118                &              0.061    & 0.041         &  0.061  &                                     0.064 &                            0.059 \\
\hline
      & \multicolumn{10}{c|}{$n$=512}                                                                                                                                                                                                                                                                                         \\
       \hline
(\ref{eq_defnntori})    &    0.09    & 0.094    &                  0.092                                &          0.12                 &                       0.118   & 0.04   & 0.058                       & 0.07    &        0.043                              &                            0.057             \\
(\ref{eq_simplified})     &    0.091    & 0.093    &                  0.108                                &          0.11                 &                       0.114   & 0.058   & 0.042                       & 0.064    &        0.053                              &                            0.054 \\
 \hline
\end{tabular}
}
\end{center}
\caption{Comparison of simulated type I errors using the setup (\ref{eq_testingcases}) with Legendre orthogonal polynomials. The results are reported based on 1,000 simulations.  
}
\label{table_cttpoly}
\end{table}

\begin{table}[ht]
\begin{center}
\setlength\arrayrulewidth{1pt}
\renewcommand{\arraystretch}{1.3}
{\fontsize{9}{9}\selectfont 
\begin{tabular}{|c|ccccc|ccccc|}
\hline
      & \multicolumn{5}{c|}{$\delta=0.2/0.5$}                                                                                                                       & \multicolumn{5}{c|}{$\delta=0.35/0.7$}                                                                                                                        \\ \hline
Statistics/Model & \multicolumn{1}{c|}{1} & \multicolumn{1}{c|}{2} & \multicolumn{1}{c|}{3} & \multicolumn{1}{c|}{4} & \multicolumn{1}{c|}{5}  & \multicolumn{1}{c|}{1} & \multicolumn{1}{c|}{2} & \multicolumn{1}{c|}{3} & \multicolumn{1}{c|}{4} & \multicolumn{1}{c|}{5} \\ 
\hline
     & \multicolumn{10}{c|}{$n$=256}                                                                                                                                                                                                                                                                                          \\
   \hline
(\ref{eq_defnntori})    &          0.84  & 0.86                          &                          0.84 &         0.837                  &                          0.94 &                    0.97   & 0.97    & 0.96   &                                     0.99 &       0.98                    \\
(\ref{eq_simplified})    &     0.843       & 0.85                          &                          0.859 &   0.877                        &           0.943                &              0.969    & 0.983         &  0.971  &                                     0.986 &                            0.983 \\
\hline
      & \multicolumn{10}{c|}{$n$=512}                                                                                                                                                                                                                                                                                         \\
       \hline
(\ref{eq_defnntori})     &            0.91      & 0.9                 &                          0.96 &      0.9                     &  0.93  &                         0.96 & 0.97 & 0.973   &    0.98                                  &              0.97                 \\
(\ref{eq_simplified})     &    0.907    & 0.94    &                  0.95                                &          0.93                 &                       0.946   & 0.94   & 0.978                       & 0.98    &        0.976                              &                            0.97 \\
 \hline
\end{tabular}
}
\end{center}
\caption{Comparison of simulated power using the setup (\ref{eq_testingcasesalternative}) with Fourier bases. The results are reported based on 1,000 simulations.  
}
\label{table_cttfourierpower}
\end{table}
  
\begin{table}[ht]
\begin{center}
\setlength\arrayrulewidth{1pt}
\renewcommand{\arraystretch}{1.3}
{\fontsize{9}{9}\selectfont 
\begin{tabular}{|c|ccccc|ccccc|}
\hline
      & \multicolumn{5}{c|}{$\alpha=0.1$}                                                                                                                       & \multicolumn{5}{c|}{$\alpha=0.05$}                                                                                                                        \\ \hline
Statistics/Model & \multicolumn{1}{c|}{1} & \multicolumn{1}{c|}{2} & \multicolumn{1}{c|}{3} & \multicolumn{1}{c|}{4} & \multicolumn{1}{c|}{5}  & \multicolumn{1}{c|}{1} & \multicolumn{1}{c|}{2} & \multicolumn{1}{c|}{3} & \multicolumn{1}{c|}{4} & \multicolumn{1}{c|}{5} \\ 
\hline
     & \multicolumn{10}{c|}{$n$=256}                                                                                                                                                                                                                                                                                          \\
   \hline
(\ref{eq_defnntori})   &     0.8       & 0.806                          &                          0.81 &   0.84                        &           0.83                &              0.97    & 0.968         &  0.95  &                                     0.97 &                            0.91   \\
(\ref{eq_simplified})    &     0.834       & 0.846                          &                          0.84 &   0.836                        &           0.87                &              0.97    & 0.99         &  0.94  &                                     0.97 &                            0.97 \\
\hline
      & \multicolumn{10}{c|}{$n$=512}                                                                                                                                                                                                                                                                                         \\
       \hline
(\ref{eq_defnntori})    &    0.9    & 0.91    &                  0.92                                &          0.893                 &                       0.91   & 0.94   & 0.95                       & 0.98    &        0.97                              &                            0.96             \\
(\ref{eq_simplified})     &    0.931    & 0.94    &                  0.92                                &          0.92                 &                       0.918   & 0.94   & 0.98                       & 0.974    &        0.97                              &                            0.977 \\
 \hline
\end{tabular}
}
\end{center}
\caption{Comparison of simulated power using the setup (\ref{eq_testingcasesalternative}) with Legendre orthogonal polynomials. The results are reported based on 1,000 simulations.  
}
\label{table_cttpolypower}
\end{table}

\subsection{Numerical results using plug-in estimators and comparison discussions with the proposed approach}\label{sec_numericalplguin} In this subsection, we examine the performance of applying Proposition \ref{prop_normal} with plug-in estimators. Moreover, we compare these results with our proposed method, i.e., Algorithm \ref{alg:boostrapping}, and justify the arguments in Remark \ref{rem_collectionofremark} using some numerical simulations. 

We consider the same simulation settings as in Section \ref{sec_poer} using the plug-in estimators. That is to say, we estimate $\Omega$ using $\widehat{\Omega}$ that
\begin{equation}
\widehat{\Omega}=\frac{1}{(n-m-b_*)} \sum_{i=b_*+1}^{n-m} \widehat{\Upsilon}_{i,m} \widehat{\Upsilon}_{i,m}^*,
\end{equation} 
where $\widehat{\Upsilon}_{i,m}$ is defined in (\ref{eq_hatUp}). Moreover, the estimator of $\Gamma,$ denoted as $\widehat{\Gamma},$ is estimated in the same way as in (\ref{eq_defnmathcalt}). Consequently, we can estimate the unknowns $f_1$ and $f_2$ in Proposition \ref{prop_normal} using 
\begin{equation}\label{eq_festimatorsuppl}
\widehat{f}_k=\left( \operatorname{Tr} \left[ \widehat{\Omega}^{1/2} \widehat{\Gamma} \widehat{\Omega}^{1/2} \right] \right)^{1/k}, k=1,2. 
\end{equation}
Consequently, under $\mathbf{H}_0,$ we have the following result asymptotically
\begin{equation}\label{eq_asymptoticpluginestimator}
\frac{nT-\widehat{f}_1}{\widehat{f}_2} \Rightarrow \mathcal{N}(0,2). 
\end{equation}
Then we can use (\ref{eq_asymptoticpluginestimator}) to test the null hypothesis. The results are recorded in Tables \ref{table_typeiplug} and \ref{table_powerplug} receptively for the type I error and power. We compare these two tables with Tables \ref{table_typei} and \ref{table_power} where our Algorithm \ref{alg:boostrapping} is implemented. We find that overall,  our proposed method has better finite sample performance, especially when the sample size $n$ is smaller.

\begin{table}[ht]
\begin{center}
\setlength\arrayrulewidth{1pt}
\renewcommand{\arraystretch}{1.3}
{\fontsize{9}{9}\selectfont 
\begin{tabular}{|c|ccccc|ccccc|}
\hline
      & \multicolumn{5}{c|}{$\alpha=0.1$}                                                                                                                       & \multicolumn{5}{c|}{$\alpha=0.05$}                                                                                                                        \\ \hline
Basis/Model & \multicolumn{1}{c|}{1} & \multicolumn{1}{c|}{2} & \multicolumn{1}{c|}{3} & \multicolumn{1}{c|}{4} & \multicolumn{1}{c|}{5}  & \multicolumn{1}{c|}{1} & \multicolumn{1}{c|}{2} & \multicolumn{1}{c|}{3} & \multicolumn{1}{c|}{4} & \multicolumn{1}{c|}{5} \\ 
\hline
     & \multicolumn{10}{c|}{$n$=256}                                                                                                                                                                                                                                                                                          \\
   \hline
Fourier     &          0.148  & 0.123                          &                          0.132 &         0.129                  &                          0.125 &                    0.08   & 0.07    & 0.075   &                                     0.074 &       0.081                    \\
Legendre    &     0.13       & 0.134                          &                          0.129 &   0.13                        &           0.141                &              0.083    & 0.079         &  0.081  &                                     0.082 &                            0.08 \\
Daubechies-9    &  0.13 & 0.129   & 0.12                        &      0.14             &         0.132                 & 0.083                 &0.076 &  0.085  &          0.078                            &                 0.084   \\
\hline
      & \multicolumn{10}{c|}{$n$=512}                                                                                                                                                                                                                                                                                         \\
       \hline
Fourier     &            0.115      & 0.128                 &                          0.12 &      0.128                     &  0.125  &                         0.065 & 0.068 & 0.071   &    0.063                                  &              0.068                 \\
Legendre     &    0.088    & 0.095    &                  0.11                                &          0.118                 &                       0.12   & 0.061   & 0.042                       & 0.071    &        0.065                              &                            0.061 \\
Daubechies-9     &  0.087  & 0.085    &      0.11                   &                         0.11   & 0.084                          &          0.04               &  0.047 & 0.062 &                                 0.063  &                  0.052           \\
 \hline
\end{tabular}
}
\end{center}
\caption{Simulated type I errors using the plug-in estimators, i.e., (\ref{eq_asymptoticpluginestimator}). The simulation settings are the same as in the caption of Table \ref{table_typei}. }
\label{table_typeiplug}
\end{table}

\begin{table}[ht]
\begin{center}
\setlength\arrayrulewidth{1pt}
\renewcommand{\arraystretch}{1.3}
{\fontsize{9}{9}\selectfont 
\begin{tabular}{|c|ccccc|ccccc|}
\hline
      & \multicolumn{5}{c|}{$\delta=0.2/0.5$}                                                                                                                       & \multicolumn{5}{c|}{$\delta=0.35/0.7$}                                                                                                                        \\ \hline
Basis/Model & \multicolumn{1}{c|}{1} & \multicolumn{1}{c|}{2} &  \multicolumn{1}{c|}{3} & \multicolumn{1}{c|}{4} & \multicolumn{1}{c|}{5}  & \multicolumn{1}{c|}{1} & \multicolumn{1}{c|}{2} &  \multicolumn{1}{c|}{3} & \multicolumn{1}{c|}{4} & \multicolumn{1}{c|}{5} \\ 
\hline
      & \multicolumn{10}{c|}{$n$=256}                                                                                                                                                                                                                                                                                          \\
   \hline
Fourier     & 0.83     & 0.88                              &                          0.83 &                  0.82          &                      0.91 &                        0.965 &  0.97 & 0.97  & 0.98                                      &              0.985              \\
Legendre     &         0.81   & 0.78                          &                          0.79 &               0.81           & 0.83                          & 0.96    & 0.96                       &  0.956  &                                     0.964 &        0.9                     \\
Daubechies-9    &     0.79 & 0.8 & 0.88                        &                       0.8   & 0.82                          &           0.96       & 0.96 & 0.97  & 0.985                                    &        0.977           \\
\hline
      & \multicolumn{10}{c|}{$n$=512}                                                                                                                                                                                                                                                                                         \\
       \hline
Fourier     &  0.92         & 0.896                       &                           0.95&          0.91              &  0.94                      &           0.95               &  0.97  &  0.96 &    0.97                                &                  0.97             \\
Legendre    &     0.92         & 0.9                       &                          0.91&                        0.91   & 0.9                           &           0.938    & 0.95          &0.97  &        0.97                             &        0.95   \\
Daubechies-9     &             0.89  & 0.85                       &                      0.91     &  0.9                          &                       0.92    &  0.95                         & 0.97 & 0.98  &        0.98                              &  0.96                             \\
 \hline
\end{tabular}
}
\end{center}
\caption{ Simulated power using the plug-in estimators, i.e., (\ref{eq_asymptoticpluginestimator}). The simulation settings are the same as in the caption of Table \ref{table_power}. 
}
\label{table_powerplug}
\end{table}

\subsection{Additional simulation results for Section \ref{sec_sub_suppledalhaus}}\label{sec_additionalofadditionalbasis} In this subsection, we enclose more simulation results associated with Section \ref{sec_sub_suppledalhaus} using the Legendre orthogonal polynomials and orthogonal wavelet basis functions.

\begin{figure}[!ht]
\hspace*{-2.0cm}
\begin{subfigure}{0.55\textwidth}
\includegraphics[width=7.8cm,height=4.8cm]{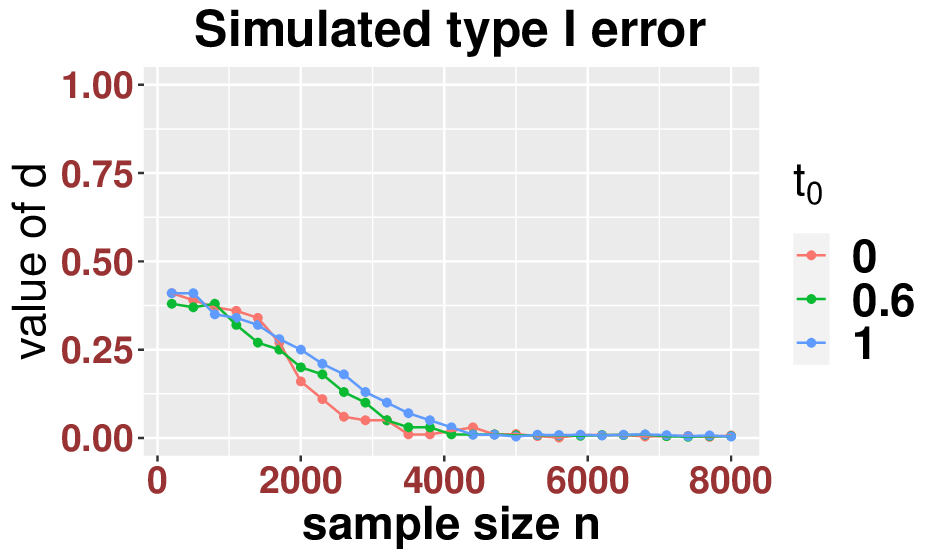}
\end{subfigure}
\begin{subfigure}{0.55\textwidth}
\includegraphics[width=7.8cm,height=4.8cm]{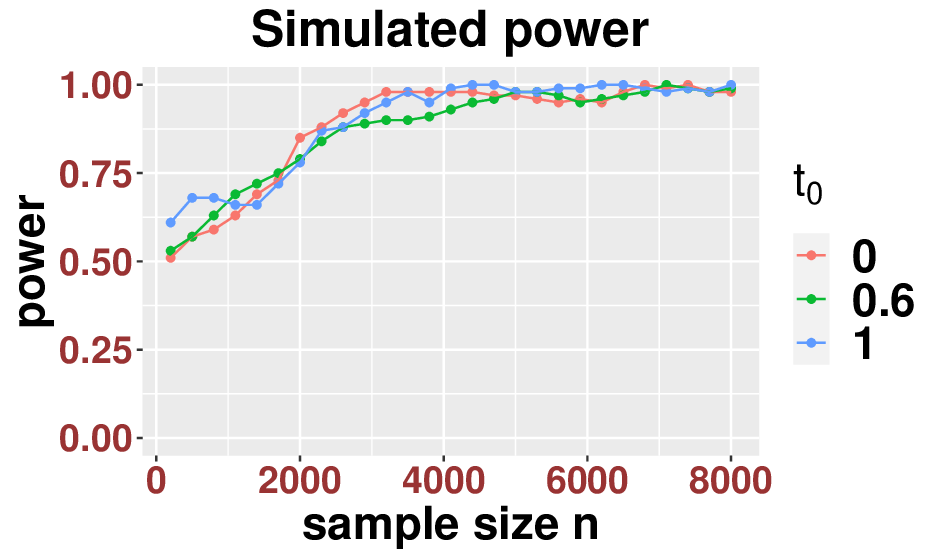}
\end{subfigure}
\caption{ {\footnotesize Simulated type I errors and power.  Here we used the Legendre orthogonal polynomials. The setup is the same as in the caption of Figure \ref{fig_additional1}.}  }
\end{figure}


\begin{figure}[!ht]
\hspace*{-2.0cm}
\begin{subfigure}{0.32\textwidth}
\includegraphics[width=6cm,height=5cm]{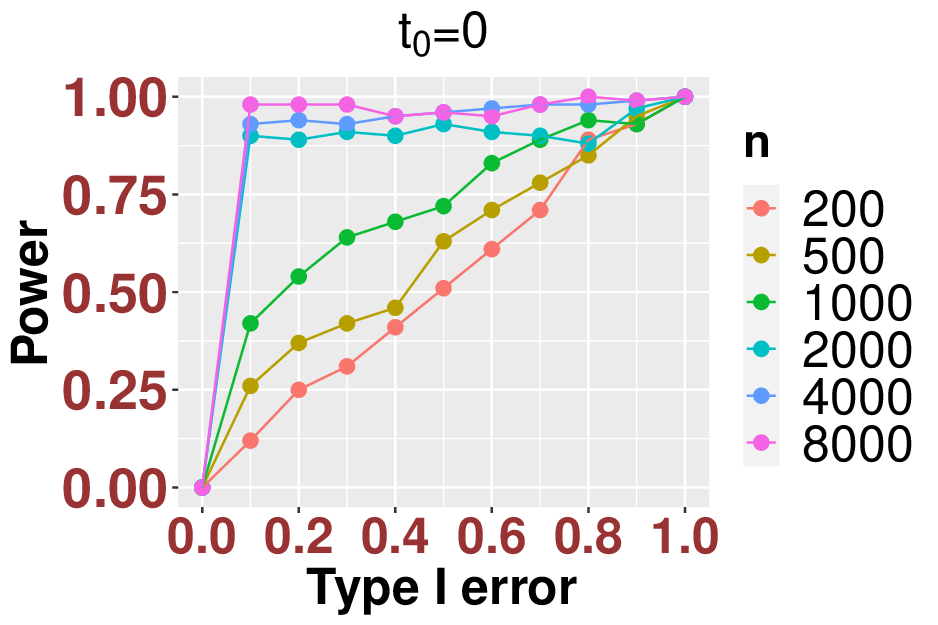}
\end{subfigure}
\begin{subfigure}{0.32\textwidth}
\includegraphics[width=6cm,height=5cm]{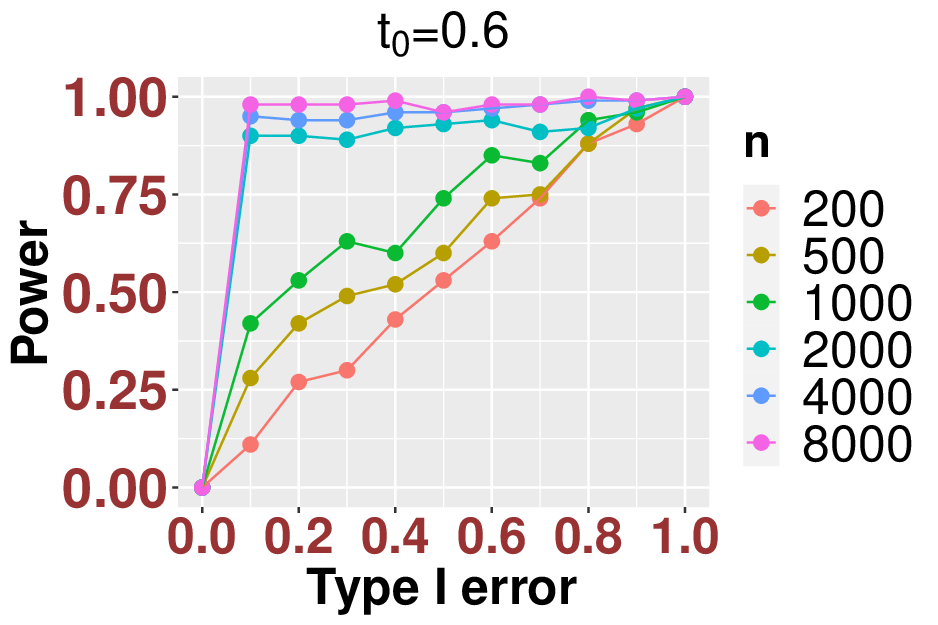}
\end{subfigure}
\begin{subfigure}{0.35\textwidth}
\includegraphics[width=6cm,height=5cm]{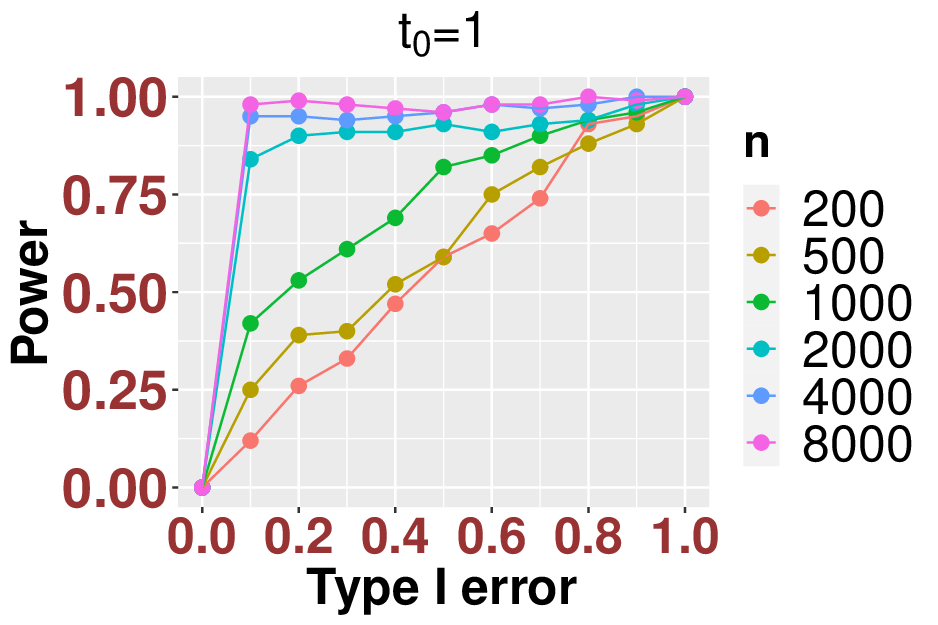}
\end{subfigure}
\caption{ {\footnotesize ROC curves for different values of $t_0.$ Here we used the Legendre orthogonal polynomials. The setup is the same as in the caption of Figure \ref{fig_additional1}. } }
\end{figure}


\begin{figure}[!ht]
\hspace*{-2.0cm}
\begin{subfigure}{0.55\textwidth}
\includegraphics[width=7.8cm,height=4.8cm]{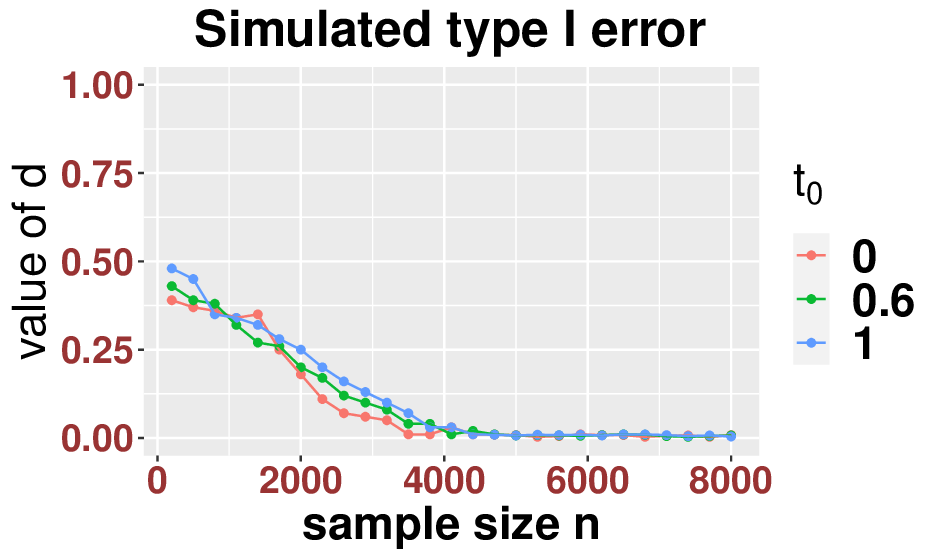}
\end{subfigure}
\begin{subfigure}{0.55\textwidth}
\includegraphics[width=7.8cm,height=4.8cm]{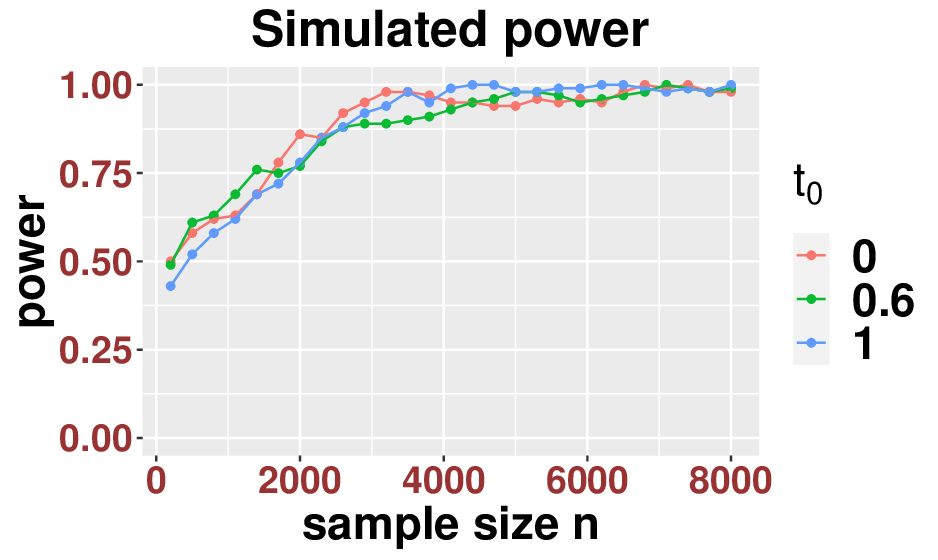}
\end{subfigure}
\caption{ {\footnotesize Simulated type I errors and power.  Here we used the Daubechies-9 orthogonal wavelets. The setup is the same as in the caption of Figure \ref{fig_additional1}.}  }
\end{figure}


\begin{figure}[!ht]
\hspace*{-2.0cm}
\begin{subfigure}{0.32\textwidth}
\includegraphics[width=6cm,height=5cm]{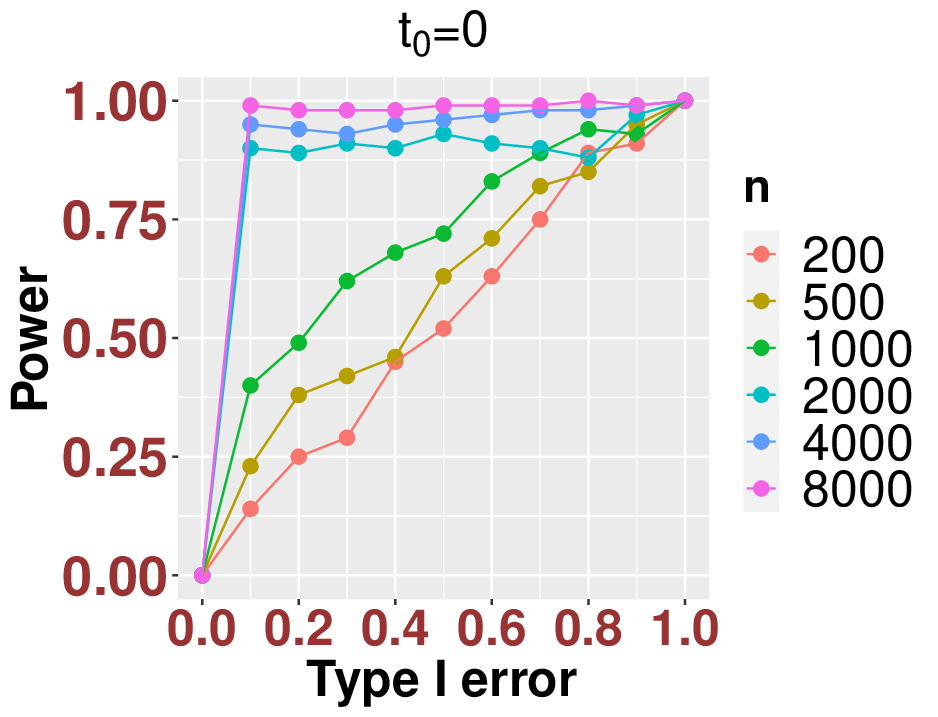}
\end{subfigure}
\begin{subfigure}{0.32\textwidth}
\includegraphics[width=6cm,height=5cm]{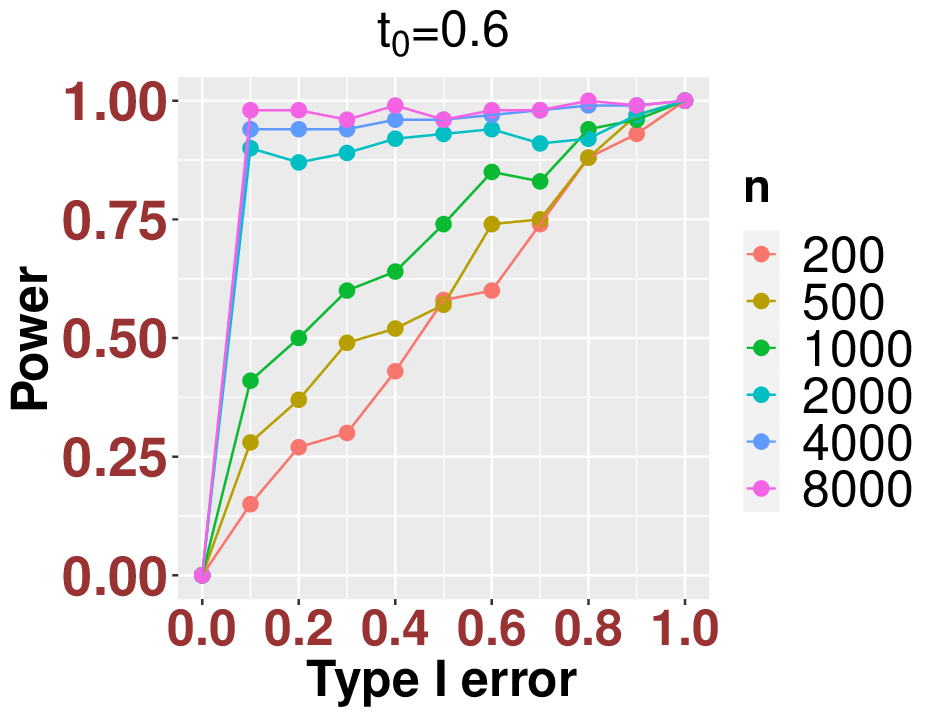}
\end{subfigure}
\begin{subfigure}{0.35\textwidth}
\includegraphics[width=6cm,height=5cm]{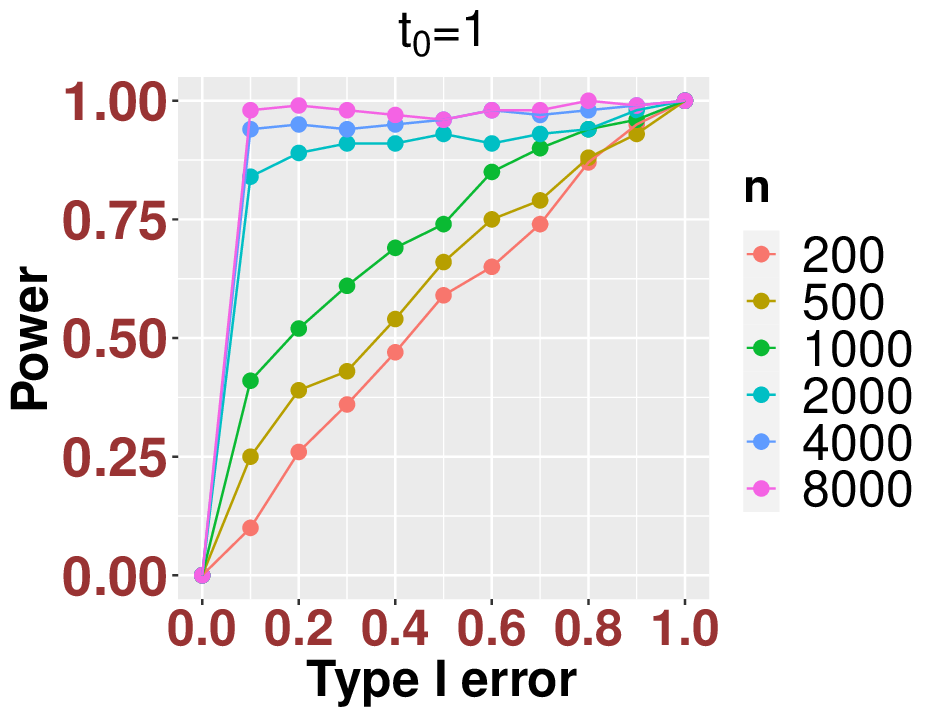}
\end{subfigure}
\caption{ {\footnotesize ROC curves for different values of $t_0.$ Here we used the Daubechies-9 orthogonal wavelets. The setup is the same as in the caption of Figure \ref{fig_additional1}. } }
\end{figure}


\subsection{Additional figures}\label{sec_figurestockappend}
In this subsection, we provide the plot of the time series for the stock return data studied in Section  \ref{sec:realdata}. The plot is provided in Figure \ref{0814_timeseries}.

\begin{figure}[ht]
\centering
\includegraphics[width=11cm,height=6cm]{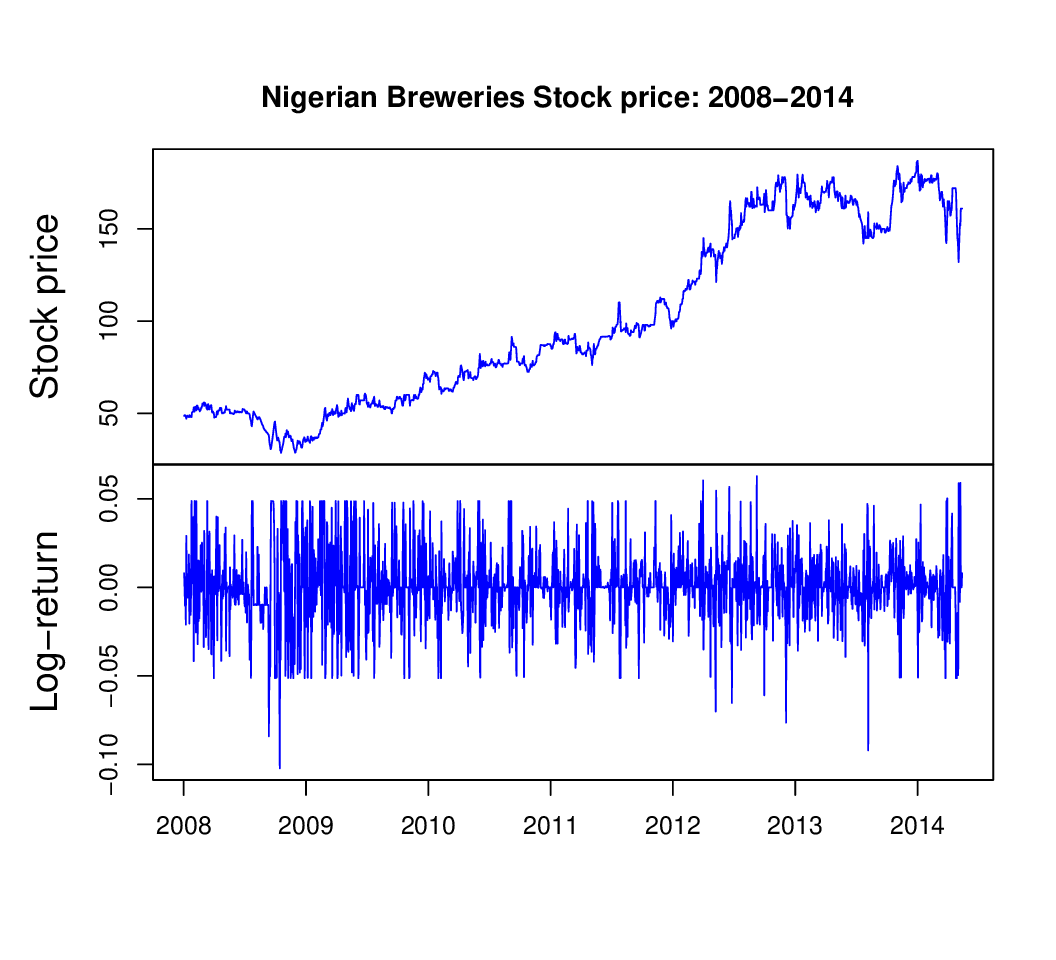}
\caption{Nigerian Breweries stock return from 2008 to 2014. The upper panel is the original stock price and the lower panel is the log-return.}
\label{0814_timeseries}
\end{figure}

}

\section{Some auxiliary lemmas}

In this section, we collect some preliminary lemmas which will be used for our technical proofs.  First of all, we collect a result which provides a deterministic bound for the spectrum of a  square matrix.  Let  $A=(a_{ij})$ be a complex $ n\times n$ matrix. For  $1 \leq i \leq n,$ let  $R_{i}=\sum _{{j\neq {i}}}\left|a_{{ij}}\right| $ be the sum of the absolute values of the non-diagonal entries in the  $i$-th row. Let  $ D(a_{ii},R_{i})\subseteq \mathbb {C} $ be a closed disc centered at $a_{ii}$ with radius  $R_{i}$. Such a disc is called a \emph{Gershgorin disc.}
\begin{lem}[Gershgorin circle theorem]\label{lem_disc} Every eigenvalue of $ A=(a_{ij})$ lies within at least one of the Gershgorin discs  $D(a_{ii},R_{i})$, where $R_i=\sum_{j\ne i}|a_{ij}|$.
\end{lem}

{ The next lemma provides a lower bound for the eigenvalues of a Toeplitz matrix in terms of its associated spectral density function.  Since the autocovariance matrix of any stationary time series is a Toeplitz matrix, we can use the following lemma to bound the smallest eigenvalue of the autocovariance matrix. It will be used in the proof of Proposition \ref{prop_pdc} and can be found in \cite[Lemma 1]{XW}. 
\begin{lem}\label{lem_spectralbound}
Let $h$ be a continuous function on $[-\pi, \pi].$ Denote by $\underline{h}$ and $\overline{h}$ its minimum and maximum, respectively. Define $a_k=\int_{-\pi}^{\pi} h(\theta) e^{- \mathrm{i} k \theta} d \theta$ and the $T \times T$ matrix $\Gamma_T=(a_{s-t})_{1 \leq s, t \leq T}.$ Then 
\begin{equation*}
 2 \pi \underline{h} \leq \lambda_{\min}(\Gamma_T) \leq \lambda_{\max}(\Gamma_T) \leq 2 \pi \overline{h}. 
\end{equation*}
\end{lem}


%

}

The following lemma indicates that, under suitable condition, the inverse of a banded matrix can also be approximated by another banded-like matrix. It will be used in the proof of Theorem \ref{lem_phibound} and can be found in \cite[Proposition 2.2]{DMS}.  We say that $A$ is $m$-banded if $$A_{ij}=0, \ \text{if} \ |i-j|>m/2.$$
\begin{lem}\label{lem_band} Let $A$ be a positive definite, $m$-banded, bounded and bounded invertible matrix.  Let $[a,b]$ be the smallest interval containing the spectrum of $A.$ Set $r=b/a, q=(\sqrt{r}-1)/(\sqrt{r}+1)$ and set $C_0=(1+r^{1/2})^2/(2ar)$ and $\lambda=q^{2/m}.$ Then we have 
\begin{equation*}
|(A^{-1})_{ij}| \leq C \lambda^{|i-j|},
\end{equation*}
where 
\begin{equation*}
C:=C(a,r)=\max\{a^{-1}, C_0\}. 
\end{equation*}

\end{lem}

The following lemma provides an upper bound for the error of solutions of perturbed linear system. It can be found in the standard numerical analysis literature, for instance see \cite{num_paper}. It will be used in the proof of Theorem \ref{lem_phibound}. Recall that the conditional number of a diagonalizable matrix $A$ is defined as 
\begin{equation*}
\kappa(A)=\frac{\lambda_{\max}(A)}{\lambda_{\min}(A)}.
\end{equation*} 
\begin{lem}\label{lem_nuem} Consider a matrix $A$ and vectors $x,v$ which satisfy the linear system
\begin{equation*}
Ax=v.
\end{equation*}
Suppose that we add perturbations on both $A$ and $v$ such that  
\begin{equation*}
(A+\Delta A)(x+\Delta x)=v+\Delta v.
\end{equation*}
Assuming that  there exists some constant $C>0,$ such that  
\begin{equation*}
\frac{\kappa(A)}{1-\kappa(A) \frac{|| \Delta A||}{A}} \leq C,
\end{equation*}
holds. Then we have that
\begin{equation*}
\frac{| \Delta x |}{| x |} \leq C \left( \frac{|| \Delta A||}{|| A||}+ \frac{|| \Delta v ||}{|| v||} \right).
\end{equation*}
\end{lem}

The following lemma provides Gaussian approximation result on convex sets for the sum of an $m$-dependent sequence, which is \cite[Theorem 2.1]{FX}. It will be used in the proof of Theorem \ref{thm_gaussian}. 
\begin{lem}\label{lem_xf} Let $W=\sum_{i=1}^n X_i$ be a sum of $\mathsf{d}$-dimensional random vectors such that $\mathbb{E}(X_i)=0$ and $\operatorname{Cov} (W)=\Sigma.$ Suppose $W$ can be decomposed as follows: 
\begin{enumerate}
\item $\forall i \in [n], \ \exists i \in N_i \subset [n],$ such that $W-X_{N_i}$ is independent of $X_i$, where $[n]=\{1,\cdots,n\}.$
\item $\forall i \in [n], j \in N_i, \ \exists N_i \subset N_{ij} \subset [n],$ such that $W-X_{N_{ij}}$ is independent of $\{X_i, X_j\}.$
\item  $\forall i \in [n], j \in N_i, \ k \in N_{ij}, \ \exists N_{ij} \subset N_{ijk} \subset [n]$ such that $W-X_{N_{ijk}}$ is independent of $\{X_i, X_j, X_k\}.$
\end{enumerate}
Suppose further that for each $i \in [n], j \in N_i, k \in N_{ij},$
\begin{equation*}
|X_i| \leq \beta, |N_i| \leq n_1, |N_{ij}| \leq n_2, |N_{ijk}| \leq n_3,
\end{equation*}
where $|\cdot|$ is the Euclidean norm of a vector. Then there exists a universal constant $C$ such that 
\begin{equation*}
\sup_{A \in \mathcal{A}}\left| \mathbb{P}(W \in A)-\mathbb{P}(\Sigma^{1/2}Z \in A) \right| \leq C \mathsf{d}^{1/4} n || \Sigma^{-1/2}||^3 \beta^3 n_1(n_2+\frac{n_3}{\mathsf{d}}),
\end{equation*}
where $Z$ is a $\mathsf{d}$-dimensional Gaussian random vector preserving the covariance structure of $W$ and where $\mathcal{A}$ denotes the collection of all the convex sets in $\mathbb{R}^{\mathsf{d}}.$ . 
\end{lem}

{
The following lemma offers a control for the summation of Chi-square random variables, which will be employed in the proof of Theorem \ref{thm_gaussian}. It can be found in \cite[Lemma S.2]{10.1093/biomet/asz020}.
\begin{lem} \label{lem_chisquare} Let $a_1 \geq a_2 \geq \cdots \geq a_p \geq 0$ such that $\sum_{i=1}^p a^2_i=1;$ let $\eta_i$ be i.i.d. $\chi_1^2$ random variables. Then for all $h>0,$ we have 
\begin{equation*}
\sup_t \mathbb{P}(t \leq \sum_{k=1}^p a_k \eta_k \leq t+h) \leq \sqrt{h} \sqrt{4/\pi}.
\end{equation*} 
\end{lem}

}

{
Next, we collect some preliminary results. The first part of  the following lemma shows that the covariance function (\ref{eq_defncov}) decays polynomially fast under suitable assumptions. It can be found in  \cite[Lemma 2.6]{DZ1}.  The second part shows that the sample covariance matrix and its inverse will converge to some deterministic limits. Its proof is similar to equation (B.6) in the supplementary file of \cite{DZ1} and we omit the details here. 

\begin{lem}\label{lem_coll} (1). Suppose (\ref{eq_physcialbounbounbound}) and Assumptions  \ref{assu_pdc}, \ref{assu_smoothtrend} and \ref{assum_local}  hold true. Then there exists some constant $C>0,$ such that 
\begin{equation*}
\sup_t |\gamma(t,j)| \leq Cj^{-\tau}, \ j \geq 1. 
\end{equation*}
(2). Recall (\ref{eq_overlinesigma}).  Suppose  (\ref{eq_physcialbounbounbound}) and Assumptions \ref{assu_pdc}, \ref{assu_smoothtrend}, \ref{assum_local} and \ref{assu_basis} hold true. Then we have that 
\begin{equation*}
|| \widehat{\Sigma}-\Sigma ||=O_{\mathbb{P}} \Big( \frac{\zeta_c \log n}{\sqrt{n}} \Big),
\end{equation*}   
where $\widehat{\Sigma}=n^{-1}Y^*Y.$

\end{lem}
}

Finally, we collect the concentration inequalities for  non-stationary process using the physical dependence measure.  It is the key ingredient for the proof of most of the theorems and lemmas. 
It can be found in  \cite[Lemma 6]{ZZ1}. 

\begin{lem}\label{lem_con} Let $x_i=G_i(\mathcal{F}_i),$ where $G_i(\cdot)$ is a measurable function and  $\mathcal{F}_i=(\cdots, \eta_{i-1}, \eta_i)$ and $\eta_i, \ i  \in \mathbb{Z}$ are i.i.d  random variables. Suppose that $\mathbb{E}x_i=0$ and $\max_i \mathbb{E}|x_i|^q<\infty$ for some $q>1.$ For some $k>0,$ let $\delta_x(k):=\max_{ 1 \leq i \leq n} \norm{G_i(\mathcal{F}_i)-G_i(\mathcal{F}_{i,i-k})}_q,$ where $\mathcal{F}_{i,i-k}:=(\mathcal{F}_{i-k-1}, \eta_{i-k}',\cdots, \eta_i)$ for an i.i.d copy $\{\eta_i'\}$ of $\{\eta_i\}.$ We further let $\delta_x(k)=0$ if $k<0.$ Write $\gamma_k=\sum_{i=0}^k \delta_x(i).$ Let $S_i=\sum_{j=1}^i x_j.$ \\
(i). For $q'=\min(2,q),$
\begin{equation*}
\norm{S_n}_q^{q'} \leq C_q \sum_{i=-n}^{\infty} (\gamma_{i+n}-\gamma_i)^{q'}.
\end{equation*}
(ii). If $\Delta:=\sum_{j=0}^{\infty} \delta_x(j) <\infty,$ we then have 
\begin{equation*}
\norm{\max_{1 \leq i \leq n}|S_i|}_q \leq C_q n^{1/q'} \Delta. 
\end{equation*}
In (i) and (ii), $C_q$ are generic finite constants which only depend on $q$ and can vary from place to place.  
\end{lem}

\section{Examples of sieve basis functions}\label{sec_basisfunctions}

In this section, we provide a list of some commonly used basis functions.  We also refer to \cite[Section 2.3]{CXH} for a more detailed discussion.  


\noindent (1). Normalized Fourier basis.  For $x \in [0,1],$ consider the following trigonometric polynomials 
\begin{equation*}
\Big\{1, \sqrt{2} \cos(2 k \pi x), \ \sqrt{2} \sin (2 k \pi x), \cdots \Big\}, k \in \mathbb{N}.
\end{equation*}
{We note that the classical trigonometric basis function is well suited for approximating periodic functions on $[0, 1]$. }

\noindent (2). Normalized Legendre polynomials \cite{BWbook}. The Legendre polynomial of degree $n$ can be obtained using Rodrigue's formula
\begin{equation*}
P_n(x)=\frac{1}{2^n n!} \frac{d^n}{dx^n} (x^2-1)^n,\ -1 \leq x \leq 1.
\end{equation*}
In this paper, we use the normalized Legendre polynomial 
\begin{equation*}
P_n^*(x)=
\begin{cases}
1, & n=0; \\
\sqrt{\frac{2n+1}{2}}P_n(2x-1), &, n>0.
\end{cases}
\end{equation*}
The coefficients of  the Legendre  polynomials can be obtained using the R package \texttt{mpoly} {and hence  they are easy to implement in R. 

\noindent(3).  Daubechies orthogonal wavelet \cite{ ID98,ID92}. For $N \in \mathbb{N},$ a Daubechies (mother) wavelet of class $D-N$ is a function $\psi \in L^2(\mathbb{R})$ defined by 
\begin{equation*}
\psi(x):=\sqrt{2} \sum_{k=1}^{2N-1} (-1)^k h_{2N-1-k} \varphi(2x-k),
\end{equation*}
where $h_0,h_1,\cdots,h_{2N-1} \in \mathbb{R}$ are the constant (high pass) filter coefficients satisfying the conditions
$\sum_{k=0}^{N-1} h_{2k}=\frac{1}{\sqrt{2}}=\sum_{k=0}^{N-1} h_{2k+1},$
as well as, for $l=0,1,\cdots,N-1$
\begin{equation*}
\sum_{k=2l}^{2N-1+2l} h_k h_{k-2l}=
\begin{cases}
1, & l =0 ,\\
0, & l \neq 0.
\end{cases}
\end{equation*} 
And $\varphi(x)$ is the scaling (father) wavelet function is supported on $[0,2N-1)$ and satisfies the recursion equation
$\varphi(x)=\sqrt{2} \sum_{k=0}^{2N-1} h_k \varphi(2x-k),$ as well as the normalization $\int_{\mathbb{R}} \varphi(x) dx=1$ and
$ \int_{\mathbb{R}} \varphi(2x-k) \varphi(2x-l)dx=0, \ k \neq l.$
Note that the filter coefficients can be efficiently computed as listed in \cite{ID92}.  The order $N$, on the one hand, decides the support of our wavelet; on the other hand, provides the regularity condition in the sense that
\begin{equation*}
\int_{\mathbb{R}} x^j \psi(x)dx=0, \ j=0,\cdots,N, \  \text{where} \ N \geq d. 
\end{equation*}
We will employ Daubechies wavelet with a sufficiently high order when forecasting in our simulations and data analysis. The basis functions can be either generated using the library \texttt{PyWavelets} in Python \footnote{For visualization for the families of Daubechies wavelet functions, we refer to \url{http://wavelets.pybytes.com}, where the library \texttt{PyWavelets} is also introduced there.} or the \texttt{wavefun} in the \texttt{Wavelet Toolbox} of Matlab. In the present paper, to construct a sequence of orthogonal wavelet, we will follow the dyadic construction of \cite{ID98}. For a given $J_n$ and $J_0,$ we will consider the following periodized wavelets on $[0,1]$
\begin{equation}\label{eq_constructone}
\Big\{ \varphi_{J_0 k}(x), \ 0 \leq k \leq 2^{J_0}-1; 
\psi_{jk}(x), \ J_0 \leq j \leq J_n-1, 0 \leq k \leq 2^{j}-1  \Big\},\ \mbox{ where}
\end{equation} 
\begin{equation*}
\varphi_{J_0 k}(x)=2^{J_0/2} \sum_{l \in \mathbb{Z}} \varphi(2^{J_0}x+2^{J_0}l-k) , \
\psi_{j k}(x)=2^{j/2} \sum_{l \in \mathbb{Z}} \psi(2^{j}x+2^{j}l-k),
\end{equation*}
or, equivalently \cite{MR1085487}
\begin{equation}\label{eq_meyerorthogonal}
\Big\{ \varphi_{J_n k}(x), \  0 \leq k \leq  2^{J_n-1} \Big\}.
\end{equation}
}



\bibliographystyle{imsart-number} 
\bibliography{corrtest}       

%
%
%

\end{document}